\documentclass[11pt,reqno]{amsart}
\usepackage[all,cmtip]{xy}
\usepackage[a4paper]{geometry}
\usepackage{float}
\usepackage[usenames,dvipsnames]{xcolor}
\usepackage{subfigure}

\usepackage[colorlinks, linkcolor=blue,anchorcolor=Periwinkle,
   citecolor=red,urlcolor=Emerald]{hyperref}
\usepackage{amsmath}           
\usepackage{amssymb}           
\usepackage{latexsym}          
\usepackage{mathrsfs}          
\usepackage{bbm}
\usepackage{bookmark}
\usepackage{enumitem}
\usepackage{array}
\usepackage{graphicx}
\usepackage[bottom]{footmisc}
\usepackage{cleveref}
\usepackage{transparent}

\usepackage{tikz}
\usetikzlibrary{matrix}
\usetikzlibrary{knots}
\usepackage{url}

\usetikzlibrary{decorations.markings,cd}
\usetikzlibrary{arrows}
\usetikzlibrary{decorations.pathreplacing,decorations.pathmorphing,shadings,fadings,calc}
\usepackage{extarrows}

\tikzset{->-/.style={decoration={  markings,  mark=at position #1 with
    {\arrow{>}}},postaction={decorate}}}
\tikzset{-<-/.style={decoration={  markings,  mark=at position #1 with
    {\arrow{<}}},postaction={decorate}}}
\tikzcdset{arrow style=tikz, diagrams={>=stealth}}

\newcolumntype{L}{>{$}l<{$}} 

\usepackage{tikz}
\usetikzlibrary{matrix, calc, arrows,cd,
decorations.pathreplacing, decorations.markings, decorations.pathmorphing,
backgrounds,fit,positioning,shapes.symbols,chains,shadings,fadings}

\usepackage{mathabx}
\let\Sun\relax

\let\Taurus\relax
\usepackage{marvosym}
\usepackage{changes}
\definechangesauthor[name={Zy}, color=red]{zhou}
\def\nsf{\on{NSF}}

\theoremstyle{plain}
\newtheorem{theorem}{Theorem}[section]
\newtheorem{lemma}[theorem]{Lemma}
\newtheorem{corollary}[theorem]{Corollary}
\newtheorem{proposition}[theorem]{Proposition}

\theoremstyle{definition}
\newtheorem{definition}[theorem]{Definition}
\newtheorem{example}[theorem]{Example}
\newtheorem{remark}[theorem]{Remark}
\newtheorem{construction}[theorem]{Construction}
\numberwithin{equation}{section}
\numberwithin{figure}{section}
\newtheorem*{vent}{Convention}

\newtheorem*{thmB}{Theorem~B}
\newtheorem*{thmC}{Theorem~C}
\newtheorem*{thmD}{Theorem~D}
\newtheorem*{consA}{Construction~A}
\newtheorem*{CorE}{Corollary~E}

\newcommand{\AKedit}[1]{\textcolor{teal}{#1}}

\newcommand\hua{\mathcal}

\newcommand\ZZ{\mathbb{Z}}

\newcommand\RR{\mathbb{R}}
\newcommand\CC{\mathbb{C}}
\newcommand{\mai}{\mathbf{i}} 
\newcommand\bfc{\mathbf{C}}
\newcommand\<{\langle}
\renewcommand\>{\rangle}

\renewcommand{\setminus}{\smallsetminus}
\renewcommand{\emptyset}{\varnothing}
\newcommand{\isom}{\cong}

\newcommand{\id}{\on{id}}
\newcommand{\Imgy}{\on{Im}} 

\newcommand\Hom{\on{Hom}}
\newcommand\End{\on{End}}

\newcommand{\Int}{\on{Int}}

\newcommand\Br{\on{Br}} 
\newcommand\Crel{\on{Co}} 
\newcommand\CBr{\on{CT}} 
\newcommand\Grot{\on{K}} 

\newcommand{\h}{\on{\hua{H}}} 
\newcommand{\C}{\on{\hua{C}}} 
\newcommand{\D}{\on{\hua{D}}} 
\newcommand{\per}{\on{per}} 

\newcommand\Aut{\on{Aut}} 
\newcommand\Autp{\Aut^\circ} 
\newcommand\Stab{\on{Stab}} 
\newcommand\Stap{\Stab^\circ} 

\newcommand{\EG}{\on{EG}} 
\newcommand{\EGp}{\EG^\circ}       
\newcommand{\SEG}{\on{SEG}} 
\newcommand{\SEGp}{\SEG^\circ}       

\newcommand{\ST}{\on{ST}}  

\newcommand{\EGT}{\EG^{\Te}}
\newcommand{\EGx}{\EG^{\times}}
\newcommand{\uEG}{\underline{\EG}} 
\newcommand{\uEGx}{\uEG^\times} 
\newcommand{\CEG}{\on{CEG}} 
\newcommand{\uCEG}{\underline{\CEG}} 

\newcommand\egt{\on{\hua{EG}}^{\Tx}}
\newcommand\eg{\on{\hua{EG}}}
\newcommand\ceg{\on{\hua{CEG}}}

\renewcommand{\k}{\mathbf{k}}
\newcommand{\Ho}[1]{\on{\bf H}_{#1}}
\newcommand{\tilt}[3]{{#1}^{#2}_{#3}}

\newcommand\Sph{\on{Sph}}

\newcommand{\numarc}{n}
\newcommand{\numtri}{\aleph}

\newcommand{\Tri}{\bigtriangleup}

\newcommand\Homeo{\on{Hoemo}} 
\newcommand{\MCG}{\on{MCG}} 
\newcommand{\BT}{\on{BT}}  

\newcommand\Bt[1]{\on{B}_{#1}}
\newcommand\bt[1]{\on{B}_{#1}^{-1}}

\newcommand\M{\mathbf{M}} 
\renewcommand\P{\vortex} 

\newcommand{\Quad}{\on{Quad}}
\newcommand{\FQuad}[2]{\on{FQuad}^{#1}(#2)}

\newcommand{\Zer}{\on{Zero}}
\newcommand{\Pol}{\on{Pol}}

\newcommand{\CA}{\on{CA}}

\newcommand{\SBr}{\on{SBr}}

\newcommand{\UHP}{\mathbf{H}} 
\newcommand{\skel}{\wp} 
\newcommand{\cub}{\on{U}} 

\newcommand\rs{\mathbf{X}} 
\newcommand\surp{\rs^\circ}
\newcommand\dd{d} 

\newcommand{\pha}{\varphi} 

\newcommand{\on}[1]{\operatorname{#1}}
\newcommand\class{\mathfrak{Q}}

\usepackage{pifont}
\def\sun{\text{\Sun}}
\def\dpole
    {node[white]{$\bullet$} node[]{\scriptsize{\sun}}}
\def\spole
    {node[white]{$\bullet$} node[red]{\scriptsize{\sun}}}
\def\vot{node[white,yshift=-.09ex]{\Large{$\bullet$}}node[rotate=-45]{\scriptsize{\Yinyang}}}
\def\tov{node[white,yshift=-.09ex]{\Large{$\bullet$}}node[rotate=135]{\scriptsize{\Yinyang}}}
\def\snn{node{\tiny{$\bullet$}}}
\def\nn{node{$\bullet$}}
\def\ww{node[white]{$\bullet$}node[red]{$\circ$}}
\def\hh{PineGreen}   
\def\hhd{blue!49}  
\def\ql{cyan!80}  
\def\qh{orange!75} 
\def\sepa{black}

\def\Lrm{\mathrm{L}}
\def\LL{{\mathrm{L}^2(\sun)}}

\def\DV{\mathrm{DV}}    
\def\DVs{\DV(\dsp)}     
\def\cp{collision path }
\def\uk{\mathbf{k}}
\def\us{\mathbf{s}}

\def\ZZP{\ZZ_2^{\oplus \vortex}}

\def\vortex{{\text{\Yinyang}}}
\newcommand\surf{\mathbf{S}}  
\newcommand\surfp{\mathbf{S}^{\sun}}  
\newcommand\surfv{\surf^{\vortex}}  
\newcommand\sop{\surfp_\Tri}  
\newcommand\surfo{\surf_\Tri}  
\newcommand\sov{\surfv_\Tri}  
\newcommand\sovs{ \surfv_{\Tri-\us} }   
\def\FM{F^{\text{\Yinyang}}}

\newcommand\smp{\mathbf{S}_{}^{\sun}}
\newcommand\smv{\mathbf{S}_{}^{\vortex}}

\newcommand\RT{\mathrm{T}} 
\newcommand\T{\mathbb{T}} 
\def\sign{\varepsilon}
\newcommand\RTe{{\RT_\sign}} 
\newcommand\RTx{\RT_\times} %
\newcommand\Te{\T_\sign} 
\newcommand\Tx{\T_\times} 
\def\sx{\text{\Cancer}}

\def\cont{\widetilde{\psi}_\us}     

\def\be{\begin{equation}}
\def\ee{\end{equation}}
\newcommand{\Qy}[1]{\textcolor{blue!69}{Qy:#1}}

\def\Phib{\Phi_{\sun}}
\def\Phiv{\Phi_{\vortex}}
\def\phib{\phi_{\sun}}
\def\phiv{\phi_{\vortex}}
\def\pibt{\pi_{B}}
\def\pisbr{\pi_{S}}

\def\rhob{\rho_{\sun}}
\def\rhov{\rho_{\vortex}}

\def\dgen{2\mathbb{N}_{\leq g}-1}
\def\dgene{2\mathbb{N}_{\leq g}}
\newcommand\iv[1]{\underline{#1}}
\newcommand\Co{\on{Co}}
\newcommand\Cop{\on{CO}(\sun)}
\newcommand\CON{\on{CO}^{\mathrm{N}}(\sun)}
\newcommand\COL{\on{CO}^{\mathrm{L}}(\sun)}

\def\tx{\widetilde{x}}

\def\AJ{\on{AJ}}
\def\kg{\kern 0.1em}

\newcommand\sbt{\on{w}}  
\newcommand\lb{\on{b}}

\newcommand{\paiwei}[1]{c(#1)}
\newcommand{\dianshu}[1]{d(#1)}
\newcommand{\newtau}[1]{\tau'_{#1}}
\def\z2{\ZZ_{2}}

\newcommand\qm{\on{S}}  
\newcommand\qmp{\qm^{\sun}}  
\newcommand\qmv{\qm^{\vortex}}  
\newcommand\dsp{\qmp_\Tri}  
\newcommand\dsv{\qmv_\Tri}  

\def\TAU{\Lambda_r}
\def\COR{\text{\Taurus}_r}

\def\SQuad{ \on{FQuad}^\times(\surfv) }

\title[Cluster braid group versus braid twist group]
{Decorated Marked Surfaces with vortices: \\Cluster braid group vs. braid twist group}
\author{Yu Qiu}
\address{Qy: Yau Mathematical Sciences Center and Department of Mathematical Sciences, Tsinghua University, 100084 Beijing, China.
    \&
Beijing Institute of Mathematical Sciences and Applications, Yanqi Lake, Beijing, China}
\email{yu.qiu@bath.edu}

\author{Yu Zhou}
\address{Zy:
School of Mathematical Sciences,
Beijing Normal University,
100875 Beijing, China}
\email{yuzhoumath@gmail.com}

\dedicatory{Dedicated to the memory of Idun Reiten}
\begin{document}
\begin{abstract}
Let $\surfv$ be a marked surface with vortices (=punctures with extra $\ZZ_2$ symmetry).
We study the decorated version $\sov$, where the $\ZZ_2$ symmetry lifts to the relation that the fourth power of the braid twist of any \cp (connecting a decoration in $\Tri$ and a vortex in $\vortex$) is identity.

We prove the following three groups are isomorphic:
King-Qiu's cluster braid group associated to $\surfv$, the braid twist group of $\sov$ and
the fundamental group of Bridgeland-Smith's moduli space 
of $\surfv$-framed GMN differentials.
Moreover, we give finite presentations of such groups.

\end{abstract}
\maketitle

\tableofcontents  
\setlength\parindent{0pt}
\setlength{\parskip}{5pt}


\section{Introduction}

\subsection{Motivations}

\paragraph{\textbf{Fundamental groups and $K(\pi,1)$-conjecture}}\

The topology of spaces of abelian or quadratic differentials has attracted a lot of attention. In particular, Kontsevich-Zorich conjecture that (most of) the strata of moduli spaces of quadratic differentials are $K(\pi,1)$, cf. comments in \cite[\S~1.2]{Q24} and references therein. Except for the genus zero case, there are little progress on such a conjecture.


In the previous series of works \cite{QQ,QQ2,QZ2,QZ3,KQ2},
by adding the set $\Tri$ of decorations on an (unpunctured) marked surface $\surf$, we show that the braid twist group of $\surfo$ is the fundamental group of the moduli space of $\surf$-framed quadratic differentials.

However, compared to the seminal works \cite{FST,BS}, the model above fails with the presence of punctures, which correspond to simple/double poles. In this paper, we fix this flaw by composing certain conditions for punctures (which we then call vortices \footnote{\footnotesize{We will use the notation yinyang \Yinyang$ $ to denote a vortex.}}).

\paragraph{\textbf{Cluster braid groups}}\

One of the key techniques used in \cite{KQ2} is the introduction of cluster braid groups, which are interesting in their own right due to their cluster origin. These groups are naturally related to spherical twist groups in the associated 3-Calabi-Yau categories. One of the byproducts of achieving our goal above is to obtain a presentation of the corresponding cluster braid groups.

\subsection{Topological realizations of \texorpdfstring{$\z2$}{Z2}-symmetry}\

First, let us summarize previous related work on the realization of $\z2$-symmetry for punctures.
\begin{itemize}
  \item In \cite{FST}, the $\z2$-symmetry, via tagging endpoints of arcs, was introduced so that the geometry of the marked surfaces matches the corresponding (combinatorics of) the cluster algebras.
  \item In \cite{BQ}, the $\z2$-symmetry was added manually when considering the mapping class group of marked surfaces.
  \item In \cite{QZ1}, we gave geometric models for cluster categories from marked surfaces, where we essentially categorify tagged arcs in \cite{FST}.
  \item In \cite{QZZ}, we gave a geometric model for perfect derived categories of graded skew-gentle algebras via marked surfaces with binary, where a binary is a boundary component with one open marked point and one closed marked point, subject to an equivalent relation $D^2=1$, where $D$ is the Dehn twist along the binary.
\end{itemize}

\paragraph{\textbf{New $\z2$-symmetry realization}}\

Let $\sop$ be a decorated marked surface (DMS) with a set $\sun$ of punctures (cf. \Cref{sec:DMSp}).

Our new insight in this paper comes from the following idea:
\begin{itemize}
  \item In the $q$-deformation series \cite{IQ1,IQ2,IQZ}, a relation between the geometric model of graded gentle algebras and the geometric model of its Calabi-Yau completion was shown by dragging closed marked points (behaved like zeros) into the interior of the surfaces as decorations.
\end{itemize}
What we would like to do is to understand the $\z2$-symmetry realization in \cite{QZZ} by pulling the closed marked point in a binary into the interior as a decoration, and regarding the open marked point as a puncture. More precisely, this leads to the following construction.

\begin{consA}
We will turn punctures into vortices by imposing the condition that
\[
    \Bt{s}^4=1, \quad \forall s\in\DVs.
\]
The resulting surface will be denoted by $\sov$ and called , and referred to as a DMSv (i.e., a decorated marked surface with vortices).
\end{consA}

\begin{vent}
We will use the following convention in all the figures:
\begin{itemize}
    \item Punctures will be drawn as $\sun$ and vortices will be drawn as yinyang \Yinyang$ $ to indicate extra $\z2$-symmetry structure.
    \item For signed vortices, we will draw the following:
    \begin{gather}\label{eq:Yinyang sign}
    \begin{tikzpicture}
    \draw(0,0)node[rotate=-45]{\large{\Yinyang}} node[right]{$\;:=$ \sun$ $ with sign $+$};
    \draw(0,-1)node[rotate=135]{\large{\Yinyang}} node[right]{$\;:=$ \sun$ $ with sign $-$};
    \end{tikzpicture}
    \end{gather}
    \item Marked points will be drawn as black bullets {\Large $\bullet$}. When a point could be either a vortex (where the sign is not involved) or a marked point, we will treat it as a marked point.
    \item Decorations will be drawn as red circles {\color{red} $\circ$}.
    \item An open arc is a curve connecting marked points or vortices. Open arcs will be drawn in blue/cyan/green, where a green arc is a loop that is supposed to be the sum of two open arcs.
    \item A closed arc is a curve connecting decorations. Closed arcs will be drawn in red/orange/violet, where a violet arc is a loop that is supposed to be the difference between two open arcs.
    \item A \cp is a simple arc that connects a vortex and a decoration; it is drawn in gray.
\end{itemize}
\end{vent}

\begin{figure}[ht]\centering
\begin{tikzpicture}[scale=.5]
\draw[fill=green!19,green!19](0,0)ellipse (.9 and .9);
\draw[green!80,very thick,-=stealth](0,.9)arc(90:-85:.9);
\draw[green!80,very thick,-stealth](0,-.9)arc(-90:-85-180:.9);
\draw[gray,ultra thick](0,-1)to node[left]{$s$}(0,1);
\begin{scope}[shift={(0,0)}]
\draw[red,ultra thick]
    (0,-1)arc(-90:-270:1.5)arc(90:-90:2.5)arc(-90:-270:3.5);
\draw[\ql,ultra thick]
    (0,1)arc(90:-90:1.5)arc(-90:-270:2.5)arc(90:-90:3.5);
\draw(0,4)\ww(0,-4)\nn(0,-1)\ww(0,1)\vot;

\draw[line width=3pt,gray,-stealth](2,-4)
    to[bend left=-45]
        node[above]{\Large{$\cong$}}node[below]{{$\Bt{s}^4$}}(-2+14,-4);
\draw[line width=3pt,gray!70,-stealth](2,4)
    to[bend left=45]node[above]{{$\Bt{s}^2$}}node[below]{{$\sx$}}(-2+7,4);
\end{scope}

\begin{scope}[shift={(14,0)}]
\draw[fill=green!19,green!19](0,0)ellipse (.9 and .9);
\draw[green!80,very thick,-stealth](0,.9)arc(90:-85:.9);
\draw[green!80,very thick,-stealth](0,-.9)arc(-90:-85-180:.9);
\draw[gray,ultra thick](0,-1)to node[left]{$s$}(0,1);
\draw[red,ultra thick]
    (0,-1)arc(-90:90:2.5);
\draw[\ql,ultra thick]
    (0,1)arc(90:270:2.5);
\draw(0,4)\ww(0,-4)\nn(0,-1)\ww(0,1)\vot;
\end{scope}

\begin{scope}[shift={(21,0)},xscale=-1]
\draw[fill=green!19,green!19](0,0)ellipse (.9 and .9);
\draw[green!80,very thick,-stealth](0,.9)arc(90:-85:.9);
\draw[green!80,very thick,-stealth](0,-.9)arc(-90:-85-180:.9);
\draw[gray,ultra thick](0,-1)to node[left]{$s$}(0,1);
\draw[\qh,ultra thick]
    (0,-1)arc(-90:-270:1.5)arc(90:-90:2.5)arc(-90:-270:3.5);
\draw[blue,ultra thick]
    (0,1)arc(90:-90:1.5)arc(-90:-270:2.5)arc(90:-90:3.5);
\draw(0,4)\ww(0,-4)\nn(0,-1)\ww(0,1)\tov;
\end{scope}

\begin{scope}[shift={(7,0)}]
\draw[fill=green!19,green!19](0,0)ellipse (.9 and .9);
\draw[green!80,very thick,-stealth](0,.9)arc(90:-85:.9);
\draw[green!80,very thick,-stealth](0,-.9)arc(-90:-85-180:.9);
\draw[gray,ultra thick](0,-1)to node[left]{$s$}(0,1);
\draw[\qh,ultra thick]
    (0,-1)arc(-90:-270:2.5);
\draw[blue,ultra thick]
    (0,1)arc(90:-90:2.5);
\draw(0,4)\ww(0,-4)\nn(0,-1)\ww(0,1)\tov;

\draw[line width=3pt,gray,-stealth](2,4)
    to[bend left=45]
        node[below]{\Large{$\cong$}}node[above]{{$\Bt{s}^4$}}(-2+14,4);
\draw[line width=3pt,gray!70,-stealth](2+7,-4)
    to[bend left=-45]node[below]{{$\Bt{s}^2$}}node[above]{{$\sx$}}(-2+14,-4);
\end{scope}
\end{tikzpicture}
\caption{Applying $\Bt{s}^2$: from left to right}
\label{fig:4}
\end{figure}
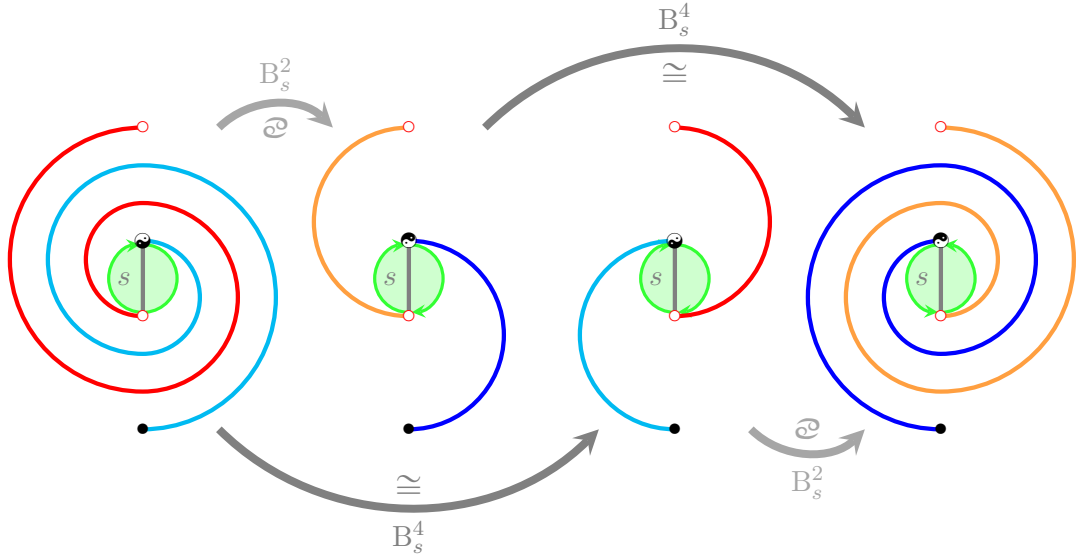

In Figure~\ref{fig:4}:
\begin{itemize}
  \item The green boundary component $\partial$ plays the role of a binary, in the sense of \cite{QZZ}, and corresponds to the neighbourhood of the collision path $s$.
  \item The decoration and the vortex, as endpoints of $s$, can be regarded as closed and open marked points on the binary $\partial$.
  \item The thick gray arrows indicate the identifications of $\Bt{s}^4=1$,
  which corresponds to the relation $D^2=1$ in the binary setting, where $s^2=D$ is the Dehn twist along $\partial$.
\end{itemize}

Denote by $\LL$ the subgroup of the mapping class group of $\sop$ generated by $\Bt{s}^4$ for all $s\in\DVs$.
Then the braid twist group for $\sov$ is
\[
    \BT(\sov)= \BT(\sop)/ \Cop,
\]
where $\Cop=\LL\cap\BT(\sop)$.

It is worth noting that this relation was previously used in \cite{All} to produce a type $D$ braid group. Our result can be seen as a generalization of his work to a more general setting.


\subsection{Context and results}\


In \Cref{sec:pre}, we recall the basics about decorated marked surfaces with punctures (DMSp). In \Cref{part:B}, we study the vortex version (DMSv) of DMSp and calculate a presentation of the braid twist group of DMSv. The key calculation is to find normal generators of the group $\Cop$ mentioned above. The main result of this section is as follows.

\begin{thmB}[{\Cref{thm:pre for sov}}]
Let $\aleph$ be the number of decorations in $\Tri$. Suppose that either $\aleph\geq 5$, or $\aleph=4$ and $2g+b+p-1\leq 2$.
Then the braid twist group $\BT(\dsv)$ admits the following finite presentation:
\begin{itemize}
	\item Generators: $\sigma_1,\cdots,\sigma_{\aleph-1},\tau_1,\cdots,\tau_{2g+b+p-1}$ (see Figure~\ref{fig:BT}).
    \item Relations: for $1\leq i,j\leq \aleph-1$ and $1\leq r,s\leq 2g+b+p-1$,
		\begin{align*}
		&\Co(\sigma_i,\sigma_j)&&  \text{if $|i-j|>1$}\\
		&\Br(\sigma_i,\sigma_{j})&&\text{if $|i-j|=1$}\\
		&\Co(\tau_r,\sigma_i)&&\text{if $i>2$}\\
		&\Br(\tau_r,x)&&\\
		&\Br(\tau_r,y)&&\\
		&\Co({\tau_r}^y,{\tau_s}^x)&&\text{if $s<r$ and $s\notin\dgen$}\\
		&\Co({\tau_r}^{\iv{y}},{\tau_s}^x)&&\text{if $s<r$ and $s\in\dgen$}\\
        &\Co(\tau_r,\tau_{r-1})&&\text{if $2g+b\leq r\leq 2g+b+p-1$.}
		\end{align*}
\end{itemize}
\end{thmB}

In \Cref{part:C}, we recall the basics about the cluster exchange groupoids and cluster braid groups. In \Cref{sec:EG}, we construct the exchange graph of signed/tagged triangulations for DMS with vortices. Based on this construction, in \Cref{sec:BT=CT} we prove that the cluster braid group $\CBr(\Tx)$ is naturally isomorphic to the braid twist group:

\begin{thmC}[\Cref{thm:BT=CT}]\label{thmC}
There is a natural isomorphism
$\BT(\Tx)\cong\CBr(\Tx)$.
\end{thmC}

In \Cref{part:D}, we apply the results above to obtain the following.

\begin{thmD}[\Cref{thm:pi1}]
For a DMSv $\sov$, we have
\begin{equation}\label{eq:thmC}
  \begin{cases}
     \pi_0\FQuad{}{\sov}\cong\Ho{1}(\surfv),\\
     \pi_1\FQuad{\emph{}\Te}{\sov}=1.
  \end{cases}
\end{equation}
The second one is equivalent to $\pi_1\FQuad{\emph{}}{\surfv}=\BT(\sov)$.
\end{thmD}

This theorem is the vortex version of \cite[Thm.~1.1]{Q24}, which upgrades the unpunctured version in \cite{KQ2}. The new idea/trick here is the construction of the $\sov$-framed version $\FQuad{\T}{\sov}$ introduced in \Cref{def:FQuad O}.

\subsection{Comparison with other works}

\paragraph{\textbf{Braid groups associated to quivers with potential}}\

By presentations of cluster exchange groupoids provided in the recent work \cite[\S~3]{BHR},
we obtain the following corollary of \Cref{thmC} directly.
Note that one can also deduce this from Theorem~\textbf{B} (but with tedious calculations).

\begin{CorE}
If any arcs in $\Tx$ is simple (with different endpoints) and the associated quiver $Q_{\Tx}$ has no double arrows,
then $\BT(\sov) = \BT(\Tx^*)$ is canonically isomorphic to the \emph{algebraic twist group} associated to the quiver with potential $(Q_{\Tx},W_{\Tx})$, in the sense of \cite[Def.~10.1]{QQ}.

In other words, it admits a presentation with
\begin{itemize}
  \item Generators: $b_\eta$ for each $\eta\in\Tx^*$;
  \item Relations:
    \begin{itemize}
      \item For any two generators, there is a commutation or a braid relation between them depending on if there is zero or one arrow between the corresponding vertices in $Q_{\Tx}$.
      \item Each term/cycles in the potential $W_{\Tx}$ contributes a cyclic relation \eqref{eq:cyclic-R}, cf. \cite[\S~10]{QQ} for more details.
    \end{itemize}
\end{itemize}
\end{CorE}

\paragraph{\textbf{DMS with mixed punctures and vortices}}\

Theorem~\textbf{B}, Theorem~\textbf{D} and Corollary~\textbf{E}  can be easily generalized to the case where the DMS has both punctures and vortices, introduced in \cite{Q24}.
We omit this generalization as it involves heavy notation and yields few new insights.

\paragraph{\textbf{Cluster Weyl groups}}\
Define the \emph{Weyl group} $W(\Tx^*)$ associated to $\CBr(\Tx^*)$ to be its quotient group by adding square-to-one relations for each generator.
Then by comparing the presentations, one sees that this is precisely the one studied in \cite{FeSTu}, cf. \cite{FeLSTu}.

Note that here, as in \cite{FST, FeSTu}, only vortices (=type D punctures) are considered.
In \cite{BHR}, other type of punctures (B/C) are considered.
Using folding/covering technique, one can easily generalize the statement above to the generality considered in \cite{BHR}, where we omit the details again.

Another related work is \cite{FeMa}, who studies braid group of the complex reflection group $G(d,d,n)$ via disk model with orbifold points.

\subsection*{Acknowledgments}
Qy would like to thank Alastair King, Zhe Han, and Akishi Ikeda for interesting discussions on the moduli space of quadratic differentials with punctures, which inspired the idea of vortices. This work was supported by National Natural Science Foundation of China (Grant Nos.~12425104, 12271279 and 12031007) and National Key R\&D Program of China (No. 2020YFA0713000).


\section{Preliminaries}\label{sec:pre}

\subsection*{Convention}

\begin{itemize}
\item \textbf{Composition:} We write $a\circ b$ as $a\kg b$.
\item \textbf{Inverse} We write $s^{-1}$ as $\iv{s}$.
\item \textbf{Conjugation:} We write $\iv{b}\kg a\kg b$ as $a^b$. Then, $\iv{(a^b)}=\iv{a}^b$.
\item \textbf{Relation:} We use the following notation for relations throughout this paper.
\[\begin{array}{llll}
\text{Commutation relation}&\Co(a,b)&\colon& a\kg b=b\kg a.\\
\text{Braid relation}&\Br(a,b)&\colon& a\kg b\kg a=b\kg a\kg b.\\
\text{Triangle relation}&\on{Tr}(a,b,c)&\colon& a\kg b\kg c\kg a=b\kg c\kg a\kg b=c\kg a\kg b\kg c.\\
\text{Rectangle relation} & \on{Rec}(a,b,c,d)&\colon&a\kg b\kg c\kg d\kg a\kg b=b\kg c\kg d\kg a\kg b\kg c=c\kg d\kg a\kg b\kg c\kg d=d\kg a\kg b\kg c\kg d\kg a.
\end{array}\]
\end{itemize}

Note that the triangle and rectangle relations are special cases of cyclic relations in \cite[Def.~10.1]{QQ}. More precisely, a cyclic relation $\on{Cyc}(a_1,\cdots,a_m)$ (of rank $m$) is one of the equivalent relations of the form (provided commutative/braid relations between $a_s$'s, cf. \cite[Lem.~10.2]{QQ})
\begin{gather}\label{eq:cyclic-R}
    \prod_{s=i}^{i+2m-3} a_s = \prod_{s=j}^{j+2m-3} a_s
\end{gather}
for any chosen $1\le i\ne j\le m$ and the subscript is taken in $\ZZ_m$.

The following two subsets of the set $\mathbb{N}$ of natural numbers will be used throughout the paper:
\[\dgene=\{2h\mid h\in\mathbb{N}\text{ with }1\leq h\leq g \};\]
\[\dgen=\{2h-1\mid h\in\mathbb{N}\text{ with }1\leq h\leq g \}.\]

\subsection{Decorated surfaces with puncture}\label{sec:DMSp}

A \emph{punctured surface} $\qmp=(\qm,\sun)$ is an oriented topological surface $\qm$ with nonempty boundary $\partial\qm$ and a set $\sun=\{V_1,\cdots,V_p\}$ of punctures in the interior $\qm\setminus\partial\qm$. We use the superscript $\sun$ on $\qm$ to emphasize the role of the set of punctures.

Up to homeomorphism, $\qmp$ is determined by the following data:
\begin{itemize}
\item the genus $g$ of the surface $\qm$;
\item the number $b$ of connected components of $\partial\qm$;
\item the number $p=\#\sun$ of punctures.
\end{itemize}

The \emph{mapping class group} $\MCG(\qmp)$ of $\qmp$ is defined to be the group of (isotopy classes of) homeomorphisms of $\qm$ that preserve $\sun$ setwise.

A \emph{decorated surface} $\dsp=(\qm,\sun,\Tri)$ is obtained from $\qmp$ by decorating a set $\Tri=\{Z_1,\cdots,Z_{\aleph}\}$ of points, called \emph{decorations}. The \emph{mapping class group} $\MCG(\dsp)$ of $\dsp$ is defined to be the group of (isotopy classes of) homeomorphisms of $\qm$ that preserve $\sun$ and $\Tri$, respectively, setwise.

\subsection{Braid twists and L-twists}

Let $\qm^\circ=\dsp-\partial\qm-\sun-\Tri$.

\begin{definition}\label{def:arcs}
The following types of curves on $\dsp$ are considered.
\begin{itemize}
\item A \emph{closed arc} in $\dsp$ is a simple curve whose interior is in $\qm^\circ$ and that connects two different decorations in $\Tri$. Denote by $\CA(\dsp)$ the set of closed arcs in $\dsp$.
\item An \emph{L-arc} $\eta$ in $\dsp$ is a simple curve with interior lying in $\qm^\circ$ such that its endpoints coincide at a decoration in $\Tri$, and it is not isotopic to a point.
\item A \emph{general closed arc} in $\dsp$ is either a closed arc or an L-arc. Denote by $\overline{\CA}(\dsp)$ the set of general closed arcs in $\dsp$.
\item A \emph{\cp}in $\dsp$ is a simple curve whose interior is in $\qm^\circ$ and which connects a decoration and a puncture. Denote by $\DVs$ the set of collision paths.
\end{itemize}
\end{definition}

\begin{definition}
The \emph{braid twist} $B_{\eta}$ along either a closed arc or a \cp $\eta$ is defined by the upper picture of Figure~\ref{fig:T}. The \emph{L-twist} $\mathrm{L}_{\delta}$ along an L-arc $\delta$ is defined by the lower picture of Figure~\ref{fig:T}.
\end{definition}

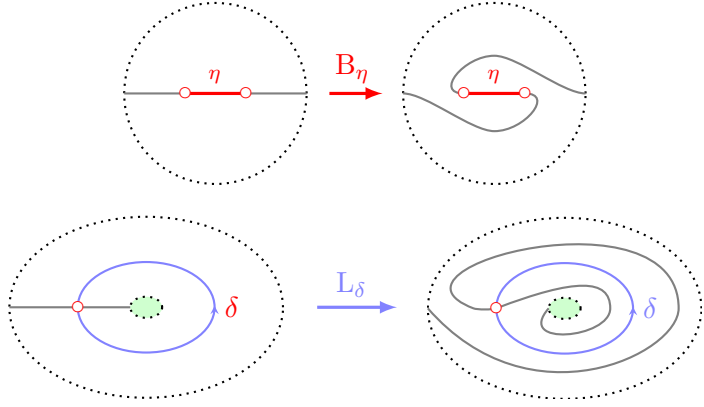
\begin{figure}[ht]\centering
\begin{tikzpicture}[scale=.2]
  \draw[dotted,thick](0,0)circle(6)node[above,red]{$_\eta$};
  \draw[red, very thick](-2,0)to(2,0);
  \draw[gray,thick](2,0)to(6,0)(-2,0)to(-6,0);
  \draw(-2,0)\ww(2,0)\ww;
  \draw[red,thick](0:7.5)edge[very thick,->,>=latex](0:11)
    (0:9)node[above]{$\Bt{\eta}$};
\end{tikzpicture}\;
\begin{tikzpicture}[scale=.2]
  \draw[dotted,thick](0,0)circle(6)node[above,red]{$_\eta$};
  \draw[red, very thick](-2,0)to(2,0);
  \draw[gray,thick](2,0).. controls +(0:2) and +(0:2) ..(0,-2.5)
    .. controls +(180:1.5) and +(0:1.5) ..(-6,0)
    (-2,0).. controls +(180:2) and +(180:2) ..(0,2.5)
    .. controls +(0:1.5) and +(180:1.5) ..(6,0);
  \draw(-2,0)\ww(2,0)\ww;
\end{tikzpicture}

	\begin{tikzpicture}[xscale=.3,yscale=.2]
	\draw[dotted,thick](0,0)circle(6);
    \draw[dotted,thick,fill=green!19](0,0)circle (.7);
	\draw[gray,thick](-6,0)to(-.7,0);
	\draw[\hhd,thick] (0,0) circle (3);
	\draw[\hhd,<-,>=stealth](3,0.1)to(3,-0.01);
	\draw(-3,0)\ww(3,0)node[right,red]{$\delta$};
	\draw[\hhd](0:7.5)edge[very thick,->,>=latex](0:11);
    \draw[\hhd,thick](0:9)node[above]{$\mathrm{L}_{\delta}$};
	\end{tikzpicture}\quad
	\begin{tikzpicture}[xscale=.3,yscale=.2]\clip(-6,8)rectangle(6,-6);
	\draw[dotted,thick](0,0)circle(6);
    \draw[dotted,thick,fill=green!19](0,0)circle (.7);
	\draw[\hhd,thick] (0,0) circle (3);
	\draw[\hhd,<-,>=stealth](3,0.1)to(3,-0.01);
	\draw[gray,thick](-6,0)
    .. controls +(-60:9) and +(-90:3) ..(5,0)
	.. controls +(90:8) and +(150:5.5) ..(-4,0)
	.. controls +(-30:.2) and +(-150:.2) ..(-3,0)
	.. controls +(30:4) and +(90:2) ..(2,0)
	.. controls +(-90:2) and +(-120:3) ..(-.7,0);
	\draw(-3,0)\ww(3,0)node[right,\hhd]{$\delta$};
	\end{tikzpicture}
	\caption{The braid twist and L-twist}
	\label{fig:T}
\end{figure}

We have the following well-known formula
\begin{equation}\label{eq:Psi}
    B_{\Psi(\eta)}=\Psi\circ B_{\eta}\circ \Psi^{-1},
\end{equation}
for any $\Psi\in\MCG(\dsp)$ and any $\eta$ belonging to either $\CA(\dsp)$ or $\DVs$.

For any \cp $s\in\DVs$, we associate an L-arc $\delta=s^\circlearrowleft$ which encloses it, as shown in Figure~\ref{fig:comp}, which is called the \emph{completion} of $s$.

\begin{figure}[htpb]
    \centering
    \begin{tikzpicture}[scale=.5]\clip(-4,2)rectangle(2,-2.5);
  \draw[\hhd,very thick](-3,0).. controls +(45:7) and +(-45:7) ..(-3,0)
        (-1,-1.3)node[below]{$\delta=s^\circlearrowleft$};
  \draw[gray,very thick](-3,0)tonode[above]{$s$}(0,0);
  \draw(-3,0)\ww(0,0)\dpole;
\end{tikzpicture}
    \caption{The completion of a \cp}
    \label{fig:comp}
\end{figure}

Although the braid twist along a \cp does not lie in $\MCG(\dsp)$, its square does.

\begin{lemma}\label{lem:non-trivial}
    Let $s\in\DVs$ and $\delta=s^\circlearrowleft$. Then $\Bt{s}^2=\mathrm{L}_{\delta}$ in $\MCG(\dsp)$.
\end{lemma}

\begin{proof}
    This follows from straightforward verification.
\end{proof}

\paragraph{\textbf{The surface braid group}}\

A classical generalization of the braid group of a decorated marked surface is given as follows.

\begin{definition}[cf. \cite{BG}]
    The \emph{surface braid group} $\SBr(\dsp)$ is defined to be the kernel of the forgetful map
    \begin{equation}\label{eq:forget1}
        F^\sun\colon \MCG(\dsp)\to \MCG(\qmp),
    \end{equation}
    that is, the fundamental group of the configuration space of $\aleph$ points in (the interior of) $\dsp$, based at the set $\Tri$.
\end{definition}

\paragraph{\textbf{The braid twist group}}\
\begin{definition}\cite[Def.~4.1]{QQ}\label{def:BT}
    The \emph{braid twist group} $\BT(\dsp)$ of the decorated surface $\dsp$ is the subgroup of $\MCG(\dsp)$ generated by the braid twists $\Bt{\eta}$ for $\eta\in\CA(\dsp)$.
\end{definition}

By formula \eqref{eq:Psi}, the group $\BT(\dsp)$ is a normal subgroup of $\MCG(\dsp)$.

\paragraph{\textbf{The kernel of the Abel-Jacobi map}}\

Let
\begin{gather}\label{eq:H}
    H:=\pi_1(\qmp,Z_1) = \SBr(\qmp_{Z_1}) \le \SBr(\dsp).
\end{gather}
Then, $\Ho{1}(\qmp)=H/[H,H]$.

Recall from \cite[Def.~2.3]{Q24} that the Abel-Jacobi map is a group homomorphism
\begin{gather}\label{eq:Xi}
  \AJ^\sun\colon\SBr(\dsp)\to\Ho{1}(\qmp),
\end{gather}
which sends an element to the product of loops formed by the traces of decorations.

The following theorem was first stated in \cite[Lem.~5.7]{BMQS}, and is slightly generalized in \cite[Thm.~2.6]{Q24}.

\begin{theorem}\label{thm:QZ+}
Suppose that $\aleph\geq 3$.
Then $\BT(\dsp)=\ker\AJ^\sun$. Equivalently, there is a short exact sequence of groups
\begin{gather}\label{eq:SES-b}
    1 \to \BT(\dsp) \to \SBr(\dsp) \xrightarrow{\AJ^{\sun}} \Ho{1}(\qmp) \to 1.
\end{gather}
Moreover, we have
\begin{gather}\label{eq:HH}
    H\cap \BT(\dsp)=[H,H].
\end{gather}
\end{theorem}

For use in \Cref{app:A}, we record the following lemma, which can be checked by direct computation.

\begin{lemma}[{\cite[Lem.~2.5]{Q24}}]\label{lem:tech}
For $1\le s<r\le 2g+b+p-1 $, the commutator
\begin{equation}\label{eq:in-BT}
    [\varepsilon_s,\varepsilon_r]=
    \begin{cases}
        \iv{\tau_{s}b}a\iv{\tau_{r}a\tau_{s}}ab\tau_{r}b  & \text{if $s+1\notin \dgene$},\\
        \iv{bb\tau_{s}b}a\tau_{r}\iv{a}\tau_{s}ab\tau_{r}b  & \text{if $s+1\in \dgene$},
    \end{cases}
\end{equation}
belongs to $\BT(\dsp)$. Denote by $n_{s,r}$ the inverse of the right-hand side of \eqref{eq:in-BT}.
\end{lemma}

\subsection{Finite presentations of \texorpdfstring{$\BT(\dsp)$}{BT}}

We consider the closed arcs
\begin{equation}\label{eq:gen}
  \sigma_1,\cdots,\sigma_{\aleph-1},\delta_1,\cdots,\delta_{2g+b+p-1},
\end{equation}
as shown in \Cref{fig:B's}. The corresponding braid twists of these arcs form a set of generators of $\SBr(\dsp)$ (see \cite{BG}).

\begin{figure}[hbt]\centering
	\begin{tikzpicture}[scale=.35]
	\draw[red](10,-7)to(14,-9)(24,-9)to(20,-9);\draw[dashed,red](14,-9)to(24,-9);

	\draw[red](10,-7)to(8+2+4-.5,0);
	\draw[red,dashed](8+2+4-.5,0).. controls +(60:1) and +(180:1) ..(17,5);
	\draw[red,->-=.5,>=stealth](10,-7).. controls +(45:5) and +(0:3) ..(17,5);
	\draw[red,-<-=.5,>=stealth](10,-7).. controls +(35:5) and +(-100:2) ..(20+2+4-.3,0.3);
	\draw[red,->-=.5,>=stealth]
              (10,-7).. controls +(18:8) and +(-25:4) ..($(24,0)!-.2!(10,-7)$);
	\draw[red](10,-7).. controls +(40:8) and +(155:4) ..($(24,0)!-.2!(10,-7)$);
	\draw[red,->-=.5,>=stealth]
              (10,-7).. controls +(55:7) and +(-5:4) ..($(12,0)!-.2!(10,-7)$);
	\draw[red](10,-7).. controls +(95:7) and +(175:4) ..($(12,0)!-.2!(10,-7)$);
	\draw[red,->-=.5,>=stealth]
              (10,-7).. controls +(113:7) and +(30:2) ..($(6,0)!-.15!(10,-7)$);
	\draw[red](10,-7).. controls +(126:7) and +(-150:2) ..($(6,0)!-.15!(10,-7)$);
	\draw[red,->-=.5,>=stealth](10,-7).. controls +(140:10) and +(45:2) ..($(1,0)!-.1!(10,-7)$);
	\draw[red](10,-7).. controls +(145:10) and +(-135:2) ..($(1,0)!-.1!(10,-7)$);
	\draw[red,->-=.5,>=stealth](10,-7).. controls +(160:10) and +(79:1) ..($(-1,-4)!-.07!(10,-7)$);
	\draw[red](10,-7).. controls +(165:10) and +(-101:1) ..($(-1,-4)!-.07!(10,-7)$);
	\draw[red,->-=.5,>=stealth](10,-7).. controls +(183:10) and +(95:1) ..($(-1,-8)!-.07!(10,-7)$);
	\draw[red](10,-7).. controls +(188:10) and +(-85:1) ..($(-1,-8)!-.07!(10,-7)$);
	\draw[red,dashed](20+2+4-.3,0.3).. controls +(80:1) and +(180:1) ..(28,5);
	\draw[red](10,-7).. controls +(8:27) and +(0:2) ..(28,5);
	
	\foreach \j in {1,2.5} {
		\draw[thick](8*\j-2+4+.5,0).. controls +(45:1) and +(135:1) ..(8*\j+2+4-.5,0);
		\draw[very thick](8*\j-2+4+.3,0.3).. controls +(-60:1) and +(-120:1) ..(8*\j+2+4-.3,0.3);
	}
	\draw[very thick]
	   (-3,5)to(32,5)(32,-11)  (-3,-11)to(32,-11)
	   (-3,5)to[bend right=90](-3,-11);
	\draw[very thick,fill=gray!14] (32,-3) ellipse (1 and 8) 	node{\tiny{$\partial_{1}$}};
	\draw[very thick,fill=gray!14](1,0)circle(.6) node{\tiny{$\partial_{b}$}};
	\draw[very thick,fill=gray!14](6,0)circle(.6) node{\tiny{$\partial_{2}$}};
	\node at(3.5,0) {$\cdots$};
	\node at(18,0) {$\cdots$};
	\foreach \x/\y in {14/-9,10/-7,24/-9,20/-9}
	   {\draw(\x,\y)\ww;}
    \draw (10,-7)node[below]{\small{$Z_{1}$}}
    	(14,-9)node[above]{\small{$Z_2$}}
    	(20,-9)node[above]{\small{$Z_{\aleph-1}$}}
    	(24,-9)node[above]{\small{$Z_{\aleph}$}};
	\draw[red](-.5,2.2)node{$\delta_{2g+b-1}$} (6,2.2)node{$\delta_{2g+1}$}
    	(12,2.3)node{$\delta_{2g}$} (19,2.8)node{$\delta_{2g-1}$}
    	(23.5,2.4)node{$\delta_2$} (28,-3)node{$\delta_{1}$}
    	(11.5,-8.5)node{$\sigma_1$} (22,-10)node{$\sigma_{\aleph-1}$};
    \draw[]
        (-1,-4)\dpole (-2.5,-4)node{$P_1$} (-1,-2.5)node[red]{$\delta_{2g+b}$}
        (-1,-6)node[rotate=90]{$\cdots$}
        (-1,-8)\dpole (-2.5,-8)node{$P_p$} (-1,-9.5)node[red]{$\delta_{2g+b+p-1}$};
	\end{tikzpicture}
	\caption{Generators for $\SBr(\dsp)$}
	\label{fig:B's}
\end{figure}

In the rest of this section, to simplify the notation, we will denote by $\eta$ the braid twist $B_\eta$ and denote by $\delta$ the $L$-twist $t_\delta$.

Similarly to \cite{QZ3}, we define $\varepsilon_r$, $1\leq r\leq 2g+b+p-1$ (see \Cref{fig:QZ's}), recursively by
\begin{gather}\label{eq:deltas}
\varepsilon_r=\begin{cases}
    \delta_r\kg\varepsilon_{r-1}&\text{if $r\notin \dgene$,}\\
    \delta_r\kg\varepsilon_{r-2}&\text{if $r\in \dgene$,}
\end{cases}
\end{gather}
where for convenience, $\varepsilon_0$ is taken to be the identity. Conversely, we have
\begin{equation}\label{eq:deltas2}
\delta_r=\begin{cases}
\varepsilon_r\kg\iv{\varepsilon_{r-1}}&\text{if $r\notin \dgene$;}\\
\varepsilon_r\kg\iv{\varepsilon_{r-2}}&\text{if $r\in \dgene$.}
\end{cases}
\end{equation}

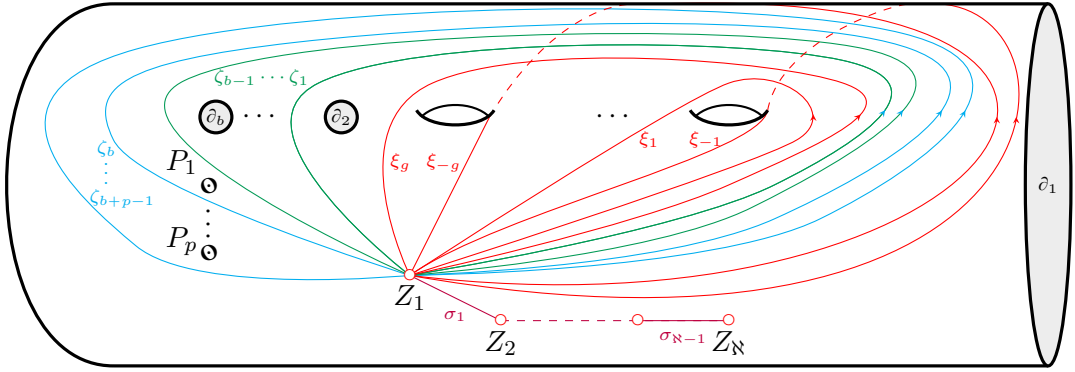
\begin{figure}[htpb]\centering
\begin{tikzpicture}[xscale=.3,yscale=.3]
	\draw[purple](10,-7)to(14,-9)(24,-9)to(20,-9);\draw[dashed,purple](14,-9)to(24,-9);
	\draw[purple](12,-8.8)node{\tiny{$\sigma_1$}} (22,-9.8)node{\tiny{$\sigma_{\aleph-1}$}};
\draw[Cyan]plot [smooth,tension=.8] coordinates {(10,-7)(-3,-1)(5,3.3)(32,3)(26,-4.5)(10,-7)};
\draw[Cyan]plot [smooth,tension=.8] coordinates {(10,-7)(-2,-5.7)(-2,3)(31.6,3.7)(28,-4.5)(10,-7)};
 \draw[Cyan](-2.6,-3.6)node{\tiny{$\zeta_{b+p-1}$}};
 \draw[Cyan](-3.3,-2.4)node[rotate=90]{\tiny{$\cdots$}};
 \draw[Cyan](-3.3,-1.4)node{\tiny{$\zeta_{b}$}};

\draw[ForestGreen]plot [smooth,tension=.8] coordinates {(10,-7)(6,1.5)(29,2.5)(26,-3)(10,-7)};
\draw[ForestGreen]plot [smooth,tension=.8] coordinates {(10,-7)(0.1,1.6)(30,3)(25,-4)(10,-7)};
\draw[ForestGreen]plot [smooth,tension=.8] coordinates {(10,-7)(6,1.5)(29,2.5)(26,-3)(10,-7)};
\draw[ForestGreen](3.5,1.7)node{\tiny{$\zeta_{b-1}\cdots\zeta_1$}};
	\draw[red]plot [smooth,tension=.8] coordinates {(10,-7)(24,-3)(29,1.5)(11,1.5)(10,-7)};
	\draw[red]plot [smooth,tension=.5] coordinates {(10,-7)(27,-1)(24,1.5)(10,-7)};
	\draw[red](9.6,-2)node{\tiny{$\xi_g$}} (20.5,-1)node{\tiny{$\xi_1$}};
	\draw[red](10,-7)to(8+2+4-.5,0);
	\draw[red,dashed](8+2+4-.5,0).. controls +(60:2) and +(180:3) ..(20,5);
	\draw[red](10,-7).. controls +(-7:26) and +(2:26) ..(21,5);
	\draw[red](10,-7).. controls +(35:5) and +(-100:2) ..(20+2+4-.3,0.3);
	\draw[red,dashed](20+2+4-.3,0.3).. controls +(80:3) and +(180:1) ..(32,5);
	\draw[red](10,-7).. controls +(-10:30) and +(0:7) ..(33,5);
	\draw[red](11.5,-2)node{\tiny{$\xi_{-g}$}} (23,-1)node{\tiny{$\xi_{-1}$}};
	\foreach \j in {1,2.5} {\draw[thick](8*\j-2+4+.5,0).. controls +(45:1) and +(135:1) ..(8*\j+2+4-.5,0);\draw[very thick](8*\j-2+4+.3,0.3).. controls +(-60:1) and +(-120:1) ..(8*\j+2+4-.3,0.3);}
	\draw[very thick] (-3,5)to(38,5)(38,-11)   (-3,-11)to(38,-11)(-3,5)to[bend right=90](-3,-11);
	\draw[very thick,fill=gray!14] (38,-3) ellipse (1 and 8)node{\tiny{$\partial_{1}$}};
	\draw[very thick,fill=gray!14](7,0)circle(.7)node{\tiny{$\partial_2$}};
	\draw[very thick,fill=gray!14](1.5,0)circle(.7)node{\tiny{$\partial_{b}$}};
	\node at(3.5,0) {$\cdots$};
	\node at(19,0) {$\cdots$};
	\foreach \x/\y in {14/-9,10/-7,24/-9,20/-9}
	{\draw(\x,\y)\ww;}
	\draw (10,-7)node[below]{$Z_{1}$}
    	(14,-9)node[below]{$Z_2$}
    	(24,-9)node[below]{$Z_{\aleph}$};

\foreach \j in {35.75,36.6}{\draw[ultra thin,red,->-=1,>=stealth](\j,0)to(\j,0.001);}
    \draw[ultra thin,red,->-=1,>=stealth](27.7,0.1)to(27.7,0.11);
    \draw[ultra thin,red,->-=1,>=stealth](29.77,0.1)to(29.77+.01,0.11);
    \draw[ultra thin,Cyan,->-=1,>=stealth](30.8+.02,0.26)to(30.72+.13,0.3);
    \draw[ultra thin,Cyan,->-=1,>=stealth](31.5+.10,0.26)to(31.5+.13,0.3);
    \draw[ultra thin,Cyan,->-=1,>=stealth](33.24+.10,0.266)to(33.22+.13,0.279);
     \draw[ultra thin,Cyan,->-=1,>=stealth](34.263+.10,0.266)to(34.24+.13,0.279);

   \draw[]
   (1.2,-3)\dpole (-0,-2)node{$P_1$}  (1.2,-4.6)node[rotate=90]{$\cdots$}  (1.2,-6)\dpole (-0.,-5.5)node{$P_p$};

    \end{tikzpicture}
	\caption{Alternative generators for $\SBr(\dsp)$,
        where $\xi_{-j}=\varepsilon_{2j-1}$, $\,\xi_j=\varepsilon_{2j}$, $\zeta_k=\varepsilon_{2g+k}$
        for $1\leq j\leq g$ and $1\leq k\leq b-1$}.
	\label{fig:QZ's}
\end{figure}

In the rest of this section, we assume that
\begin{itemize}
  \item $\aleph\geq 5$ or
  \item $\aleph=4$ and $2g+b+p-1\leq 2$ (i.e., either $g=1,b=1,p=0$, or $g=0,b+p\leq 3$).
\end{itemize}

As in \cite{QZ3}, we define
\begin{equation}\label{eq:x and y}
    x:=\sigma_1^{\iv{\sigma_2}}=\sigma_2^{\sigma_1}\text{ and }y:=\sigma_3^{\iv{\sigma_2}}=\sigma_2^{\sigma_3}.
\end{equation}

\begin{proposition}[{\cite[Theorem~4.1]{QZ3}}]\label{prop:pre for sob}
    Suppose that either $\aleph\geq 5$, or $\aleph=4$ and $2g+b+p-1\leq 2$. The group $\BT(\dsp)$ admits the following presentation.
    \begin{itemize}
		\item Generators: $\sigma_i$, $1\leq i\leq \aleph-1$, $\tau_r$, $1\leq r\leq 2g+b+p-1$.
		\item Relations: for $1\leq i,j\leq \aleph-1$ and $1\leq r,s\leq 2g+b+p-1$,
		\begin{align}
		&\Co(\sigma_i,\sigma_j)\label{nr:01}&&  \text{if $|i-j|>1$}\\
		&\Br(\sigma_i,\sigma_{j})\label{nr:02}&&\text{if $|i-j|=1$}\\
		&\Co(\tau_r,\sigma_i)\label{nr:03}&&\text{if $i>2$}\\
		&\Br(\tau_r,x)\label{nr:04}&&\\
		&\Br(\tau_r,y)\label{nr:05}&&\\
		&\Co({\tau_r}^y,{\tau_s}^x)\label{nr:06}&&\text{if $s<r$ and $s\notin\dgen$}\\
		&\Co({\tau_r}^{\iv{y}},{\tau_s}^x)\label{nr:07}&&\text{if $s<r$ and $s\in\dgen$}
		\end{align}
	\end{itemize}
\end{proposition}

\begin{figure}[htpb]\centering
\begin{tikzpicture}[xscale=.3,yscale=.3]
	\draw[purple](10,-7)to(14,-9)(24,-9)to(20,-9);\draw[dashed,purple](14,-9)to(24,-9);
	\draw[purple](12,-8.8)node{\tiny{$\sigma_1$}} (22,-9.8)node{\tiny{$\sigma_{\aleph-1}$}};
	\draw[Cyan]plot [smooth,tension=.8] coordinates {(10,-7)(1.5,-7.6)(-5,-2.5)(1,3)(13,4.5)(30,3.5)(32,-2.5)(14,-9)};
	\draw[Cyan]plot [smooth,tension=.8] coordinates
        {(10,-7)(2.1,-4.4)(-2,1)(13,4)(30,3)(29.5,-3)(14,-9)};
    \draw[Cyan](-2.8,-3.9)node{\tiny{$\nu_{b+p-1}$}};
    \draw[Cyan](-3,-2.7)node[rotate=75]{\tiny{$\cdots$}};
    \draw[Cyan](-2.8,-1.6)node{\tiny{$\nu_{b}$}};

	\draw[ForestGreen]plot [smooth,tension=.8] coordinates {(10,-7)(0,1.4)(28,3)(28,-3)(14,-9)};
	\draw[ForestGreen]plot [smooth,tension=.8] coordinates {(10,-7)(6,1)(28,2.5)(26,-3)(14,-9)};
	\draw[ForestGreen](4.5,1.7)node{\tiny{$\nu_{b-1}\cdots\nu_1$}};

	\draw[red]plot [smooth,tension=.8] coordinates {(14,-9)(21.5,-5)(28,1.3)(15,2.2)(9,-.8)(10,-7)};
	\draw[red]plot [smooth,tension=.5] coordinates {(14,-9)(26,-1)(24,1.5)(10,-7)};
	\draw[red](9.6,-2)node{\tiny{$\omega_g$}} (20.5,-1)node{\tiny{$\omega_1$}};
	\draw[red](10,-7)to(8+2+4-.5,0);
	\draw[red,dashed](8+2+4-.5,0).. controls +(60:2) and +(180:3) ..(20,5);
	\draw[red](14,-9).. controls +(7:23) and +(2:23) ..(21,5);
	\draw[red](10,-7).. controls +(35:5) and +(-100:2) ..(20+2+4-.3,0.3);

	\draw[red,dashed](20+2+4-.3,0.3).. controls +(80:3) and +(180:1) ..(32,5);
	\draw[red](14,-9).. controls +(3:30) and +(0:4) ..(33,5);	

    \draw[red](11.5,-2)node{\tiny{$\omega_{-g}$}} (23,-1)node{\tiny{$\omega_{-1}$}};
	\foreach \j in {1,2.5} {\draw[thick](8*\j-2+4+.5,0).. controls +(45:1) and +(135:1) ..(8*\j+2+4-.5,0);\draw[very thick](8*\j-2+4+.3,0.3).. controls +(-60:1) and +(-120:1) ..(8*\j+2+4-.3,0.3);}
	\draw[very thick] (-3,5)to(38,5)(38,-11)   (-3,-11)to(38,-11)(-3,5)to[bend right=90](-3,-11);
	\draw[very thick,fill=gray!14] (38,-3) ellipse (1 and 8)node{\tiny{$\partial_{1}$}};
	\draw[very thick,fill=gray!14](7,0)circle(.7)node{\tiny{$\partial_2$}};
	\draw[very thick,fill=gray!14](1.5,0)circle(.7)node{\tiny{$\partial_{b}$}};
	\node at(3.5,0) {$\cdots$};
	\node at(18,0) {$\cdots$};
	\foreach \x/\y in {14/-9,10/-7,24/-9,20/-9}{\draw(\x,\y)\ww;}
	\draw (10,-7)node[below]{$Z_{1}$}
    	(14,-9)node[below]{$Z_2$}
    	(24,-9)node[below]{$Z_{\aleph}$};

   \draw[]
   (1.2,-3)\dpole (-0,-2)node{$P_1$}  (1.2,-4.6)node[rotate=90]{$\cdots$}  (1.2,-6)\dpole (-0.,-5.5)node{$P_p$};

	\end{tikzpicture}
    \caption{Generators for $\BT(\dsp)$,
        where $\omega_{-j}=\tau_{2j-1}$, $\omega_j=\tau_{2j}$, $\nu_k=\tau_{2g+k}$
        for $1\leq j\leq g$ and $1\leq k\leq b-1$}.
	\label{fig:BT}
\end{figure}
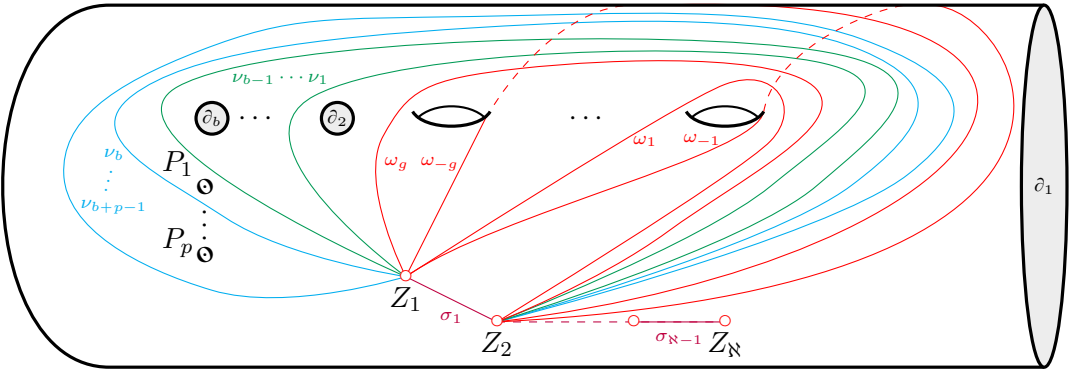

By a direct picture-chasing argument (cf. also \cite[Lemma~4.5]{QZ3}), the conjugation of $\BT(\dsp)$ by $H$ .is given by the following formulas:
\begin{align}
{\sigma_i}^{\iv{\varepsilon_t}}\label{eq:ss02}&=\begin{cases}
\tau_t&\text{\qquad\ if $i=1$,}\\
\sigma_i&\text{\qquad\ if $i\neq 1$,}
\end{cases}\\
{\tau_r}^{\iv{\varepsilon_t}}\label{eq:ss03}&=\begin{cases}
\tx^{\sigma_2\kg\tau_t}&\text{if $r=t$,}\\
{\tau_r}^{x\kg\tau_t\kg\sigma_2\kg\tau_t}&\text{if $r\leq t$ and $r\notin\dgen$,}\\
{\tau_r}^{x\kg\iv{\tau_t}\kg\sigma_2\kg\tau_t}&\text{if $r\leq t$ and $r\in\dgen$,}\\
{\tau_r}^{\iv{x}\kg\iv{\tau_t}\kg\sigma_2\kg\tau_t}&\text{if $r\geq t$ and $t\notin\dgen$,}\\
{\tau_r}^{\iv{x}\kg\tau_t\kg\sigma_2\kg\tau_t}&\text{if $r\geq t$ and $t\in\dgen$,}
\end{cases}
\end{align}
and
\begin{align}
{\sigma_i}^{{\varepsilon_t}}\label{eq:ss05}&=\begin{cases}
{\tau_t}^{\iv{\sigma_1}}&\qquad\ \text{if $i=1$,}\\
\sigma_i&\qquad\ \text{if $i\neq 1$,}
\end{cases}\\
{\tau_r}^{{\varepsilon_t}}&=\label{eq:ss06}\begin{cases}
\sigma_1&\text{if $r=t$,}\\
{\tau_r}^{\iv{\sigma_2}\kg\iv{\tau_t}\kg\iv{\sigma_1}\kg\iv{\sigma_2}}&\text{if $r\leq t$ and $r\notin\dgen$,}\\
{\tau_r}^{\sigma_2\kg\iv{\tau_t}\kg\iv{\sigma_1}\kg\iv{\sigma_2}}&\text{if $r\leq t$ and $r\in\dgen$,}\\
{\tau_r}^{\sigma_2\kg\tau_t\kg\iv{\sigma_1}\kg\iv{\sigma_2}}&\text{if $r\geq t$ and $t\notin\dgen$,}\\
{\tau_r}^{\iv{\sigma_2}\kg\tau_t\kg\iv{\sigma_1}\kg\iv{\sigma_2}}&\text{if $r\geq t$ and $t\in\dgen$,}
\end{cases}
\end{align}

\section{Decorated surface with vortices}\label{part:B}
Recall that $\dsp=(\qm,\P,\Tri)$ is a decorated surface with punctures such that $\partial\qm\neq\emptyset$.

\subsection{Vortices}\label{subsec:vortex}

By adding an extra $\z2$-symmetry to the punctures, we refer to them as vortices. We will then use the notation $\qmv$ and refer to it as a \emph{marked surface with vortices}.

Let
\begin{gather}
    \LL=\< \Bt{s}^4 \mid s\in\DVs \>
\end{gather}
be the subgroup of $\MCG(\dsp)$ generated by all the fourth powers of braid twists along collision paths in $\DVs$. It follows from formula~\eqref{eq:Psi} that $\LL$ is a normal subgroup of $\MCG(\dsp)$. Note that $\LL\subseteq\SBr(\dsp)$.

\begin{definition}\label{def:BTv}
The \emph{mapping class group} $\MCG(\dsv)$ is defined as $$\MCG(\dsv)=\MCG(\dsp)/\LL.$$
Since $\LL\subseteq\SBr(\dsp)$, the forgetful map \eqref{eq:forget1} induces a group homomorphism
\begin{gather}\label{eq:forget2}
    \FM\colon\MCG(\dsv)\to\MCG(\qmp),
\end{gather}
whose kernel $\ker \FM$ is defined as the \emph{surface braid group} $\SBr(\dsv)$ of $\dsv$. The \emph{braid twist group} $\BT(\dsv)$ of $\dsv$ is defined to be
\begin{gather}
    \BT(\dsv)=\BT(\dsp)/(\LL\cap\BT(\dsp)).
\end{gather}
\end{definition}

\begin{proposition}\label{pp:non-trivial}
    Let $\eta_1,\eta_2\in\CA(\dsp)$ and let $\delta_1,\delta_2$ be the completions of $s_1,s_2\in\DVs$, respectively, as shown in Figure~\ref{eq:Co}. Then the following relation holds in $\MCG(\dsp)$:
\begin{equation}\label{eq:relation}
    \Bt{\eta_1}\cdot\Bt{\eta_2}\cdot \mathrm{L}^2_{\delta_2}
    =\Bt{\eta_2}\cdot\Bt{\eta_1}\cdot \mathrm{L}^2_{\delta_1}.
\end{equation}

    \begin{figure}[htpb]
        \centering
        \begin{tikzpicture}[xscale=.3,yscale=.2]
    \begin{scope}[shift={(0,0)},scale=1.8]
    \draw[red] (-3,0)arc(180:0:3) (0,3)node[above]{$\eta_1$};
    \draw[red] (-3,0)arc(180:360:3) (0,-3)node[below]{$\eta_2$};;
    \draw[\hhd,very thick](-3,0).. controls +(45:7) and +(-45:7) ..(-3,0)
    (3,0).. controls +(135:7) and +(-135:7) ..(3,0)
    (-1,-1.3)node[below]{$\delta_1$}
    (1,-1.3)node[below]{$\delta_2$};
    \draw(-3,0)\ww(3,0)\ww(0,0)\dpole;
    \end{scope}
    \begin{scope}[shift={(14,0)},scale=1.8]
    \draw[red] (-3,0)arc(180:0:3) (0,3)node[above]{$\eta_1$};
    \draw[red] (-3,0)arc(180:360:3) (0,-3)node[below]{$\eta_2$};;
    \draw[gray,ultra thick](-3,0)tonode[above]{$s_1$}(0,0)
    (3,0)to node[above]{$s_2$}(0,0);
    \draw(-3,0)\ww(3,0)\ww(0,0)\dpole;
    \end{scope}
    \end{tikzpicture}
        \caption{}
        \label{eq:Co}
    \end{figure}

\end{proposition}

\begin{proof}
Write $s_i=\Bt{s_i}$ for short. The braid relations imply
\[
    s_1^{-m}s_2s_1=s_2s_1s_2^{-m},\quad s_2^{-m}s_1s_2=s_1s_2s_1^{-m},\quad \forall m\in\ZZ.
\]
Moreover, we have
\begin{gather*}
  \eta_1=\Bt{s_1}(s_2)\quad\Longrightarrow\quad
    \Bt{\eta_1}=s_1s_2s_1^{-1} =s_2^{-1}s_1s_2
\end{gather*}
and by Lemma~\ref{lem:non-trivial}, $\mathrm{L}_{\delta_2}=s_2^2$.
Thus, we have
\[\begin{array}{rcl}
    \Bt{\eta_1}\cdot\Bt{\eta_2}\cdot \mathrm{L}^2_{\delta_2}&=&
        (s_2s_1s_2^{-1})\cdot(s_1^{-1}s_2s_1)\cdot s_2^4\\
    &=& s_1s_2s_1 \cdot( s_1^{-3}s_2s_1)\cdot s_2^4\\
    &=& s_1s_2s_1 \cdot( s_2s_1s_2^{-3})\cdot s_2^4\\
    &=& s_1s_2s_1s_2s_1s_2.
\end{array}\]
Similarly, we have $\Bt{\eta_2}=s_2s_1s_2^{-1}=s_1^{-1}s_2s_1$, $\mathrm{L}_{\delta_1}=s_1^2$
and so
\[
    \Bt{\eta_2}\cdot\Bt{\eta_1}\cdot \mathrm{L}^2_{\delta_2}=s_1s_2s_1s_2s_1s_2.
\]
Noticing $(s_2s_1s_2)^2=(s_1s_2s_1)^2$, \eqref{eq:relation} follows.
\end{proof}

A direct corollary is the following.

\begin{corollary}
Let $\eta_1,\eta_2$ be two closed arcs that share both endpoints
and enclose a vortex as shown in the left picture of Figure~\ref{eq:Co}.
Then, in $\BT(\dsv)$, we have the relation $\Crel(\Bt{\eta_1},\Bt{\eta_2})$.
\end{corollary}

\begin{remark}
Another way to understand \eqref{eq:relation} is that both sides are equivalent to
the same mapping class, uniquely determined (by Alexander’s trick)
by the image of the gray arcs in Figure~\ref{fig:Alex}.

\begin{figure}[htpb]\centering
\begin{tikzpicture}[xscale=.3,yscale=.2]
\begin{scope}[shift={(0,0)},scale=1.8]
\draw[dashed,gray, thick](0,0)ellipse(5 and 7.5);
\draw[gray,ultra thick](5,0)to(3,0)(-5,0)to(-3,0);
  \draw[red!50] (-3,0)arc(180:0:3);
  \draw[red!50] (-3,0)arc(180:360:3);
  \draw[blue!25,very thick](-3,0).. controls +(45:7) and +(-45:7) ..(-3,0)
        (3,0).. controls +(135:7) and +(-135:7) ..(3,0);
  \draw(-3,0)\ww(3,0)\ww(0,0)\vot;
\end{scope}
\draw[->,line width=3pt, gray,>=stealth](10,0)to
    node[above]{$\Bt{\eta_1}\cdot\Bt{\eta_2}\cdot \mathrm{L}^2_{\delta_2}$}
    node[below]{$\Bt{\eta_2}\cdot\Bt{\eta_1}\cdot \mathrm{L}^2_{\delta_1}$}(16,0);

\begin{scope}[shift={(26,0)},scale=1.8,yscale=1.5]
\draw[dashed,gray, thick](0,0)circle(5);
\draw[gray,ultra thick](5,0)arc(0:180:4.5)arc(180:360:3.5)
    (-5,0)arc(-180:0:4.5)arc(0:180:3.5);
\end{scope}

\begin{scope}[shift={(26,0)},scale=1.8]
  \draw[red!50] (-3,0)arc(180:0:3);
  \draw[red!50] (-3,0)arc(180:360:3);
  \draw[blue!25,very thick](-3,0).. controls +(45:7) and +(-45:7) ..(-3,0)
        (3,0).. controls +(135:7) and +(-135:7) ..(3,0);
  \draw(-3,0)\ww(3,0)\ww(0,0)\vot;
\end{scope}
\end{tikzpicture}
\caption{}
\label{fig:Alex}
\end{figure}
Therefore, Proposition~\ref{pp:non-trivial} remains valid when the vortex is replaced by an arbitrary object.
\end{remark}

\begin{definition}
    The set of closed arcs on $\dsv$ is defined to be $$\CA(\dsv)=\CA(\dsp)/\LL.$$
\end{definition}

\subsection{Finite presentations: vortex version}\label{sec:presentation}

Recall that $\LL=\< \Bt{s}^4 \mid s\in\DVs \>$ is a normal subgroup of $\SBr(\dsp)$.
Let $H$ be the group defined in \eqref{eq:H}, and let $N=\BT(\dsp)$ for convenience  in the proofs.

\begin{lemma}\label{lem:gen}
    The squares of $L$-twists $\delta_{2g+b}^2,\cdots,\delta_{2g+b+p-1}^2$ normally generate $\LL$.
\end{lemma}

\begin{proof}
For any $1\leq i\leq p$, let $\delta_i':=\delta_{2g+b+i-1}$, and let $s_i$ be the \cp in $\DVs$ such that $\delta_i'=s_i^\circlearrowleft$.

Consider any \cp $s\in\DVs$ with vortex $V_j$ as an endpoint. Let $Z$ and $Z_i$ denote the decoration endpoints of $s$ and $s_j$, respectively (they may coincide).
Then there exists an element $f$ in $\SBr(\dsp)$ that moves $Z_i$ to $Z$ the path composed of $s_i$ followed by $s$, bypassing $V$ on either side, and moves $Z$ to $Z_i$ if they are different, so that $s=f(s_{i})$. Thus, by \eqref{eq:Psi},
we have $B_s=B_{s_i}^{\iv{f}}$, and hence $B_s^4={\left( \mathrm{L}_{\delta'_i}^2 \right)}^{\iv{f}}$.
\end{proof}

\begin{definition}
Let $\Cop:=\LL\cap \BT(\dsp)$, and define two more groups.
\begin{itemize}
  \item Define $\CON$ to be the normal subgroup of $\BT(\dsp)$ generated by
    \begin{gather}\label{eq:gen-COR}
        \COR:=[\tau_r,\tau_{r-1}]=\tau_r\tau_{r-1} \underline{\tau_r\tau_{r-1}}
    \end{gather}
    for $2g+b\leq r\leq 2g+b+p-1$. It is clear that all of the following elements belong to $\CON$:
    $$[\tau_{r-1},\tau_{r}]=\iv{\COR},\ [\tau_{r-1},\iv{\tau_r}]=\COR^{\tau_r},\ [\tau_{r},\iv{\tau_{r-1}}]=\iv{\COR}^{\tau_{r-1}}.$$
  \item Define $\COL$ to be the normal subgroup of $\LL$ generated by
  \begin{gather}\label{eq:gen-COL}
    [\sbt,\delta_r^{2}],
  \end{gather}
  for any $\sbt\in \SBr(\dsp)$ and any $2g+b\leq r\leq 2g+b+p-1$.
\end{itemize}
\end{definition}

We will show that $\CON=\Cop=\COL$; that is, the definitions above provide normal generators of $\Cop$ inside both $\BT(\sop)$ and $\LL$, respectively.

\begin{lemma}\label{lem:coh}
The group $\CON$ is a normal subgroup of $\SBr(\dsp)$.
\end{lemma}

\begin{proof}
    By \eqref{eq:ss03} and \eqref{eq:ss06}, for any $1\leq t\leq 2g+b+p-1$ and $2g+b\leq r\leq 2g+b+p-1$, we have \begin{align}\label{eq:coh1}
    \COR^{\iv{\varepsilon_t}}=\begin{cases}
       \COR^{x\kg\tau_t\kg\sigma_2\kg\tau_t} & \text{if $r\leq t$,}\\
       \COR^{\iv{x}\kg\iv{\tau_t}\kg\sigma_2\kg\tau_t} & \text{if $t<r$ and $t\notin\dgen$,}\\
       \COR^{\iv{x}\kg{\tau_t}\kg\sigma_2\kg\tau_t} & \text{if $t<r$ and $t\in\dgen$,}\\
    \end{cases}
\end{align}
and
\begin{align}\label{eq:coh2}
    \COR^{\varepsilon_t}=\begin{cases}
       \COR^{\iv{\sigma_2}\kg\iv{\tau_t}\kg\iv{\sigma_1}\kg\iv{\sigma_2}} & \text{if $r\leq t$,}\\
       \COR^{{\sigma_2}\kg{\tau_t}\kg\iv{\sigma_1}\kg\iv{\sigma_2}} & \text{if $t<r$ and $t\notin\dgen$,}\\
       \COR^{\iv{\sigma_2}\kg{\tau_t}\kg\iv{\sigma_1}\kg\iv{\sigma_2}} & \text{if $t<r$ and $t\in\dgen$.}\\
    \end{cases}
\end{align}
As $\CON$ is a normal subgroup of $\BT(\dsp)$, both $\COR^{\varepsilon_t}$ and $\COR^{\iv{\varepsilon_t}}$ are in $\CON$.
Noting that $\SBr(\dsp)=NH$, and that $\varepsilon_t,1\leq t\leq 2g+b+p-1$ generate $H$,
we deduce that $\CON$ is a normal subgroup of $\SBr(\dsp)$.
\end{proof}

\begin{lemma}\label{lem:quo}
For any $\sbt\in \SBr(\dsp)$ and any $2g+b\leq r\leq 2g+b+p-1$, we have
    \begin{equation}\label{eq:co}
        [\sbt,\delta_r^{2}]\in\CON
    \end{equation}
and hence $\COL\leq\CON$.
\end{lemma}

\begin{proof}
\eqref{eq:co} follows from App.~\ref{app:A}.
In combination with \Cref{lem:coh}, this completes the proof.
\end{proof}

\begin{proposition}\label{prop:==}
$\CON=\Cop=\COL$.
\end{proposition}

\begin{proof}
    For any $2g+b\leq r\leq 2g+b+p-1$, the arcs $\eta_1=\tau_{r-1}$ and $\eta_2=\tau_r$ are in the configuration illustrated in \Cref{eq:Co}; see also \Cref{fig:BT}.  Hence, by \Cref{pp:non-trivial}, we have
    \[
    \COR=B_{\eta_2}B_{\eta_1}\iv{B_{\eta_2}}\iv{B_{\eta_1}}
    =(\iv{B_{\eta_2}}\iv{B_{\eta_1}}B_{\eta_2}B_{\eta_1})^{\iv{B_{\eta_2}}\iv{B_{\eta_1}}}
    =(\mathrm{L}^2_{\delta_2}\iv{\mathrm{L}^2_{\delta_1}})^{\iv{B_{\eta_2}}\iv{B_{\eta_1}}}
    \;\in\LL
\]
    So $\COR\in\Cop$ and thus $\CON\le\Cop$. In combination with \Cref{lem:quo}, we only need to show $\Cop\le\COL$.

Take an arbitrary element $1\neq\lb\in\Cop$. By \Cref{lem:gen}, $\lb$ can be written as a product
    $\lb=\displaystyle\prod_{i=1}^s (\theta_i^2)^{\sbt_i},$
    for some $s\ge0$, $\theta_i\in\bigcup_{2g+b\leq r\leq 2g+b+p-1}\{\delta_{r},\iv{\delta_{r}}\}$, and $\sbt_i\in\SBr(\dsp)$, $1\leq i\leq s$. We prove that $\lb\in\COL$ by induction on $s$, starting with the trivial case $s=0$, where $\lb=1$.

Since $\SBr(\dsp)=NH$, there are $n_i\in N$ and $h_i\in H$ such that $\sbt_i=n_ih_i$. Since $N$ is a normal subgroup of $\SBr(\dsp)$, we have $\lb=\displaystyle n \cdot \prod_{i=1}^s (\theta_i^2)^{h_i} $ for some $n\in N$. This implies that $\displaystyle\prod_{i=1}^s (\theta_i^2)^{h_i}\in N$ as $\lb\in N$. Then we have
\[\displaystyle
    \prod_{i=1}^s (\theta_i^2)^{h_i}\in N\cap H\xlongequal{\eqref{eq:HH}}[H,H].
\]
Thus, $\displaystyle\prod_{i=1}^s {(\theta_i^2)^{h_i}}$ becomes the identity
in abelianization $\Ho{1}(\qmp)=H/[H,H]$ of $H$, which is freely generated by $\delta_{?}$'s. This implies that there exists $1<j\leq s$ such that $\theta_1=\iv{\theta_j}$. Consider the following expression of $\lb$
\[
    \lb=\left(\theta_1^2\right)^{\sbt_1}\left(\theta_j^2\right)^{\sbt_j}
        \left(\prod_{i=2}^{j-1}(\theta_i^2)^{\sbt_i\left(\theta_j^2\right)^{\sbt_j}}\right)\left(\prod_{i=j+1}^s(\theta_i^2)^{\sbt_i}\right),
\]
    where $\left(\theta_1^2\right)^{\sbt_1} (\theta_j^2)^{\sbt_j}=[\sbt_j\iv{\sbt_1},\theta_1^2]^{\sbt_j}\in\COL$ by inductive assumption. Let $\sbt'_i=\sbt_i(\theta_j^2)^{\sbt_j}$ for $2\leq i\leq j-1$ and $\sbt'_i=\sbt_i$ for $j+1\leq i\leq s$. Then the rest term of $\lb$ equals
    $$\lb'=\prod_{1\leq i\leq s,\ i\neq 1,j}\left(\theta_i^2\right)^{\sbt_i'}.$$
    By the induction hypothesis, $\lb'$ belongs to $\COL$, and hence so does $\lb$. This completes the proof.
\end{proof}

As a corollary, we obtain a finite presentation of $\BT(\sov)$.

\begin{theorem}\label{thm:pre for sov}
Suppose that either $\aleph\geq 5$, or $\aleph=4$ and $2g+b+p-1\leq 2$.
The braid twist group $\BT(\dsv)$ admits the following finite presentation (see Figure~\ref{fig:B's}):
    \begin{itemize}
	   \item Generators: $\sigma_1,\ldots,\sigma_{\aleph-1}$ and $\tau_1,\ldots,\tau_{2g+b+p-1}$.
        \item Relations: \eqref{nr:01}--\eqref{nr:07}, together with
        \begin{align}
        &\Co(\tau_r,\tau_{r-1})&&\text{if $2g+b\leq r\leq 2g+b+p-1$.}\label{nr:sov}
        \end{align}
    \end{itemize}
\end{theorem}

\begin{proof}
Compared to the presentation of $N=\BT(\dsp)$ given in \Cref{prop:pre for sob}, this follows directly from the equality $\CON=\Cop$ in \Cref{prop:==}.
\end{proof}

The first homology group $\Ho{1}(\qmv)$ of $\qmv$ is defined as the quotient group of $\Ho{1}(\qmp)$
by the image $\AJ^\sun(\LL)\subseteq\Ho{1}(\qmp)$.
Since $$\Ho{1}(\qmp)=\Ho{1}(\qm)\oplus\Ho{1}(\sun),$$
and the elements $[\delta_r]$, $2g+b\le r\le 2g+b+p-1$, form a set of generators of $\Ho{1}(\sun)\cong\ZZ^{\oplus\sun}$,
we obtain
$$\Ho{1}(\qmv)=\Ho{1}(\qm)\oplus\Ho{1}(\vortex),$$
where $\Ho{1}(\vortex)\cong\ZZP$, generated by $\AJ^\sun([\delta_r])$ for $2g+b\le r\le 2g+b+p-1$.

We now establish the vortex version of the short exact sequence \eqref{eq:SES-b} in \Cref{thm:QZ+}.

\begin{theorem}\label{thm:BT+}
Suppose that $\aleph\geq 3$.
Then $\BT(\dsv)=\ker\AJ^\vortex$; equivalently, there is a short exact sequence
\begin{gather}\label{eq:SES-v}
    1 \to \BT(\dsv) \to \SBr(\dsv) \xrightarrow{\AJ^{\vortex}} \Ho{1}(\qmv) \to 1.
\end{gather}
\end{theorem}

\begin{proof}
Let $L^2=\LL$ and $G=\SBr(\dsp)$. By the isomorphism theorems of groups, we have
\[
   \Ho{1}(\qmv)=\frac{\Ho{1}(\qmp)}{\AJ^\sun(L^2)}=\frac{G/N}{L^2N/N}=\frac{G}{L^2N}=\frac{G/L^2}{L^2N/L^2}
   =\frac{G/L^2}{N/(L^2\cap N)}=\frac{\SBr(\dsv)}{\BT(\dsv)}.\qedhere
\]
\end{proof}

\subsection{A variation of the presentation}

In this section, we assume $\aleph\geq 2(2g+b+p-1)+b$. We fix $b$ positive integers $m_1,\dots,m_b$ (not necessarily unique) such that
$$\sum_{i=1}^b m_i=\aleph-2(2g+b+p-1).$$
We consider a new set of generators $\sigma_1$, $x_i$ for $1\leq i\leq 2g+b+p-2+\sum_{i=2}^b m_i$, $y_i$ for $1\leq i\leq 2g+b+p-2+m_1$, and $\newtau{r}$ for $1\leq r\leq 2g+b+p-1$, shown in \Cref{fig:new-generator}. Here, $x_1=x$, $y_1=y$, and $\newtau{r}$ is defined by conjugating $\tau_r$ as follows.
$$\newtau{r}={\tau_r}^{\iv{x_1}\kg\cdots\kg\iv{x_{\dianshu{r}}}\kg y_1\kg\cdots\kg y_{\paiwei{r}}},$$
where
$$\dianshu{r}=\begin{cases}
    r-1 & \text{if $1\leq r\leq 2g$,}\\
    r-1+\sum_{i=2}^{r-2g+1}m_i & \text{if $2g+1\leq r\leq 2g+b-1$,}\\
    r-1+\sum_{i=2}^{b}m_i & \text{if $2g+b\leq r\leq 2g+b+p-1$.}
\end{cases}$$
and
$$\paiwei{r}=\begin{cases}
    2g+b+p-2+m_1 & \text{if $r=1$ and $g\neq 0$,}\\
    2g+b+p-1-\frac{r+1}{2} & \text{if $r\in\dgen$ and $r\neq 1$,}\\
    \frac{r}{2}-1 & \text{if $r\in\dgene$,}\\
    r-1-g & \text{if $2g+1\leq r\leq 2g+b+p-1$.}
\end{cases}$$

\begin{figure}[htpb]\centering
    \begin{tikzpicture}[scale=.7]
    \draw[red,thick,dashed]  (-7,0) to (-6,0)  (-6,0) to (-5,0) to (-4,0) to (-3,0)  (10,0) to (9,0)    (9,0) to (8,0)     (7,0) to (6,0)     (5,0) to (4,0)  to  (3,0) to (2,0);
    \draw[red,thick]      (-8,0) to (-7,0)         
    (-3,0) to (-2,0)
    (-2,0) to (-1,0)
    		(-1,0) to (0,0)
    		(0,0) to (1,0)
    		(1,0) to (2,0)
    		(5,0) to (6,0) to
    		(7,0) to (8,0);
    \begin{knot}
    \strand[red,thick]plot[smooth,tension=.9]
    coordinates{
    	(-8,0) (-1,13) (7,0)};
    \strand[red,thick]plot [smooth,tension=.9]
    coordinates{
    	(-8,0) (-1,11.5) (6,0)};
    \strand[red,thick]plot [smooth,tension=.9]
    coordinates{
       (-7,0) (-1,10) (5,0)};
    \strand[red,thick]plot [smooth,tension=.9]
    coordinates{
       (-6,0) (-1,7) (4,0)};
    \strand[red,thick]plot [smooth,tension=.8]
    coordinates{
    (-3,0) (-0.5,2.5) (2,0)};
    \strand[red,thick]plot [smooth,tension=.8]
    coordinates{
        (-1,0) (0,1) (1,0)};
    \strand[violet,thick]plot [smooth,tension=.8]
    coordinates{
    	(-4,0) (2,6) (8,0)};
    \strand[violet,thick]plot[smooth,tension=.8]
    coordinates{
    (-2,0) (3.5,4.5) (9,0)};
    \strand[violet,thick]plot[smooth,tension=.8]
    coordinates{
    	(0,0) (5,2.9) (10,0)};
    \end{knot}
    \draw
    (-1,12.3)node{\Yinyang}
    (-0.4,12.3)node{$V_p$};
    \draw
    (-1,10.7)node{\Yinyang}
    (-0.15,10.7)node{$V_{p-1}$};
    \draw
    (-1,9.3)node{\Yinyang}
    (-0.15,9.2)node{$V_{p-2}$};
    \draw
    (-1,7.5)node{\Yinyang}
    (-0.15,7.5)node{$V_{1}$};
    \draw(-1,8.5)node{$\vdots$};
    \filldraw[fill=gray!13](-1,6) circle (0.25)
    (-0.8,5.9)node[right]{$\partial_p$};
    \draw(-1,4.8)node{$\vdots$};
    \filldraw[fill=gray!13](-1,3.2) circle (0.25)(-0.8,3.1)node[right]{$\partial_2$};
    \draw[violet]  (9.5,1)node{$\tau'_1$};
    \draw[violet](6.5,3.3)node{$\tau'_3$};
    		\draw[violet]
    		(1.8,6.4)node{$\tau'_{2g-1}$};
    \draw[red]
    (-2.5,-0.4)node{$x_3$}
    (-1.5,-0.4)node{$x_2$}
    (-0.5,-0.4)node{$x_1$}
    (0.5,-0.4)node{$\sigma_1$}
    (1.5,-0.4)node{$y_1$}
    ;
    \foreach \x/\y in {-8/0,-7/0,-6/0,-4/0,-3/0,-2/0,-1/0,0/0,10/0,9/0,8/0,7/0,6/0,5/0,4/0,2/0,1/0}{\draw(\x,\y)\ww;}
    \end{tikzpicture}
\caption{Alternative generators of $\BT(\dsv)$}
\label{fig:new-generator}
\end{figure}

There is an alternative presentation of $\BT(\dsv)$ as follows.

\begin{proposition}\label{prop:alt pre}
    Suppose that $\aleph\geq 2(2g+b+p-1)$. The braid twist group $\BT(\dsv)$ has the following finite presentation
    \begin{itemize}
        \item Generators: $\sigma_1$, $x_i$ for $1\leq i\leq 2g+b+p-2+\sum_{i=2}^b (m_i-1)$, $y_i$ for $1\leq i\leq 2g+b+p-2+m_1-1$, and $\newtau{r}$ for $1\leq r\leq 2g+b+p-1$.
        \item Relations:
        \[\begin{array}{cl}
        \Co(a,b)&\text{if $a,b$ are disjoint,}\\
        \Br(a,b)&\text{if $a,b$ are disjoint except sharing a common endpoint,}\\
        \on{Tr}(a,b,c)&\text{if $a,b,c$ are disjoint except sharing a common endpoint, and they}\\
        &\text{are in clockwise order at that point, see the left picture of \Cref{fig:tri rect},}\\
        \on{Rec}(a,b,c,d)&\text{if $a,b,c,d$ form a rectangle in the anti-clockwise order with}\\
        &\text{exactly one puncture in the interior, see the right picture of \Cref{fig:tri rect}.}
    \end{array}\]
    \end{itemize}
\end{proposition}

\begin{figure}[htpb]\centering
\begin{tikzpicture}[scale=1]
\begin{scope}[shift={(0,0)},rotate=60]
\foreach \j in {1,2,3} {\draw[red,thick](0,0)to(90+120*\j:2) (90+120*\j:2)\ww;}
\draw[red](0,0)\ww (70:1)node{$a$}(190:1)node{$b$}(-20:1)node{$c$};
\end{scope}
\begin{scope}[shift={(6,-.3)},xscale=2,scale=.6]
\draw[red,thick](2,2) \ww to (2,-2)\ww to (-2,-2) \ww to (-2,2) \ww to (2,2)\ww;
\draw[red] (180:2.2)node{$a$}(-90:2.4)node{$b$}(90:2.4)node{$d$}(0:2.2)node{$c$};
\draw(0,0)node{\sun};
\end{scope}
\end{tikzpicture}
\caption{Triangle/Rectangle relations}
    \label{fig:tri rect}
\end{figure}
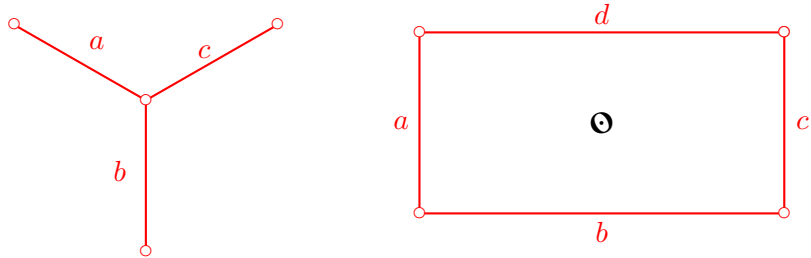

\begin{proof}
    In the case $\aleph\leq 3$, this presentation can be obtained by an easy direct calculation, similarly as in \cite{QZ3}. Indeed, the only subcase different from \cite{QZ3} here is that $g=0$, $b=1$, $p=1$, and $\aleph=2$ or $3$.

    For the case $\aleph\geq 4$, since $2g+b+p-1\leq \frac{\aleph}{2}$, the assumption in \Cref{thm:pre for sov} holds. Hence the presentation of $\BT(\dsv)$ in \Cref{thm:pre for sov} applies.

    It is easy to see that the generators $\sigma_1$, $x_i$, $y_i$, and $\newtau{r}$ form a generating set of $\BT(\dsv)$, and the relations listed hold in $\BT(\dsv)$. So we only need to show these relations imply the relations in \Cref{thm:pre for sov}. The relations \eqref{nr:01}–\eqref{nr:07} are already deduced in \cite{QZ3} (by using Lemmas~5.7 and 5.8 therein). So we only need to check the new relation \eqref{nr:sov}. For any $2g+b\leq r\leq 2g+b+p-1$, conjugating by the element $\iv{x_1}\kg\cdots\kg\iv{x_{r-2+d}}\kg y_1\kg\cdots\kg y_{r-g-2}$, where $d=\sum_{i=2}^b(m_i-1)$, this relation is equivalent to
    \begin{equation}\label{eq:newzhuan}
        \Co({\newtau{r}}^{x_{r-1+d}\kg \iv{y_{r-g-1}}},\newtau{r-1}).
    \end{equation}
    Because $x_{r-1+d}$, $\newtau{r-1}$, $y_{r-g-1}$ and $\newtau{r}$ form a rectangle in the anti-clockwise order with exactly one vortex $V_{r-(2g+b-1)}$ in its interior, see \Cref{fig:new-generator} (for the case $r=2g+b$), we have the relation $\on{Rec}(x_{r-1+d},\newtau{r-1},y_{r-g-1},\newtau{r})$, which implies \eqref{eq:newzhuan}.
\end{proof}

\section{Cluster braid groups: surface case}\label{part:C}

In this section, we recall the notions of cluster braid groups and cluster exchange groupoids as introduced in \cite[\S~2]{KQ2}. However, we avoid categories and only use surface combinatorics.

\subsection{Marked surfaces with vortices}

A marked surface with punctures (MSp) $\surfp$ is a punctured surface $\qmp=(\qm,\sun)$ together with a set $\M$ of marked points on $\partial\qm$, such that each connected component of $\partial\surf$ contains at least one marked point. We denote by $\surf=(\qm,\M)$ and $\surfp=(\surf,\sun)$.

Up to homeomorphism, $\surf$ is determined by the genus $g$, the number $b$ of boundary components, the number $p=\#\sun$ of punctures, and the integer partition of $m=\#\M$ into $b$ parts describing the numbers of marked points on the boundary components of $\surf$. The rank of $\smp$ is defined to be

\begin{gather}\label{eq:n}
n=6g+3p+3b+m-6.
\end{gather}

The \emph{mapping class group} $\MCG(\smp)$ is defined to be the group of (isotopy classes of) homeomorphisms of $\surf$ that preserve the sets $\M$ and $\sun$, respectively, setwise, respectively.

A vortex is a puncture endowed with an extra $\ZZ_2$-symmetry (i.e., a choice of sign). A marked surface with vortex (MSv) $\smv=(\surf,\M,\P)$ is a MSp with all punctures turned to vortices.

\begin{definition}[{\cite{BQ}}]
    The \emph{mapping class group} of $\smv$ is defined to as the semi-direct product $$\MCG(\smv)=\MCG(\smp)\rtimes\ZZP.$$
\end{definition}

We will identify $\ZZ_2$ with $\{\pm1\}$.

\subsection{Signed triangulations}

\begin{definition}
An \emph{open arc} $\gamma$ is (the isotopy class of) a curve on $\smp$ with interior lying in $\surf^\circ$, and endpoints in $\M\cup\P$, which is \emph{simple} (does not intersect itself in $\surf^\circ$) and \emph{essential} (not homotopic to a constant arc or a boundary arc, i.e., an arc lying entirely in $\partial\surf$).

An \emph{(ideal) triangulation} $\RT$ of $\smp$ is a maximal collection of open arcs on $\smp$ that are compatible, that is, they do not intersect in $\surf^\circ$. See \Cref{fig:type IV} for possible puzzle pieces in a triangulation.
\end{definition}

\begin{figure}[ht]\centering
    \includegraphics[width=\textwidth]{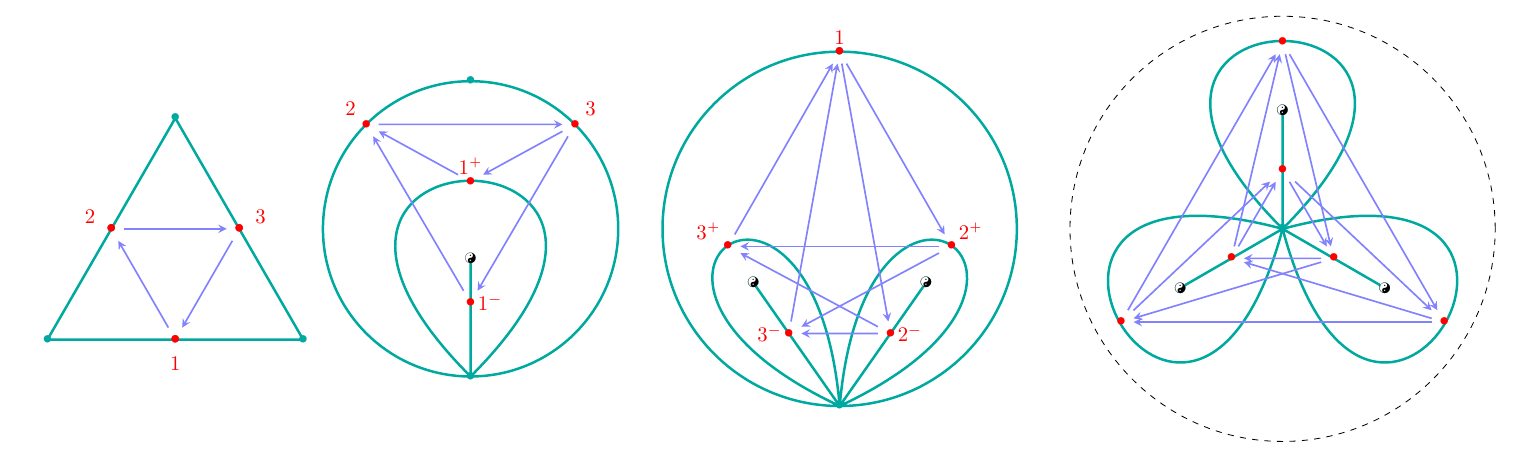}
\caption{The types I, II, III, IV (decorated) puzzle pieces}
\label{fig:type IV}
\end{figure}

It is known that any triangulation $\RT$ contains exactly $n$ arcs. Equivalently, a triangulation $\RT$ is a collection of open arcs that divides $\surfp$ into triangles. There are two types of triangles: ordinary triangles with three distinct edges, and self-folded triangles where two edges coincide. In the self-folded case, there is a loop edge and a self-folded edge enclosing a puncture, see Figure~\ref{fig:pm}.

\begin{definition}
    The \emph{(forward=backward) flip} of a triangulation $\RT$ of $\smp$ at a non-self-folded arc $\gamma$ is the unique triangulation $\mu_{\gamma}(\RT)$ such that $\RT\setminus \mu_{\gamma}(\RT)=\{\gamma\}$. Indeed, $\mu_{\gamma}(\RT)=(\T\setminus\{\gamma\})\cup\{\gamma^\sharp\}$, where $\gamma^\sharp$ is obtained from $\gamma$
    by moving its endpoints anticlockwise along the quadrilateral containing $\gamma$, as shown in the second row of Figure~\ref{fig:flip}.
\end{definition}

The \emph{unoriented exchange graph $\uEG(\smp)$ of triangulations} of $\smp$ is a graph whose vertices are the triangulations of $\smp$ and whose edges correspond to flips.

\begin{definition}\label{def:EGv}
A \emph{signed triangulation} of $\smp$ is a pair $\RTe=(\RT,\sign)$, where $\RT$ is a triangulation of $\smp$ and $\sign \in \ZZP$ is a sign assignment on vortices. The \emph{flip} of a signed triangulation $(\RT,\sign)$ at a non-self-folded arc $\gamma$ in $\RT$ is the signed triangulation $(\mu_{\gamma}(\RT),\sign)$.
\end{definition}

The \emph{unoriented exchange graph $\uEG^\pm(\smp)$ of signed triangulations} has vertices given by the signed triangulations of $\smp$, with edges corresponding to flips.

When punctures are replaced by vortices, we introduce the following equivalence relation on signed triangulations of $\smp$.

\begin{definition}\label{def:equivsmb}
    Two signed triangulations $(\RT_1,\sign_1)$ and $(\RT_2,\sign_2)$ are said to be \emph{$\sim$ equivalent} if and only if
    \begin{itemize}
    \item $\RT_1=\RT_2$;
    \item For any vortex $V \in \P$, if $\sign_1(V) \neq \sign_2(V)$, then $V$ lies in a self-folded triangle as depicted in Figure~\ref{fig:pm}.
    \end{itemize}
    In such a case, i.e., when $\sign_1(V) \neq \sign_2(V)$, we switch the labeling (or coloring) of the loop edge and the self-folded edge locally. This ensures that the edges of equivalent signed triangulations admit a canonical correspondence.

    Moreover, two flips $\mu_{\gamma_i}\colon \RT_i \to \RT_i'$ are equivalent if $\RT_1\sim \RT_2$, $\RT_1'\sim \RT_2'$, and $\gamma=\gamma'$.
\end{definition}

\begin{figure}[htpb]\centering
\begin{tikzpicture}[scale=1.2]\clip(0-1.1,0-1.1)rectangle(2.5+1.1,0+1.1);
\begin{scope}[shift={(0,0)}]
\draw[blue,very thick](0,0)to(0,-1)
    (1.25,0)node[black]{\Large{$\sim$}};
\draw[\ql,very thick](0,-1)
    .. controls +(45:2.5) and +(135:2.5) .. (0,-1);
\draw(0,0)circle(1)(0,1)\nn(0,-1)\nn;
\draw(0,0)\vot;\end{scope}
\begin{scope}[shift={(2.5,0)}]
\draw[\ql,very thick](0,0)to(0,-1);
\draw[blue,very thick](0,-1)
    .. controls +(45:2.5) and +(135:2.5) .. (0,-1);
\draw(0,0)circle(1)(0,1)\nn(0,-1)\nn;
\draw(0,0)\tov;\end{scope}
\end{tikzpicture}
\caption{Vortexes in self-folded triangles}\label{fig:pm}
\end{figure}
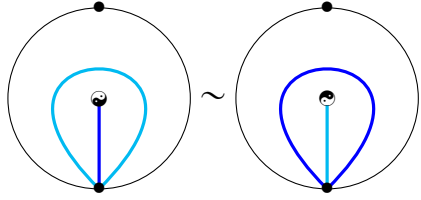

We define the \emph{unoriented exchange graph $\uEG(\smv)$ of signed triangulations} of $\smv$ to be the quotient graph $\uEG^\pm(\smp)$ by the equivalence relation $\sim$:
\begin{gather}
    \uEG(\smv)=\uEG^\pm(\smp)/\sim.
\end{gather}

Moreover, we identify $\uEG(\smp)$ with the subgraph $\uEG^{+}(\smp)=\{(\RT,+)\mid\RT\in\uEG(\smp)\}$.

\begin{lemma}\label{lem:nreg}
$\uEG(\smv)$ is $n$-regular and any each edge starting from $(\RT,\sign)$ is labeled by exactly one open arc in $\RT$.
\end{lemma}

\begin{proof}
Let $(\RT,\sign)$ be a representative in $\uEG(\smv)$.
For each non-self-folded arc in $\RT$, there is clearly a unique flip labeled by that arc.  Moreover, for any self-folded arc $\gamma$, by switching the sign at the corresponding vortex, $\gamma$ becomes (the labeling of) the corresponding loop edge, which is flippable; this corresponds to a flip $\mu'_\gamma$. Hence there are exactly $n$ flips at $(\RT,\sign)$ labeled by arcs of $\RT$.

For any other representative $(\RT',\sign')\sim(\RT,\sign)$, the same argument applies, yielding $n$ flips at $(\RT',\sign')$. Furthermore, these flips correspond exactly (i.e., are equivalent) to those with the same labels at $(\RT,\sign)$.
Therefore, the lemma follows.
\end{proof}

%

For any arc $\gamma$ in a triangulation $\RT$, denote by $\pi(\gamma)$ the arc in $\RT$ defined as follows: if $\gamma$ is a self-folded arc, then $\pi(\gamma)$ is the loop edge enclosing $\gamma$; otherwise, set $\pi(\gamma)=\gamma$.

\begin{definition}[\cite{FST,CL}]\label{def:puzzleQP}
Let $\smv$ be a MSv and $\RT$ a triangulation of $\smv$. The associated quiver $Q_\RT$ is defined as follows:
\begin{itemize}
\item The vertices $1,\cdots,n$ are indexed by the open arcs $\gamma_1,\cdots,\gamma_n$ in $\RT$.
\item For the unreduced quiver $\overline{Q_\RT}$, there is an arrow from $i$ to $j$ whenever there is a triangle of $\RT$ with $\pi(\gamma_i)$ and $\pi(\gamma_j)$ as two edges such that $\pi(\gamma_j)$ follows $\pi(\gamma_i)$ in the clockwise order.
\item Remove a maximal collection of disjoint 2-cycles from $\overline{Q_\RT}$ to get $Q_\RT$.
\end{itemize}
Moreover, there is an unreduced potential $\overline{W_\RT}$
associated to the quiver $\overline{Q_\RT}$ (for details, cf. \cite{CL}).
Furthermore, denote by $(Q_\RT,W_\RT)$ the reduced quiver with potential of $(\overline{Q_\RT}, \overline{W_\RT})$, in the sense of \cite{DWZ}.
\end{definition}

An easy observation is the following, where the cases correspond to those in \Cref{def:gpds}, respectively.

\begin{lemma}\label{lem:no ar}
    Let $\RT$ be a triangulation of $\smp$. Let $i$ and $j$ be two distinct vertices in the quiver $Q_\RT$. Then there is no arrow between $i$ and $j$ if and only if one of the following holds.
    \begin{itemize}
        \item $\gamma_i$ and $\gamma_j$ are not adjacent in any triangle of $\RT$.
        \item $\gamma_i$ and $\gamma_j$ are adjacent in two triangles of $\RT$ that form a digon.
        \item $\gamma_i$ and $\gamma_j$ form a self-folded triangle of $\RT$.
    \end{itemize}
    There is exactly one arrow between $i$ and $j$ if and only if:
    \begin{itemize}
        \item $\gamma_i$ and $\gamma_j$ are adjacent in only one non-self-folded triangle of $\RT$.
    \end{itemize}
    There are exactly (or equivalently, at most) two arrows between $i$ and $j$ if and only if:
    \begin{itemize}
        \item $\gamma_i$ and $\gamma_j$ are adjacent in two triangles of $\RT$ that form an annulus.
    \end{itemize}
\end{lemma}

Next, we recall the notion of tagged arcs, which provide an alternative description of signed triangulations.

\subsection{Signed triangulations via tagged arcs}
\begin{definition}[\cite{FST}]\label{def:tagged}
On $\smv$, we have the following notions.
\begin{itemize}
\item A \emph{tagged arc} is a pair $(\gamma,\kappa)$ where
\begin{itemize}
    \item $\gamma$ is an open arc that does not cut out a once-punctured monogon by a self-intersection at its endpoints (if any), and
    \item $\kappa$ is a map assigning a sign $\pm$ to each endpoint of $\gamma$ that is a vortex, such that if both endpoints of $\gamma$ coincide at the same vortex, then their assigned signs agree.
\end{itemize}
\item Two tagged arcs are called \emph{compatible} if
    \begin{itemize}
      \item they do not intersect in the interior $\surf^\circ$ of the surface, and
      \item at any common endpoint, their taggings agree unless the arcs are homotopic and differ in tagging at exactly one common endpoint.
    \end{itemize}
\item A \emph{tagged triangulation} $\RTx$ of $\surfp$ is a maximal collection of pairwise compatible tagged arcs.
\end{itemize}
Denote by $\EGx(\surfp)$ the set of all tagged triangulations on $\surfp$.
Note that in a tagged triangulation, a vortex is drawn with signs if all tagged arcs incident to it have the same tagging; in this case, the tagging signs on the arcs can be omitted. Otherwise, the vortex is drawn without signs.
\end{definition}

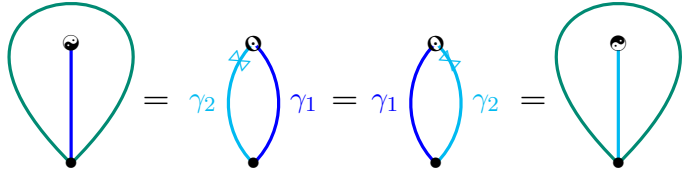
\begin{figure}[ht]\centering
\begin{tikzpicture}[scale=1.6]\clip(0-1.1,0-1.1)rectangle(6+1.1,0.4);
\draw(1.2,-.5)node[black]{\Large{$=$}}
    (2.75,-.5)node[black]{\Large{$=$}}
    (4.3,-.5)node[black]{\Large{$=$}};
\begin{scope}[shift={(.5,0)}]
\draw[blue,very thick](0,0)to(0,-1);
\draw[\hh,very thick](0,-1)
    .. controls +(45:2.5) and +(135:2.5) .. (0,-1);
\draw(0,0)\vot (0,-1)\nn;
\end{scope}
\begin{scope}[shift={(5,0)}]
\draw[\ql,very thick](0,0)to(0,-1);
\draw[\hh,very thick](0,-1)
    .. controls +(45:2.5) and +(135:2.5) .. (0,-1);
\draw(0,0)\tov (0,-1)\nn;
\end{scope}
\begin{scope}[shift={(2,0)}]
\draw[\ql,very thick](0,0)to[bend left=-45]node[left]{$\gamma_2$}(0,-1)
    (-.115,-.15)node[rotate=-40]{$\bowtie$};
\draw[blue,very thick](0,0)to[bend left=45]node[right]{$\gamma_1$}(0,-1) ;
\draw(0,0)\dpole (0,-1)\nn;
\end{scope}
\begin{scope}[shift={(3.5,0)},xscale=-1]
\draw[\ql,very thick](0,0)to[bend left=-45]node[right]{$\gamma_2$}(0,-1)
    (-.115,-.15)node[rotate=-40]{$\bowtie$};
\draw[blue,very thick](0,0)to[bend left=45]node[left]{$\gamma_1$}(0,-1) ;
\draw(0,0)\dpole (0,-1)\nn;
\end{scope}
\end{tikzpicture}
\caption{Self-folded triangles via tagged arcs}\label{fig:pm2}
\end{figure}

For any boundary component $\partial\in\partial\surfp$,
the rotation $\rho_\partial$ is the operation of moving a marked point to the next consecutive marked point,
as shown in Figure~\ref{fig:rotate}, where the orientation is anticlockwise if the surface lies inside the boundary curve, and clockwise if the surface lies outside.

\begin{figure}[htpb]\centering
\begin{tikzpicture}[scale=.3]
\draw[thick](0,0) circle (5) 
;
\draw[thick,fill=gray!20](0,0) circle (1.5) (126:5.9)node{$\rho$}(126:.5)node{$\rho$};
  \foreach \j in {1,...,5}{\draw(90-72*\j:5) node{$\bullet$};}
  \foreach \j in {1,...,3}{\draw(-90-120*\j:1.5) node{$\bullet$};}

\draw[blue,->,>=latex](90+36-15:5.4)arc(90+36-15:90+36+15:5.4);
\draw[blue,->,>=latex](90+36+25:1.1)arc(90+36+25:90+36-45:1.1);
\end{tikzpicture}
\caption{The rotation on $\surfp$}
\label{fig:rotate}
\end{figure}

\begin{definition}\label{def:T-rotate}\cite[Def.~3.4]{BQ}
The \emph{universal tagged rotation} $\widetilde{\rho}$ is the element
\begin{gather}\label{eq:T-rotation}
    \widetilde{\rho}=\prod_{\partial\subset\partial\surfp} \rho_\partial \cdot \prod_{V\in\P} \sx_V
\end{gather}
in $\MCG(\surfv)$,
where the first product runs over all connected components $\partial$ of $\partial\surfp$
and $\sx_V$ is the tagged switching at $V$ for the $\z2$ associated to $V$.
Note that the order in \eqref{eq:T-rotation} does not matter since the $\rho_\partial$'s and $\sx_V$'s commute.
\end{definition}

One can restrict the universal tagged rotation to a local one. When cutting along arcs in a tagged triangulation $\RTx$, $\surfv$ will be divided into polygons (triangles as type I puzzle pieces, digons together with square or pentagon or hexagon in type II/III/IV puzzle pieces respectively.) Cf. Figure~\ref{fig:type IV}.

\begin{definition}\label{def:l.a.}
The \emph{local anticlockwise tagged (=l.a.t.) rotation} for a tagged arc $\gamma$ in a tagged triangulation $\RTx$ is the universal tagged rotation operation on $\gamma$ which takes place in the neighbourhood $N(\gamma)$,
which is the union of the two polygons that contain $\gamma$ (as boundaries when cutting along arcs in $\RTx$).
Note that in $N(\gamma)$:
\begin{itemize}
  \item A vortex is still considered as a vortex in $N(\gamma)$ if
  there are exactly two tagged arcs in $\RTx$ that are incident at it with different taggings. So the tagging of $\gamma$ at it will be changed. Other vortices of $\surfv$ will be regarded as marked points in $N(\gamma)$.
  \item When moving the endpoints of $\gamma$ on the boundary of $N(\gamma)$ to the next consecutive one, one needs to skip vortices in $N(\gamma)$.
\end{itemize}
\end{definition}

\begin{figure}[ht]\centering
\begin{tikzpicture}[scale=1.2];
\begin{scope}[shift={(1.8,.4)},scale=1.2]
\draw[blue,thin](0,0)circle(.4);
\draw[blue,-stealth](.4,0)to(.4,.01);
\draw[blue](0,0)node{\tiny{\sx}} (0,.4)\snn(0,-.4)\snn;
\end{scope}
\begin{scope}[shift={(-1.8,.4)},scale=1.2]
\draw[\ql,thin](0,0)circle(.4);
\draw[\ql,-stealth](.4,0)to(.4,.01);
\draw[\ql](0,0)node{\tiny{\sx}} (0,.4)\snn(0,-.4)\snn;
\end{scope}

\begin{scope}[shift={(0,0)}]
\draw[\ql,very thick](0,0)to[bend left=-45]node[left]{}(0,-1)
    (-.115,-.15)node[rotate=-40]{$\bowtie$};
\draw[blue,very thick](0,0)to[bend left=45]node[right]{}(0,-1) ;
\draw[thick](0,0)circle(1)\dpole (0,1)\nn(0,-1)\nn;
\end{scope}

\draw[blue,->,thick,>=stealth](1.2,-.3)to node[below,black]{\footnotesize{in $\surfp$}}(2.4,-.3);
\draw[\ql,->,thick,>=stealth](-1.2,-.3)to node[below,black]{\footnotesize{in $\surfp$}}(-2.4,-.3);

\begin{scope}[shift={(3.6,0)}]
\draw[thick](0,0)circle(1);
\draw[blue,very thick](0,0)to(0,1);
\draw[\ql,very thick](0,0)to(0,-1);
\draw(0,0)(0,1)\nn(0,-1)\nn;
\draw(0,0)\dpole;\end{scope}
\begin{scope}[shift={(-3.6,0)},yscale=-1]
\draw[thick](0,0)circle(1);
\draw[blue,very thick](0,0)to(0,1);
\draw[\ql,very thick](0,0)to(0,-1);
\draw(0,0)(0,1)\nn(0,-1)\nn;
\draw(0,0)\dpole;\end{scope}
\end{tikzpicture}
\caption{l.a.t. rotation as forward flip with type II puzzle piece}
\label{fig:flip II}
\end{figure}
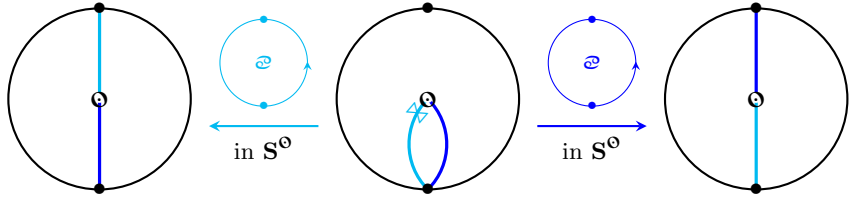

\begin{lemma}\label{lem:1}
There is a canonical bijection between $\uEG(\smv)$ and $\uEGx(\smv)$.
Under this bijection, the forward flip in $\uEG(\smv)$ corresponds to l.a.t. rotation in $\uEGx(\smv)$ and thus defines an isomorphism between oriented graphs.
\end{lemma}
\begin{proof}
One can turn a signed triangulation into a tagged triangulation as follows:
\begin{itemize}
  \item For an arc that is not a loop edge,
  its endpoints inherit the sign of the vortices at which it is incident.
  \item For an arc that is a loop edge, replace it with the self-folded edge, but with different signs at the vortex inside the self-folded triangle.
\end{itemize}
It is straightforward to check that such a construction gives rise to the required bijection and isomorphism.
\end{proof}
As a direct consequence, this provides an alternative proof of \Cref{lem:nreg}.

One of the key results in \cite{FST} is the following,
which is a generalization of a classical topological result concerning the exchange graphs of ideal triangulations.

\begin{theorem}[{\cite[Thm.~9.17]{FST}}]\label{thm:FST917}
$\pi_1(\uEGx(\smv))$ is generated by squares and pentagons.
\end{theorem}

\subsection{The cluster exchange groupoid for MSv}\label{sec:ceg MSv}

We adopt the following canonical isomorphism from \cite[Thm.~7.11]{FST} to define
the unoriented cluster exchange graph $\uCEG(\surfp)$ of $\surfp$:
\begin{gather}\label{eq:egceg}
    \uEG(\surfv)\cong\uCEG(\surfv).
\end{gather}
Denote by $\EG(\smv)$ the oriented version of $\uEG(\smv)$, obtained by replacing each unoriented edge with a directed 2-cycle.

\begin{definition}\label{def:gpds}
The exchange groupoid $\eg(\smv)$ is the quotient of the path groupoid of $\EG(\smv)$ by the following Sq-1/2/3./Pen./Hex. relations (collectively denoted SPH), starting at any triangulation $\T$ in $\EG(\smv)$:
\begin{itemize}
\item[Sq-1.] If two open arcs are not adjacent in any triangle of $\RT$, then the forward flips with respect to them form a square in $\EG(\smv)$, as in the left picture of \cite[Fig. 9]{KQ2}, and we impose the commuting square relation.
\item[Sq-2.] If two open arcs are adjacent in two triangles of $\RT$ that form a digon, then the forward flips with respect to them form a square in $\EG(\smv)$, as in the upper picture of Figure~\ref{fig:S2.lift}, and we impose the commuting square relation. Note that in the figure, the green/red disks are the source/sink of the square.
\item[Sq-3.] If two open arcs form a self-folded triangle of $\RT$, then the forward flips with respect to them form a square in $\EG(\smv)$, as in the upper picture of Figure~\ref{fig:S2.lift2}, and we impose the commuting square relation. Note that in the figure, the green/red disks are the source/sink of the square.

\item[Pen.] If two open arcs are adjacent in only one non-self-folded triangle of $\RT$, then they induce an oriented pentagon in $\EG(\smv)$, as in the right picture of \cite[Fig.~9]{KQ2}, and we impose the corresponding commuting pentagon relation.
\item[Hex.] If two open arcs are adjacent in two triangles of $\RT$ that form an annulus, then they induce an oriented hexagon in $\EG(\smv)$, as in \cite[Fig.~10]{KQ2}, and we impose the corresponding commuting hexagon relation.
\end{itemize}
\end{definition}

\begin{remark}
    For a pair of open arcs in a triangulation of $\smv$, the five cases in Definition~\ref{def:gpds} are all the possible cases of the pair. Moreover,
    \begin{itemize}
    \item in the Sq-1/2/3 cases, by \Cref{lem:no ar}, there is no arrow between the corresponding vertices of $Q_{\RT}$, and the square relation required here corresponds to relation $1^\circ$ in \cite[Definition~2.3]{KQ2};
    \item in the Pen. case, there is only one arrow between the corresponding vertices of $Q_{\RT}$, and the pentagon relation corresponds to relation $2^\circ$ in \cite[Definition~2.3]{KQ2};
    \item in the Hex. case, there are two arrows between the corresponding vertices of $Q_{\RT}$, and the hexagon relation corresponds to relation $3^\circ$ in \cite[Definition~2.3]{KQ2}.
\end{itemize}
\end{remark}

By definition, we can upgrade the isomorphism \eqref{eq:egceg} to an isomorphism of groupoids
\begin{equation}\label{eq:iso}
    \eg(\smv)\cong\ceg(\smv),
\end{equation}
which may also be regarded as the definition of $\ceg(\smv)$ in the sense of \cite{KQ2}.

At each vertex $\RTe$ of $\eg(\surfv)$, the 2-cycle/loop $t_\gamma$ at $\RTe$, corresponding to an edge $\gamma$ of $\RTe$, is called a local twist.

\begin{definition}[{\cite[Def.~2.6]{KQ2}}]
The cluster twist group $\CBr(\RTe)$ of $\surfv$ at $\RTe$ is the subgroup of $\pi_1(\eg(\surfv), \RTe)$ generated by the local twists $t_\gamma$, for all $\gamma \in \RT$.
\end{definition}

By \cite[Prop.~2.9]{KQ2}, \Cref{thm:FST917} implies the following.
\begin{corollary}\label{cor:FST}
$\CBr(\RTe)=\pi_1(\eg(\surfv), \RTe)$.
\end{corollary}

For later use in the proof of \Cref{prop:cbg}, we record the following lemma.

\begin{lemma}[{\cite[Lem.~2.7]{KQ2}}]\label{lem:cobr}
Let $\T=\{\gamma_i\}$ be a triangulation of $\smv$.
\begin{itemize}
\item If there is no arrow between $i$ and $j$ in $Q_{\T}$, then $\Co(t_i,t_j)$ holds;
\item If there is exactly one arrow between $i$ and $j$ in $Q_{\T}$, then $\Br(t_i,t_j)$ holds.
\end{itemize}
\end{lemma}

\section{Exchange graphs for DMSv}\label{sec:EG}

\subsection{Ideal decorated triangulations with flips}\

Following \cite{QQ}, we decorate $\smp$ as follows.
\begin{definition}[{\cite[Def.~3.1]{QQ}}]\label{def:DMS}
The \emph{decorated marked surface with punctures} (DMSp) $\sop$ is a marked surface $\smp$ together with a fixed set $\Tri$ of $\aleph$ `decorating' points in the interior $\surf^\circ$, where the number $\aleph$ of decorations is given by
\begin{gather}\label{eq:Tri}
    \aleph=\frac{2n+|\M|}{3}=4g+2p+2b+m-4.
\end{gather}
\end{definition}

The \emph{mapping class group} $\MCG(\sop)$ is defined to be the group of (isotopy classes of) homeomorphisms of $\surf$ that preserve $\M$, $\sun$ and $\Tri$, respectively, setwise.

\begin{definition}[\cite{QQ}]\label{def:arcsv}
Let $\sop$ be a DMSp.
\begin{itemize}
\item An \emph{open arc} on $\sop$ is (the isotopy class of) a curve on $\surf$ whose interior is in $\surf^\circ-\Tri$, whose endpoints are in $\M$, which is simple and essential.
\item Two open arcs are called \emph{compatible} if they do not intersect in $\surf^\circ$.
\item A \emph{(decorated) triangulation} $\T$ of $\sop$ is a collection of compatible open arcs on $\sop$,
dividing $\sop$ into $\numtri$ triangles, each of which contains exactly one point in $\Tri$.
\item The \emph{forward flip} of a triangulation $\T$ of $\sop$ at a non-self-folded arc $\gamma$ is the triangulation $\mu_{\gamma}(\T)=(\T\setminus\{\gamma\})\cup\{\gamma^\sharp\}$, where $\gamma^\sharp$ is obtained from $\gamma$
by moving its endpoints anticlockwise along the quadrilateral containing $\gamma$, as shown in the first row of Figure~\ref{fig:flip}.
\item The \emph{exchange graph} $\EG(\sop)$ is the oriented graph whose vertices
are triangulations of $\sop$ and whose edges correspond to forward flips between them.
\end{itemize}
\end{definition}

The forgetting map \eqref{eq:forget1} maps open arcs in $\sop$ to open arcs in $\smp$.
Hence there is an induced map of oriented graphs
\begin{equation}\label{eq:forget*}
    F^\sun_*\colon\EG(\sop)\xrightarrow{/\SBr(\sop)}\EG(\smp),
\end{equation}
sending a triangulation $\T$ of $\sop$ to a triangulation $\{F(\gamma)\mid \gamma\in\T\}$ of $\smp$, which is clearly surjective and is in fact a $\SBr(\sop)$-covering (cf. \cite[Lem.~3.8]{KQ2}).
See Figure~\ref{fig:flip} for the map on local flips.

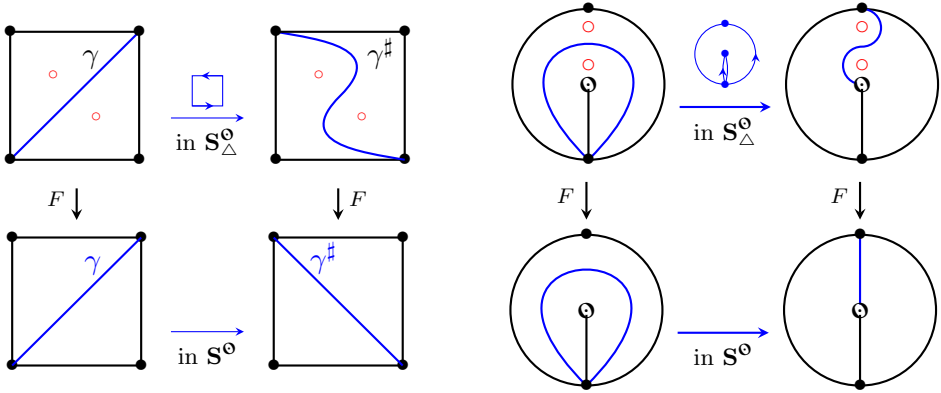
\begin{figure}[htpb]\centering
\begin{tikzpicture}[scale=.3]
    \path (-135:4) coordinate (v1)
          (-45:4) coordinate (v2)
          (45:4) coordinate (v3);
\draw[thick](v1)to(v2)node{$\bullet$}to(v3);
    \path (-135:4) coordinate (v1)
          (45:4) coordinate (v2)
          (135:4) coordinate (v3);
\draw[blue, thick](v1)to(v2);
\draw[thick](v2)node{$\bullet$}to(v3)node{$\bullet$}to
    (v1)node{$\bullet$}(45:1)node[above]{$\gamma$};
\draw[red,thick](135:1.333)node{\tiny{$\circ$}}(-45:1.333)node{\tiny{$\circ$}};
\end{tikzpicture}
\begin{tikzpicture}[scale=.8, rotate=180]
\draw[blue,<-,>=stealth](3-.6,.7)to(3+.6,.7);
\draw(3,.7)node[below,black]{\footnotesize{in $\sop$}};
\draw[blue](3-.25,.5-.5)rectangle(3+.25,.5);\draw(3,1.5)node{};
\draw[blue,-stealth](3-.25,.5-.5)to(3+.1,.5-.5);
\draw[blue,-stealth](3+.25,.5)to(3-.1,.5);
\end{tikzpicture}
\begin{tikzpicture}[scale=.3];
    \path (-135:4) coordinate (v1)
          (-45:4) coordinate (v2)
          (45:4) coordinate (v3);
\draw[,thick](v1)to(v2)node{$\bullet$}to(v3)
(45:1)node[above right]{$\gamma^\sharp$};
    \path (-135:4) coordinate (v1)
          (45:4) coordinate (v2)
          (135:4) coordinate (v3);
\draw[blue,,thick](135:4).. controls +(-10:2) and +(45:3) ..(0,0)
                             .. controls +(-135:3) and +(170:2) ..(-45:4);
\draw[,thick](v2)node{$\bullet$}to(v3)node{$\bullet$}to(v1)node{$\bullet$};
\draw[red,thick](135:1.333)node{\tiny{$\circ$}}(-45:1.333)node{\tiny{$\circ$}};
\end{tikzpicture}
\begin{tikzpicture}[xscale=1,yscale=1];
\begin{scope}[shift={(0,0)}]
\draw[thick](0,0)to(0,-1);
\draw[blue,thick](0,-1)
    .. controls +(45:2.9) and +(135:2.9) .. (0,-1);
\draw(0,.25)\ww(0,0.75)\ww;
\draw[thick](0,0)circle(1)(0,1)\nn(0,-1)\nn;
\draw(0,0)\dpole;\end{scope}
\draw[blue,->,thick,>=stealth](1.2,-.3)to node[below,black]{\footnotesize{in $\sop$}}(2.4,-.3);
\begin{scope}[shift={(3.6,0)}]
\draw[thick](0,0)circle(1);
\draw[thick](0,-1)to(0,0);
\draw[blue,thick](0,0)arc(-90:-270:.25)arc(-90:90:.25);
\draw(0,.25)\ww(0,0.75)\ww;
\draw(0,0)(0,1)\nn(0,-1)\nn;
\draw(0,0)\dpole;\end{scope}
\begin{scope}[shift={(1.8,.4)}]
\draw[blue,thin](0,0)circle(.4);
\draw[blue,-stealth](.4,0)to(.4,.01);
\draw[blue](0,-.4)to[bend left=-15](0,0);
\draw[blue,->-=.55,>=stealth](0,-.4)to[bend left=15](0,0);
\draw[blue](0,0)\snn(0,.4)\snn(0,-.4)\snn;
\end{scope}
\end{tikzpicture}

\begin{tikzpicture}[scale=.3]
\draw[thick,>=stealth,->](0,5)to node[left]{$^F$}(0,3.6);\draw(0,-4)node{ };
    \path (-135:4) coordinate (v1)
          (-45:4) coordinate (v2)
          (45:4) coordinate (v3);
\draw[,thick](v1)to(v2)node{$\bullet$}to(v3);
    \path (-135:4) coordinate (v1)
          (45:4) coordinate (v2)
          (135:4) coordinate (v3);
\draw[,thick](v2)node{$\bullet$}to(v3)node{$\bullet$}to(v1)node{$\bullet$};
\draw[blue,,thick](-135:4)to(45:4) (45:1)node[above]{$\gamma$};
\end{tikzpicture}
\begin{tikzpicture}[scale=.8, rotate=180]
\draw(3,1.5)node{}(3,.5)node[below]{\footnotesize{in $\surfp$}};
\draw[blue,<-,>=stealth](3-.6,.5)to(3+.6,.5);;
\end{tikzpicture}
\begin{tikzpicture}[scale=.3]
\draw[thick,>=stealth,->](0,5)to node[right]{$^F$}(0,3.6);\draw(0,-4)node{ };
    \path (-135:4) coordinate (v1)
          (-45:4) coordinate (v2)
          (45:4) coordinate (v3);
\draw[,thick](v1)to(v2)node{$\bullet$}to(v3);
    \path (-135:4) coordinate (v1)
          (45:4) coordinate (v2)
          (135:4) coordinate (v3);
\draw[,thick](v2)node{$\bullet$}to(v3)node{$\bullet$}to(v1)node{$\bullet$};
\draw[blue,,thick](135:4)to(-45:4) (130:1)node[above]{$\gamma^\sharp$};;
\end{tikzpicture}
\begin{tikzpicture}[xscale=1,yscale=1];
\draw[thick,>=stealth,->](0,1.7)to node[left]{$^F$}(0,1.2);
\begin{scope}[shift={(0,0)}]
\draw[thick](0,0)to(0,-1);
\draw[blue,thick](0,-1)
    .. controls +(45:2.9) and +(135:2.9) .. (0,-1);
\draw[thick](0,0)circle(1)(0,1)\nn(0,-1)\nn;
\draw(0,0)\dpole;\end{scope}
\draw[blue,->,thick,>=stealth](1.2,-.3)to node[below,black]{\footnotesize{in $\surfp$}}(2.4,-.3);
\begin{scope}[shift={(3.6,0)}]
\draw[thick,>=stealth,->](0,1.7)to node[left]{$^F$}(0,1.2);
\draw[thick](0,0)circle(1);
\draw[thick](0,-1)to(0,0);
\draw[blue,thick](0,0)to(0,1);
\draw(0,0)(0,1)\nn(0,-1)\nn;
\draw(0,0)\dpole;\end{scope}
\end{tikzpicture}
\caption{The forward flip in $\EG(\sop)$ and the forgetful map}
\label{fig:flip}
\end{figure}

\subsection{Signed triangulations in the vortex setting}\

We now construct the `universal cover' of $\EG(\surfv)$.

Similar to \Cref{def:EGv}, we begin with the native version of signed decorated triangulations.
\begin{definition}
A signed triangulation of $\sop$ is a pair $\Te=(\T,\sign)$ with $\T$ a triangulation of $\sop$ and $\sign\in\ZZP$. The \emph{forward flip} of $\Te$ at a non-self-folded arc $\gamma$ in $\T$ is defined to be $\mu_{\gamma}(\Te)=(\mu_{\gamma}(\T),\sign)$.

The \emph{oriented graph $\EG^\pm(\sop)$ of signed triangulations of $\sop$} is the oriented graph whose vertices are the signed triangulations of $\sop$ and whose oriented edges correspond to forward flips between them.
\end{definition}

The map \eqref{eq:forget*} naturally extends to
\begin{equation}\label{eq:forget*1}
    F_*^{\pm}\colon\EG^\pm(\sop)\xrightarrow{/\SBr(\sop)}\EG^\pm(\smp).
\end{equation}

Denote by $\T^{\nsf}$ the partial triangulation consisting of non-self-folded arcs in $\T$.

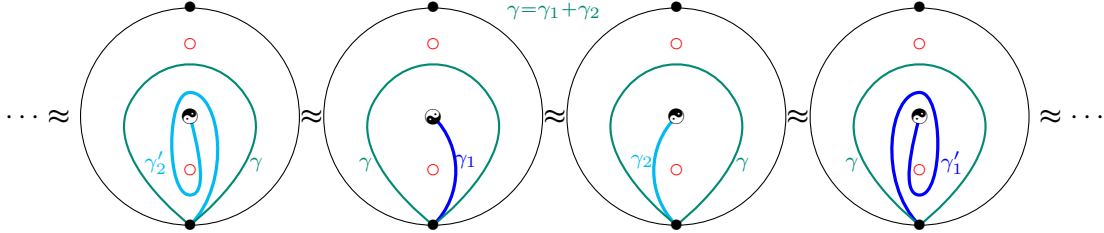
\begin{figure}[ht]\centering
\begin{tikzpicture}[scale=.8]
\begin{scope}[shift={(0,0)}]
\draw[blue,very thick](0,.8)to[bend left=45](0,-1)
    (0.55,0)node{$^{\gamma_1}$};
\draw[\hh,thick] plot [smooth,tension=1.6] coordinates
    {(0,-1) (-1,1) (1,1) (0,-1)};
\draw(0,.8)\vot(0,-1)\nn(0,2.6)\nn(0,0.8)circle(1.8)(0,-.1)\ww(0,2)\ww;
\draw[\hh](2,2.5)node{$^{\gamma=\gamma_1+\gamma_2}$}(-1.1,0)node{$_\gamma$};
\end{scope}
\draw(2,.8)node{$\thickapprox$};
\begin{scope}[shift={(4,0)}]
\draw[\hh,thick](1.1,0)node{$_\gamma$};
\draw[\ql,very thick](0,.8)to[bend left=-45](0,-1)
    (-0.55,0)node{$^{\gamma_2}$};
\draw[\hh,thick] plot [smooth,tension=1.6] coordinates
    {(0,-1) (-1,1) (1,1) (0,-1)};
\draw(0,.8)\tov(0,-1)\nn(0,2.6)\nn (0,0.8)circle(1.8)(0,-.1)\ww(0,2)\ww;
\draw(2,.8)node{$\thickapprox$};
\end{scope}
\begin{scope}[shift={(8,0)},xscale=-1]
\draw[\hh,thick](1.1,0)node{$_\gamma$};
\draw[blue,very thick](0,.8)
    .. controls +(-75:1) and +(0:.2) ..(0,-.5)
    .. controls +(180:.4) and +(180:.4) ..(0,1.2)
    .. controls +(0:.5) and +(45:1) ..(0,-1)
    (-0.55,0)node{$^{\gamma_1'}$};
\draw[\hh,thick] plot [smooth,tension=1.6] coordinates
    {(0,-1) (-1,1) (1,1) (0,-1)};
\draw(0,.8)\tov(0,-1)\nn(0,2.6)\nn (0,0.8)circle(1.8)(0,-.1)\ww(0,2)\ww;
\draw(-1.8,.8)node[right]{$\thickapprox\cdots$};
\end{scope}
\begin{scope}[shift={(-4,0)}]
\draw[\hh,thick](1.1,0)node{$_\gamma$};
\draw[\ql,very thick](0,.8)
    .. controls +(-75:1) and +(0:.2) ..(0,-.5)
    .. controls +(180:.4) and +(180:.4) ..(0,1.2)
    .. controls +(0:.5) and +(45:1) ..(0,-1)
    (-0.55,0)node{$^{\gamma_2'}$};
\draw[\hh,thick] plot [smooth,tension=1.6] coordinates
    {(0,-1) (-1,1) (1,1) (0,-1)};
\draw(0,.8)\tov(0,-1)\nn(0,2.6)\nn (0,0.8)circle(1.8)(0,-.1)\ww(0,2)\ww;
\draw(2,.8)node{$\thickapprox$};
\draw(-1.8,.8)node[left]{$\cdots\thickapprox$};
\end{scope}
\end{tikzpicture}
\caption{Equivalent triangulations when switching sign at a vortex}\label{fig:switch}
\end{figure}
\begin{definition}\label{def:equivsov}
Two signed triangulations $(\T_1,\sign_1)$ and $(\T_2,\sign_2)$
are said to be \emph{$\sim$ equivalent} as triangulations of $\sop$, if the following hold.
\begin{itemize}
  \item $\T_1^{\nsf}=\T_2^{\nsf}$.
  \item If a puncture $P$ is not inside a self-folded triangle in $\T_2$, then $\sign_1(P)=\sign_2(P)$.
  \item If a puncture $P$ is inside a self-folded triangle $T$,
  then $\sign_1(P)=\sign_2(P)$ if and only if the self-folded edges in $T_1$ and $T_2$ differ by $\LL$.
  Equivalently, they are related by a sequence of even number of flips, cf. \Cref{fig:switch}.
\end{itemize}
Two signed triangulations $(\T_1,\sign_1)$ and $(\T_2,\sign_2)$ are said to be \emph{$\approx$ equivalent} as triangulations of $\sov$, if there is a mapping class $b\in\LL$ such that $(b(\T_1),\sign_1)\sim(\T_2,\sign_2)$.
Moreover, two forward flips $\mu_{\gamma_i}\colon(\T_i,\sign_i)\to(\T_i',\sign_i')$ are $\approx$ equivalent if there exists a common $b\in\LL$ such that $(b(\T_1),\sign_1)\sim(\T_2,\sign_2)$
and $(b(\T_1'),\sign_1')\sim(\T_2',\sign_2')$.
\\
The \emph{exchange graph of signed triangulations} on $\sov$ is defined to be
\begin{gather}\label{eq:Lrm}
    \EG(\sov)=\EG^\pm(\sop)/\approx.
\end{gather}
\end{definition}
Similar to \Cref{lem:nreg}, we have the following, which can be alternatively deduced using tagged arcs in the next subsection.
\begin{lemma}
$\EG(\sov)$ is well-defined and in fact an $(n,n)$-regular graph.
\end{lemma}
We proceed to prove the main result in this subsection.
\begin{theorem}\label{thm:cover}
The covering \eqref{eq:forget*1} upgrades to
\begin{equation}\label{eq:forgetv}
    F_*^\P:\EG(\sov) \xrightarrow{/\SBr(\sov)} \EG(\smv)
\end{equation}
sending a signed triangulation $(\T,\sign)$ to $(F_*^\sun(\T),\sign)$.
\end{theorem}

\begin{proof}
Firstly, we need to show that two equivalent signed triangulations  $(\T_1,\sign_1)\approx(\T_1,\sign_2)$ map to the equivalent ones in $\EG(\smv)$.
By definition, there exists $b\in\LL$ such that $b(\T_1^{\nsf})=\T_2^{\nsf}$.
Hence either $F_*^\sun( \T_1^{\nsf} ) = F_*^\sun( \T_2^{\nsf} )$
or $F_*^\sun( \T_1 )^{\nsf} = F_*^\sun( \T_2 )^{\nsf}$.
But for triangulations of $\surfp$, this implies $F_*^\sun( \T_1 ) = F_*^\sun( \T_2 )$.
For any puncture $P$ not inside a self-folded triangle in $\T_2$,
we have $\sign_1(P)=\sign_2(P)$, hence
$F_*^\P(\T_1,\sign_1)\sim F_*^\P(\T_1,\sign_2)$.

Secondly, we show that $F_*^\P$ is a $\SBr(\sov)$-covering of sets.
Take any $(\RT,\sign)\in\EG(\smv)$ and let $\T$ be any triangulation of $\sop$ such that
$F_*^\sun(\T)=\RT$.
We observe the following:
\begin{itemize}
  \item[Ob.1] $F_*^\P(\T',\sign)=(\RT,\sign)$ if and only if $\T'=b(\T)$ for some $b\in\SBr(\sop)$.
  \item[Ob.2] $(\T',\sign)\approx(\T,\sign)$ if and only if $\T'=b(\T)$ for some $b\in\LL$.
  This is because when the sign is fixed, the mapping class $b'$ satisfying $(\T')^{\nsf}=b'(\T^{\nsf})$ forces the two triangulations $\T'$ and $b'(\T)$ to differ by $\LL$ at each self-folded triangle.
\end{itemize}
Therefore, we deduce that the preimages of $(\RT,\sign)$ of the form $(\T,\sign)$ are parameterized by $\SBr(\sov)$.
It remains to show that if $F_*^\P(\T',\sign')\sim(\RT,\sign)$,
then $(\T',\sign')\approx(b(\T),\sign)$ for some $b\in\SBr(\sop)$.
As $F_*^\P(\T',\sign')=(\RT,\sign)$, we know that $\sign'$ and $\sign$ can only differ at punctures that are inside self-folded triangles.
Locally by forward flipping self-folded edges of $\T'$, we can change these signs
in the representatives in the $\approx$ equivalent class of $(\T',\sign')$.
Namely, there exists $(\T'',\sign)\approx(\T',\sign')$.
By the (first) observation above, we see that $F_*^\P$ is a $\SBr(\sov)$-covering of sets.

Finally, we need to show that $F_*^\P$ also preserves edges, i.e. it is a covering of graphs locally. This follows from the lemma above.
\end{proof}

\subsection{Signed triangulations via tagged arcs}

The definition of tagged arcs is the same as in the undecorated case \Cref{def:tagged}.
\begin{definition}

A \emph{tagged triangulation} $\Tx$ of $\sop$ is a collection $\Tx$ of tagged arcs in $\sop$ such that the following hold.
\begin{itemize}
  \item $\Tx$ divides $\sop$ into once-decorated digons, triangles or squares.
  \item The forgetful map $F^\sun_*$ turns $\Tx$ into a tagged triangulation of $\surfp$.
\end{itemize}
Denote by $\EGx(\sop)$ the set of tagged triangulations on $\sop$.
\end{definition}

\begin{construction}\label{con:t to s}
In a tagged triangulation $\Tx$ of $\sop$, let $\gamma_1$ and $\gamma_2$ be the tagged arcs of a digon that $\Tx$ cuts out. Then $\gamma_1$ and $\gamma_2$ are not loops. Moreover, their taggings agree at exactly one shared endpoint. Denote this endpoint by $R$ and the other endpoint, which must be a vortex, by $V$. In this case, there is a loop $\gamma$ (based at $R$) obtained from $\gamma_1\cup\gamma_2$ by smoothing out their endpoint at $V$. If we replace one of $\gamma_i$ by $\gamma$, we obtain a self-folded triangle locally.
If we apply this replacement to all the digons (say there are $K$ many of them in $\Tx$),
we can obtain $2^K$ associated signed triangulations, which are $\sim$ equivalent to each other.

\begin{figure}[htpb]\centering
\begin{tikzpicture}[scale=1]
\begin{scope}[shift={(0,0)}]
\draw[\hh,thick](-1.1,0)node{$_\gamma$};
\draw[blue,very thick](0,.8)to[bend left=45](0,-1)
    (0.55,0)node{$^{\gamma_1}$};
\draw[\hh,thick] plot [smooth,tension=1.6] coordinates
    {(0,-1) (-1,1) (1,1) (0,-1)};
\draw(0,.8)\vot node[above]{$p$}(0,-1)\nn node[below]{$R$}(0,-.1)\ww;
\end{scope}
\draw(1.8,0)node{\Large{$\Longleftarrow$}};
\begin{scope}[shift={(3,0)}]
\draw[\ql,very thick](0,.8)to[bend left=-45](0,-1)
    (-0.55,0)node{$^{\gamma_2}$}
        (-.23,.8-.3)node[rotate=-40]{$\bowtie$};;
\draw[blue,very thick](0,.8)to[bend left=45](0,-1)
    (0.55,0)node{$^{\gamma_1}$};
\draw(0,.8)\dpole node[above]{$p$}(0,-1)\nn node[below]{$R$}(0,-.1)\ww;
\draw(1.2,0)node{\Large{$\Longrightarrow$}};
\end{scope}
\begin{scope}[shift={(6,0)},xscale=1]
\draw[\hh,thick](1.1,0)node{$_\gamma$};
\draw[\ql,very thick](0,.8)to[bend left=-45](0,-1)
    (-0.55,0)node{$^{\gamma_2}$};
\draw[\hh,thick] plot [smooth,tension=1.6] coordinates
    {(0,-1) (-1,1) (1,1) (0,-1)};
\draw(0,.8)\tov node[above]{$p$}(0,-1)\nn node[below]{$R$}(0,-.1)\ww;
\end{scope}
\end{tikzpicture}
\caption{A tagged triangulation corresponds to two signed triangulations, locally}
\label{fig:switch3}
\end{figure}
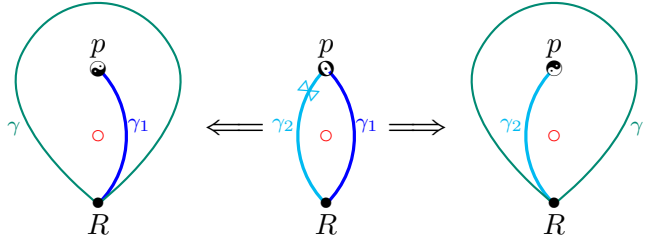
\end{construction}

Clearly, Construction~\ref{con:t to s} above induces a well-defined map from $\EGx(\sop)$ to $\EG(\sop)/\sim$ (considered as sets).

Passing to the vortex version: a tagged arc on $\sov$ is an equivalence class of tagged arcs on $\sop$ under the action of $\LL$.

\begin{definition}
A tagged triangulation on $\sov$ is a collection of tagged arcs on $\sov$ such that there are representatives, one from each tagged arc in the collection, which form a tagged triangulation $\Tx$ of $\sop$. We may use $\Tx$ to represent such a collection.
Denote by $\EGx(\sov)$ the set of all tagged triangulations on $\sov$.
\end{definition}

\begin{lemma}\label{lem:differ}
Any two representatives $\Tx$ and $\Tx'$ of a tagged triangulation of $\sov$ differ by some $b\in\LL$.
\end{lemma}

\begin{proof}
We proceed by induction on the rank \eqref{eq:n} of $\surfp$ starting with the trivial case for $n=0$.
Consider the inductive step.
Take any $\gamma\in\Tx$, which corresponds to $\gamma'\in\Tx'$.
There exists $b_0\in\LL$ such that $b_0(\gamma)=\gamma'$ and $b_0(\Tx)$ still represents the same tagged triangulation of $\sov$.
Cut $\sov$ along $\gamma'$ so that the inductive assumption can be applied, completing the proof.
\end{proof}

A direct corollary is the following.
\begin{corollary}\label{lem:t to s}
Construction~\ref{con:t to s} above induces a well-defined map from $\EGx(\sov)$ to $\EG(\sov)$ (considered as sets).
\end{corollary}

Next, we give the inverse of the map in the corollary above.

\begin{construction}\label{con:s to t}
Consider a once-decorated self-folded triangle in a signed triangulation $\Te$ of $\sop$, with loop edge $\gamma$, self-folded edge $\gamma_1$ and vortex $V$. Then there are exactly two tagged arcs in this self-folded triangle that are compatible with $\gamma_1$ (regarded as a tagged arc), cf. Figure~\ref{fig:switch4}. More precisely, they can be obtained from $\gamma\cup\gamma_1$, by smoothing out (in the self-folded triangle), either the intersection from $\gamma_1$ to $\gamma$, or the intersection from $\gamma$ to $\gamma_1$. The former smoothing gives $\gamma_2$, and the latter gives $\gamma_2'$ in Figure~\ref{fig:switch4}.
Note that the sign of $\gamma_1$ at $V$ is inherited from $V$ while $\gamma_2$ or $\gamma_2'$ has a different sign at $V$.

If we replace $\gamma$ with either $\gamma_2$ or $\gamma_2'$, we obtain a tagged triangulation, locally. Applying this replacement to all self-folded triangles (say there are $K$ of them in $\Te$) yields $2^K$ tagged triangulations of $\sop$, which are representatives of the same tagged triangulation of $\sov$.

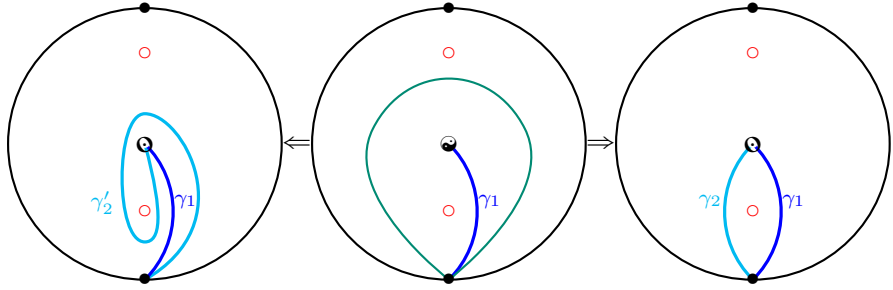
\begin{figure}[htpb]\centering
\begin{tikzpicture}[scale=1]
\begin{scope}[shift={(0,0)}]
\draw[blue,very thick](0,.8)to[bend left=45](0,-1)
    (0.55,0)node{$^{\gamma_1}$};
\draw[\hh,thick] plot [smooth,tension=1.6] coordinates
    {(0,-1) (-1,1) (1,1) (0,-1)};
\draw[thick](0,.8)\vot(0,-1)\nn(0,2.6)\nn(0,0.8)circle(1.8)(0,-.1)\ww(0,2)\ww;
\end{scope}
\draw(2,.8)node{$\Rightarrow$};
\begin{scope}[shift={(4,0)}]
\draw[\ql,very thick](0,.8)to[bend left=-45](0,-1)
    (-0.55,0)node{$^{\gamma_2}$};
\draw[blue,very thick](0,.8)to[bend left=45](0,-1)
    (0.55,0)node{$^{\gamma_1}$};
\draw[thick](0,.8)\dpole(0,-1)\nn(0,2.6)\nn (0,0.8)circle(1.8)(0,-.1)\ww(0,2)\ww;
\end{scope}
\begin{scope}[shift={(-4,0)}]
\draw[\ql,very thick](0,.8)
    .. controls +(-75:1) and +(0:.2) ..(0,-.5)
    .. controls +(180:.4) and +(180:.4) ..(0,1.2)
    .. controls +(0:.5) and +(30:1.5) ..(0,-1)
    (-0.55,0)node{$^{\gamma_2'}$};
\draw[blue,very thick](0,.8)to[bend left=45](0,-1)
    (0.55,0)node{$^{\gamma_1}$};
\draw[thick](0,-1)\nn(0,2.6)\nn (0,0.8)circle(1.8)(0,-.1)\ww(0,2)\ww;
\draw(2,.8)node{$\Leftarrow$}(0,.8)\dpole;
\end{scope}
\end{tikzpicture}
\caption{A signed triangulation corresponds to two tagged triangulations, locally}
\label{fig:switch4}
\end{figure}
\end{construction}

\begin{lemma}\label{lem:s to t}
The construction in Construction~\ref{con:s to t} induces a well-defined map from $\EG(\sov)$ to $\EGx(\sov)$, which is the inverse of the map in Lemma~\ref{lem:t to s}.
Therefore, we obtain a canonical identification
$$\mathbb{I}\colon \EG(\sov)\xrightarrow{1\leftrightarrow1}\EGx(\sov).$$
\end{lemma}

\begin{proof}
The well-definedness follows locally from the fact that the two signed triangulations in the middle of Fig.~\ref{fig:switch} correspond to the same tagged triangulation by the construction.
It is straightforward to verify that the maps above are inverse to each other, as shown in the local pictures in \Cref{fig:switch3} and \Cref{fig:switch4}.
\end{proof}

\begin{definition}\label{def:a.l.2}
    The \emph{l.a.t. rotation} for a tagged arc in a tagged triangulation of $\sov$ is similar to the one for $\surfv$ in Definition~\ref{def:l.a.}. See Figure~\ref{fig:tr1}. The arrows of the tagged exchange graph $\EGx(\sov)$ of $\sov$ are defined by the l.a.t. rotations.
\end{definition}

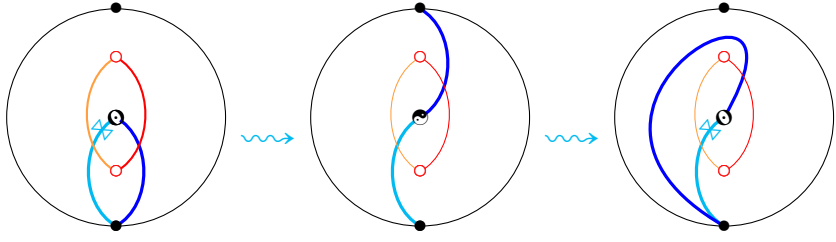
\begin{figure}[htpb]\centering
\begin{tikzpicture}[scale=.8]\clip(0-1.9,-1.1)rectangle(10+1.9,2.7);
\draw[\ql](2.5,0.4)node{\Large{$\rightsquigarrow$}}
    (7.5,0.4)node{\Large{$\rightsquigarrow$}};
\begin{scope}[shift={(0,0)}]
\draw[\ql,very thick](0,.8)to[bend left=-60](0,-1)
        (-.23,.8-.2)node[rotate=-40]{$\bowtie$};;
\draw[blue,very thick](0,.8)to[bend left=60](0,-1);
\draw(0,.8)\dpole (0,-1)\nn (0,-.1)\ww;
\draw(0,.8)\dpole(0,-1)\nn(0,2.6)\nn (0,0.8)circle(1.8)(0,-.1)\ww(0,1.8)\ww;
\draw[\qh, thick](0,-.1)\ww to[bend left=60](0,1.8)\ww;
\draw[red, thick](0,-.1)\ww to[bend left=-60](0,1.8)\ww;
\end{scope}
\begin{scope}[shift={(10,0)}]
\draw[\ql,very thick](0,.8)to[bend left=-60](0,-1)
        (-.23,.8-.2)node[rotate=-40]{$\bowtie$};;
\draw[blue,very thick](0,.8)
    .. controls +(60:3.5) and +(150:4) ..(0,-1);
\draw(0,.8)\dpole    (0,-1)\nn (0,-.1)\ww;
\draw(0,2.6)\nn (0,0.8)circle(1.8)(0,-.1)\ww(0,1.8)\ww;
\draw[\qh](0,-.1)\ww to[bend left=60](0,1.8)\ww;
\draw[red](0,-.1)\ww to[bend left=-60](0,1.8)\ww;
\end{scope}
\begin{scope}[shift={(5,0)}]
\draw[\ql,very thick](0,.8)to[bend left=-60](0,-1);
\draw[blue,very thick](0,.8)to[bend left=-60](0,2.6);
\draw(0,.8)\tov (0,-1)\nn (0,-.1)\ww;
\draw(0,-1)\nn(0,2.6)\nn (0,0.8)circle(1.8)
    (0,-.1)\ww(0,1.8)\ww;
\draw[\qh](0,-.1)\ww to[bend left=60](0,1.8)\ww;
\draw[red](0,-.1)\ww to[bend left=-60](0,1.8)\ww;
\end{scope}
\end{tikzpicture}
\caption{Tagged rotation as forward flip and the composition as a negative braid twist}
\label{fig:tr1}
\end{figure}

\begin{lemma}
    Under the identification $\mathbb{I}$ in Lemma~\ref{lem:s to t}, the forward flip in $\EG(\sov)$ corresponds to the tagged rotation in $\EGx(\sov)$. In particular, Definition~\ref{def:a.l.2} is well-defined.
\end{lemma}

\begin{proof}
    This can be deduced by comparing Figure~\ref{fig:tr1} with Figures~\ref{fig:flip}, \ref{fig:S2.lift} and \ref{fig:S2.lift2}.
\end{proof}

In combination with \Cref{lem:1}, the covering \eqref{eq:forgetv} in \Cref{thm:cover}
can alternatively be written as
\begin{equation}\label{eq:forgetx}
    F_*^\P:\EGx(\sov) \xrightarrow{/\SBr(\sov)} \EGx(\smv)
\end{equation}

\section{Braid twist group as cluster braid group}\label{sec:BT=CT}

\subsection{Dual triangulations of DMSv}
\begin{definition}
A \emph{dual triangulation} of DMSv $\sov$ is the dual graph associated to some tagged triangulation $\Tx$, denoted by $\Tx^*$, which consists of closed arcs in $\CA(\sov)$.

The dual $\Te^*$ of a signed triangulation $\Te\in\EG(\sov)$ is defined to be the dual triangulation with respect to the tagged triangulation $\mathbb{I}(\Te)$.
It can be checked that $\Te^*$ is well-defined by checking the local pictures (e.g.
the left picture of \Cref{fig:tr1} and $\eta_1$ and $\eta_2$ in \Cref{fig:switch2}).
\end{definition}

\begin{example}[Naive dual of $\Te$]
For a self-folded edge $\gamma_0$ of $\T$ (e.g., $\gamma_1$ or $\gamma_2$ in Figure~\ref{fig:switch2}), let $\gamma$ be the loop that encloses $\gamma_0$. Then the dual arc $\gamma_0^*$ can be interpreted as `difference' of the graph dual arcs corresponding to $\gamma_0$ and $\gamma$. Namely, in Figure~\ref{fig:switch2}, we have
\begin{itemize}
  \item $\gamma_1^*=\eta,\, \gamma^*=\eta_2$
    for the left picture/triangulation and
  \item $\gamma_2^*=\eta,\, \gamma^*=\eta_1$
    for the right picture/triangulation.
\end{itemize}
\end{example}

\begin{figure}[htpb]\centering
\begin{tikzpicture}[xscale=1.2,yscale=1.2]
\begin{scope}[shift={(0,0)}]
\draw[blue](0,.8)to[bend left=45](0,-1)
    (0.52,0)node{$_{\gamma_1}$};
\draw[\hh] plot [smooth,tension=1.6] coordinates
    {(0,-1) (-1,1) (1,1) (0,-1)};
\draw[red,very thick,dashed](0,-.1).. controls +(30:1.2) and +(-30:1.2) ..(0,2);
\draw[\hhd,ultra thick](0,-.1).. controls +(60:2) and +(120:2) ..(0,-.1)
    (0,1.3)node{$_\eta$}
    (2,-1)node{$\eta=|\eta_1-\eta_2|$};
\draw[\qh,ultra thick](0,-.1).. controls +(150:1.2) and +(-150:1.2) ..(0,2)
    (-.6,1)node{$_{\eta_2}$}(0,-1);
\draw(0,.8)\vot(0,-1)\nn(0,2.6)\nn (0,0.8)circle(1.8)(0,-.1)\ww(0,2)\ww;
\draw[\hh](2,2.5)node{$\gamma=\gamma_1+\gamma_2$}(-1.1,0)node{$_\gamma$};
\end{scope}
\draw(2,.8)node{$=$};
\begin{scope}[shift={(4,0)}]
\draw[\hh](1.1,0)node{$_\gamma$};
\draw[\ql](0,.8)to[bend left=-45](0,-1)
    (-0.52,0)node{$_{\gamma_2}$};;
\draw[\hh] plot [smooth,tension=1.6] coordinates
    {(0,-1) (-1,1) (1,1) (0,-1)};
\draw[red,ultra thick](0,-.1).. controls +(30:1.2) and +(-30:1.2) ..(0,2)
    (.6,1)node{$_{\eta_1}$}(0,-1);
\draw[\hhd,ultra thick](0,-.1).. controls +(60:2) and +(120:2) ..(0,-.1)
    (0,1.3)node{$_\eta$};
\draw[\qh,very thick,dashed](0,-.1).. controls +(150:1.2) and +(-150:1.2) ..(0,2);
\draw(0,.8)\tov(0,-1)\nn(0,2.6)\nn (0,0.8)circle(1.8)(0,-.1)\ww(0,2)\ww;
\end{scope}
\end{tikzpicture}
\caption{Dual triangulations (which are equivalent)}
\label{fig:switch2}
\end{figure}
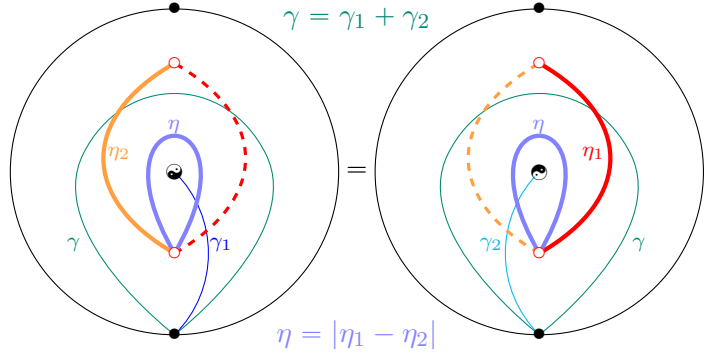

The following is a straightforward observation.

\begin{lemma}\label{lem:dualtri}
Let $\T=\{\gamma_i\}$ be a triangulation of $\smv$ with dual $\T^*=\{\eta_i\}$.
For any $i\ne j$:
\begin{itemize}
\item $\eta_i,\eta_j$ are disjoint if and only if $\gamma_i$ and $\gamma_j$ are not adjacent in any triangle of $\T$;
\item $\eta_i,\eta_j$ form a once-vortexed digon
        if and only if the triangles containing them are of the form in one of the cases in \Cref{fig:tr1};
\item $\eta_i,\eta_j$ are disjoint except sharing a common endpoint if and only if $\gamma_i$ and $\gamma_j$ are adjacent in exactly one non-self-folded triangle of $\T$.
\end{itemize}
In particular,
\begin{itemize}
  \item the first two cases occur if and only if there is no arrow between vertices $i$ and $j$;
  \item the last case occurs if and only if there is exactly one arrow between vertices $i$ and $j$,
\end{itemize}
in the corresponding quiver $Q_{\T}$.
\end{lemma}

\begin{definition}\label{def:BT2}
Let $\Tx$ be a tagged triangulation of $\sov$. The \emph{braid twist group} $\BT(\Tx^*)$ of the triangulation $\Tx$ is the subgroup of $\MCG(\sov)$ generated by the braid twists $\Bt{\eta}$ for the closed arcs $\eta$ in $\Tx^*\subset\CA(\sov)$.
\end{definition}

The following is a generalization of \cite[Lem.~2.2]{QZ2}, cf. also \cite[Lem.~3.14]{QQ}.

\begin{lemma}\label{lem:cut}
    Let $\Tx$ be a tagged triangulation of $\sov$ and $\Tx^*$ its dual. Let $\eta$ be a closed arc in $\CA(\sov)\setminus\Tx^*$. Then there are two arcs $\alpha,\beta\in\CA(\sov)$ such that
    \begin{enumerate}
        \item[(1)] $\alpha$ and $\beta$ do not intersect in $\surf^\circ-\Tri$;
        \item[(2)] $\eta=B_{\alpha}(\beta)$ or $B_{\alpha}^{-1}(\beta)$;
        \item[(3)] $\Int(\gamma,\alpha)<\Int(\gamma,\eta)$ and $\Int(\gamma,\beta)<\Int(\gamma,\eta)$ for any $\gamma\in\Tx$.
    \end{enumerate}
\end{lemma}

\begin{proof}
    Assume that $\eta$ is in a minimal position with $\Tx$. According to the proof of \cite[Lem.~2.2]{QZ2}, we may assume that $\eta$ intersects exactly two regions of $\Tx$ which share exactly two edges. Then there are three cases: (1) the two regions from an annulus, (2) the two regions from a digon, (3) exactly one of them is a digon, corresponding to Cases~6, 4.2$^\circ$ and 4.3$^\circ$ in \Cref{def:gpd.b}. Case (1) is proved in \cite[Lem.~2.2]{QZ2}, while the other two cases are similar.
\end{proof}

A direct consequence of \Cref{lem:cut} is the following generalization of \cite[Prop.~2.3]{QZ2}.

\begin{lemma}
For any tagged triangulation $\Tx$ of $\sov$, we have $\BT(\sov) = \BT(\Tx^*)$ in $\MCG(\sov)$.
\end{lemma}

\begin{lemma}\label{lem:neg bt}
Let $\Tx\in\EGx(\sov)$ and tagged arc $\gamma\in\Tx$ with dual closed arc $\eta\in\Tx^*$.
Then the composition of forward flips of $\Tx$ with respect to $\gamma$ and $\gamma^\sharp$
is the negative braid twist $\bt{\eta}(\Tx)$ of $\Tx$ in $\EGx(\sov)$.
\end{lemma}
\begin{proof}
There are several cases.
\begin{itemize}
    \item If neither $\gamma$ nor $\gamma^\sharp$ is an edge of a digon,
    this reduces locally to the unpunctured case, which is well-known (see \cite[Fig.~]{QQ} or \cite[Fig.~3.14]{KQ2}).
    \item If $\gamma'$ is an edge of a digon, see the first row of Fig.~\ref{fig:tr2}.
    \item If $\gamma$ is an edge of a digon, see the second row of Fig.~\ref{fig:tr2}.\qedhere
\end{itemize}
\end{proof}

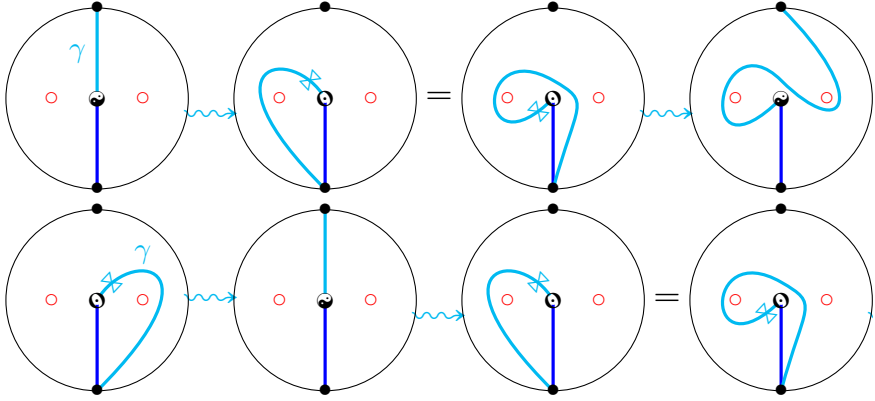
\begin{figure}[htpb]\centering
\begin{tikzpicture}[xscale=1.2,yscale=1.2]\clip(-1,-1.1) rectangle (8.5,1.1);
\draw[\ql](1.25,-.2)node{\Large{$\rightsquigarrow$}}
    (6.25,-.2)node{\Large{$\rightsquigarrow$}};
\draw[](0,0)circle(1);
\draw[\ql,very thick](0,1)tonode[left]{$\gamma$}(0,0);
\draw[blue,very thick](0,0)to(0,-1);
\draw(0,1)\nn(0,-1)\nn;
\draw(0,0)\vot;
\draw[red](.5,0)\ww(-.5,0)\ww;

\begin{scope}[shift={(2.5,0)}]
\draw[\ql](-.16,.2)node[rotate=43]{$\bowtie$};
\draw[blue,very thick](0,0)to(0,-1);
\draw[\ql,very thick](0,0)
    .. controls +(120:.7) and +(135:2) ..(0,-1);
\draw(1.25,0)node{\Large{$=$}};
\draw(0,0)circle(1)(0,1)\nn(0,-1)\nn;
\draw[red](.5,0)\ww(-.5,0)\ww;
\draw(0,0)\dpole;
\end{scope}
\begin{scope}[shift={(5,0)}]
\draw[blue,very thick](0,0)to(0,-1);
\draw[\ql](-.16,-.14)node[rotate=-43]{$\bowtie$};
\draw[\ql,very thick](0,-1)
    .. controls +(75:1) and +(-30:.5) ..(0,0.2)
    .. controls +(150:1) and +(-135:1.2) ..(0,0);
\draw(0,0)circle(1)(0,1)\nn(0,-1)\nn;
\draw[red](.5,0)\ww(-.5,0)\ww;
\draw(0,0)\dpole;
\end{scope}

\begin{scope}[shift={(7.5,0)}]
\draw[blue,very thick](0,0)to(0,-1);
\draw[\ql,very thick](0,1)
    .. controls +(-45:1.4) and +(-45:1) ..(0,0.2)
    .. controls +(135:1) and +(-135:1.4) ..(0,0);
\draw(0,0)circle(1)(0,1)\nn(0,-1)\nn;
\draw[red](.5,0)\ww(-.5,0)\ww;
\draw(0,0)\vot;
\end{scope}
\end{tikzpicture}
\begin{tikzpicture}[xscale=1.2,yscale=1.2]\clip(-1,-1.1) rectangle (8.5,1.1);
\begin{scope}[shift={(0,0)}]
\draw[\ql](.16,.2)node[rotate=-43]{$\bowtie$};
\draw[blue,very thick](0,0)to(0,-1);
\draw[\ql,very thick](0,0)
    .. controls +(60:.7) and +(45:2) ..(0,-1);
\draw[\ql] (.5,.5)node{$\gamma$};
\draw[\ql](1.25,0)node{\Large{$\rightsquigarrow$}};
\draw(0,0)circle(1)(0,1)\nn(0,-1)\nn;
\draw[red](.5,0)\ww(-.5,0)\ww;
\draw(0,0)\dpole;
\end{scope}
\begin{scope}[shift={(2.5,0)}]
\draw[\ql](1.25,-.2)node{\Large{$\rightsquigarrow$}}
    (6.25,-.2)node{\Large{$\rightsquigarrow$}};
\draw[](0,0)circle(1);
\draw[\ql,very thick](0,1)to(0,0);
\draw[blue,very thick](0,0)to(0,-1);
\draw(0,1)\nn(0,-1)\nn;
\draw(0,0)\vot;
\draw[red](.5,0)\ww(-.5,0)\ww;
\end{scope}
\begin{scope}[shift={(5,0)}]
\draw[\ql](-.16,.2)node[rotate=43]{$\bowtie$};
\draw[blue,very thick](0,0)to(0,-1);
\draw[\ql,very thick](0,0)
    .. controls +(120:.7) and +(135:2) ..(0,-1);
\draw(1.25,0)node{\Large{$=$}};
\draw(0,0)circle(1)(0,1)\nn(0,-1)\nn;
\draw[red](.5,0)\ww(-.5,0)\ww;
\draw(0,0)\dpole;
\end{scope}
\begin{scope}[shift={(7.5,0)}]
\draw[blue,very thick](0,0)to(0,-1);
\draw[\ql](-.16,-.14)node[rotate=-43]{$\bowtie$};
\draw[\ql,very thick](0,-1) .. controls +(75:1) and +(-30:.5) ..(0,0.2) .. controls +(150:1) and +(-135:1.2) ..(0,0);
\draw(0,0)circle(1)(0,1)\nn(0,-1)\nn;
\draw[red](.5,0)\ww(-.5,0)\ww;
\draw(0,0)\dpole;
\end{scope}
\end{tikzpicture}
\caption{Composition of forward flips as a negative braid twist, up to $\LL$}
\label{fig:tr2}
\end{figure}

\subsection{The surjective map}\label{subsec:pre dual}\

On one hand, we provide a finite presentation of $\BT(\sov)$ with respect to a specific dual triangulation.

Let $\T_0 = {\gamma_i}$ be the tagged triangulation of $\sop$ shown in \Cref{fig:tri}
with dual $\T_0^* = {\eta_i}$.
Note that the red bullet there means possibly a collection of decorations (corresponding to multiple marked points on a boundary component, cf. right picture in \cite[Fig.~13]{QZ2}).
Observe that $\RT^*$ coincides with the set of generators of $\BT(\sov)$ described in \Cref{prop:alt pre}, except that the generator
$$
    x_{2g+b-1+\sum_{i=2}^b(m_i-1)} \quad \text{($=a_0$ shown in \Cref{fig:tri}) }
$$
is replaced by $a'=a_0^{a_1\kg\cdots\kg a_p}$.
The relations in \Cref{prop:alt pre} can then be directly translated into relations for $\BT(\T_0^*)$, which retain the same form.
Therefore, this yields a finite presentation of $\BT(\T_0^*)$.

\begin{figure}[htpb]
    \centering
    \includegraphics[width=11cm]{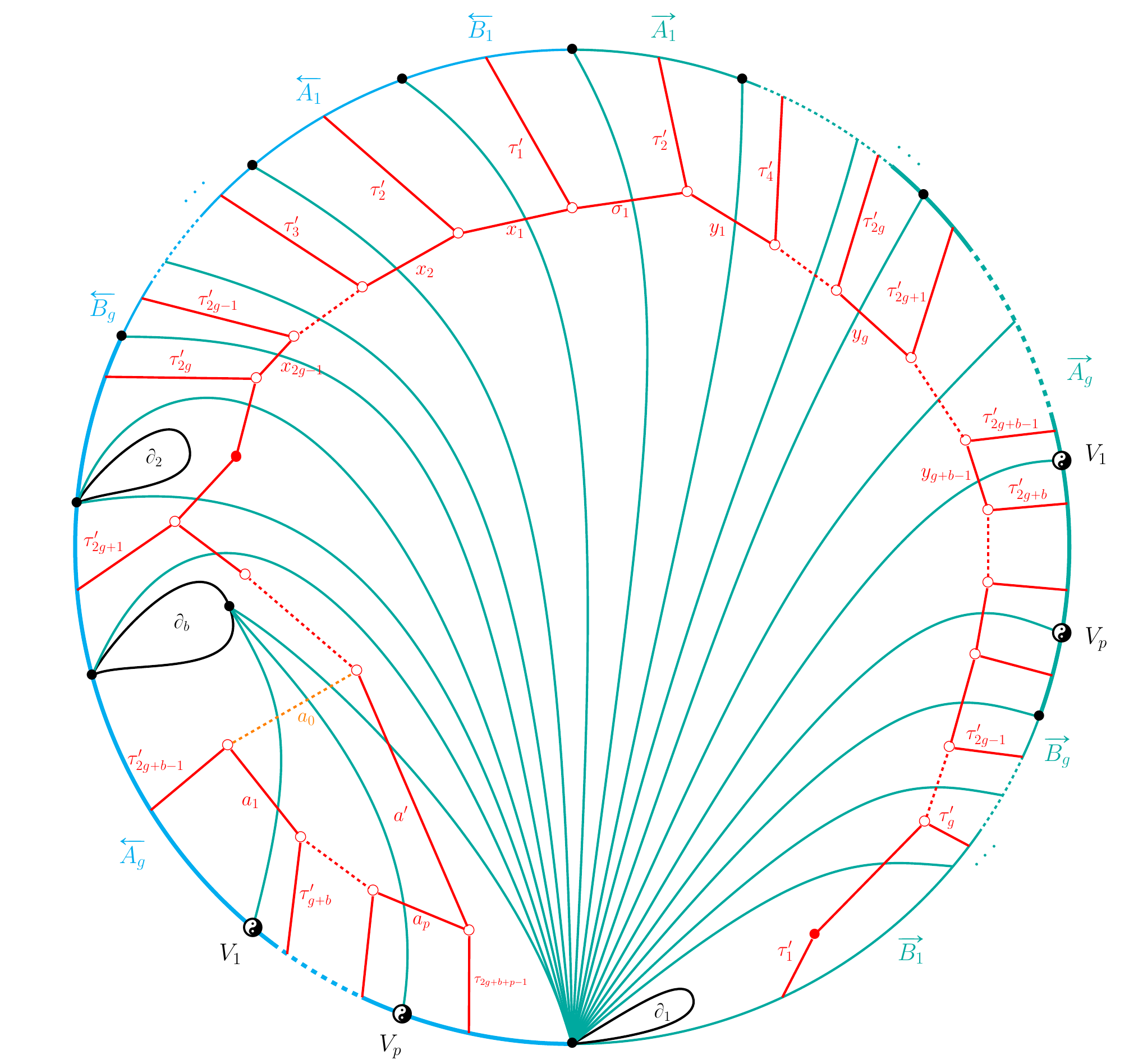}
    \caption{A triangulation of $\sop$, where $a_j=x_{2g+b-1+\sum_{i=2}^b(m_i-1)+j-1}$.}
    \label{fig:tri}
\end{figure}

On the other hand,
we consider the local twists/standard generators $\{t_i=t_{\gamma_i}\}$ of the cluster braid group $\CBr(\RT_0)$ for $\RT_0=F_*^\P(\T_0)$,
introduced in \Cref{sec:ceg MSv}.

\begin{lemma}\label{prop:cbg}
    The local twists $t_i$'s satisfy the relations in \Cref{prop:alt pre}, i.e., for any $\eta_i,\eta_j,\eta_k,\eta_l\in \T^*$,
    \begin{itemize}
        \item if $\eta_i,\eta_j$ are disjoint, then the commutation relation $\Co(t_i,t_j)$ holds,
        \item if $\eta_i,\eta_j$ are disjoint except sharing a common endpoint, then the braid relation $\Br(t_i,t_j)$ holds,
        \item if $\eta_i,\eta_j,\eta_k$ are disjoint except for sharing a common endpoint, and they are in clockwise order at that point, then the triangle relation $\on{Tr}(t_i,t_j,t_k)$ holds,
        \item if $\eta_i,\eta_j,\eta_k,\eta_l$ form a rectangle in the anti-clockwise order with exactly one puncture in the interior, then the rectangle relation $\on{Rec}(t_i,t_j,t_k,t_l)$ holds.
    \end{itemize}
\end{lemma}
\begin{proof}
The first two cases are \Cref{lem:cobr}.

In the third case, we flip $\gamma_i$ to get a new triangulation $\T'_0$,
see the first row of the left picture of \Cref{fig:flip}.
Since $\eta'_j$ and $\eta'_k$ are disjoint, by applying \Cref{lem:cobr} to $\T'_0$,
the commutation relation $\Co(t'_j,t'_k)$ holds.
Taking conjugation, by \cite[Prop.~2.8]{KQ2}, the braid relation $\Co({t_j},{t_k}^{t_i})$ holds, which implies the triangle relation $\on{Tr}(t_i,t_j,t_k)$.

In the fourth case, we flip $\T_0$ at $\gamma_i$ and $\gamma_k$ successively to get a new triangulation $\T_0''$. Then $\eta''_j$ and $\eta''_l$ form a digon with one vortex in the interior. By applying \Cref{lem:cobr} to $\T_0''$, the commutation relation $\Co(t''_j,t''_l)$ holds. Taking conjugation again, the braid relation $\Co({t_j}^{t_k},{t_l}^{t_i})$ holds, which implies the rectangle relation $\on{Rec}(t_i,t_j,t_k,t_l)$.
\end{proof}

\begin{corollary}\label{cor:surj}
There is a canonical surjective map
\begin{gather}\label{eq:can}
    c_{\Tx}\colon \BT(\Tx^*) \to \CBr( F_*^\P(\Tx^*) ),
\end{gather}
which sends the standard generators to the standard ones.
\end{corollary}

\begin{proof}
For $\Tx=\T_0$, the claim follows from \Cref{prop:cbg}.
Then we only need to show that the claim still holds after a (forward) flip.
This follows directly from the conjugating formula for change of generators,
i.e. \cite[Prop.~2.8]{KQ2} for cluster braid groups that have been used above,
and the corresponding formula for braid twist group by direct checking.
\end{proof}

\subsection{The Galois covering}

\begin{figure}[htpb]\centering
    \includegraphics[width=11cm]{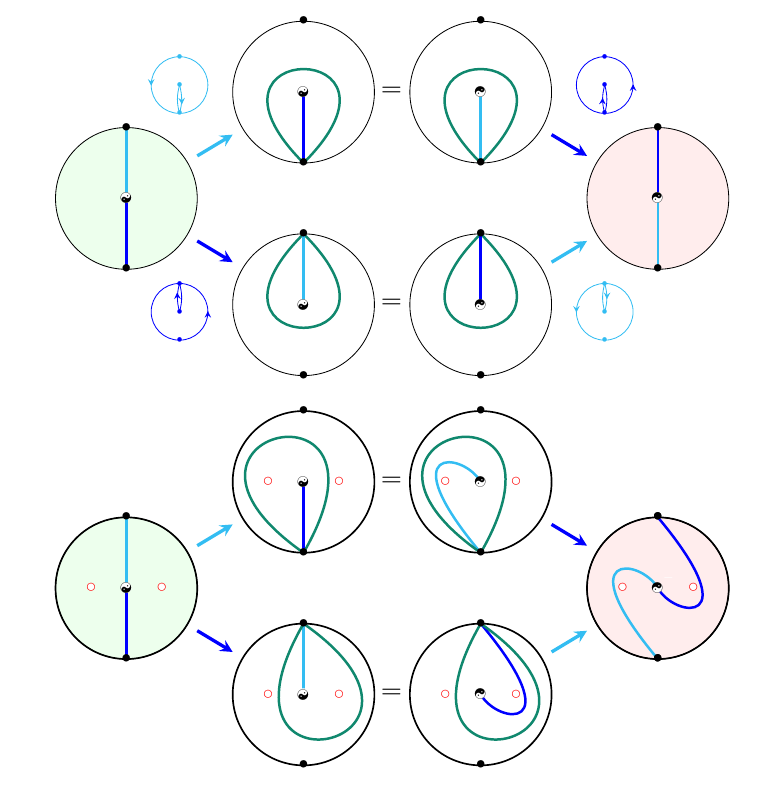}
\caption{Type $4.2^\circ$ square relations and its lift}
\label{fig:S2.lift}
\end{figure}

\begin{figure}[htpb]\centering
    \includegraphics[width=11cm]{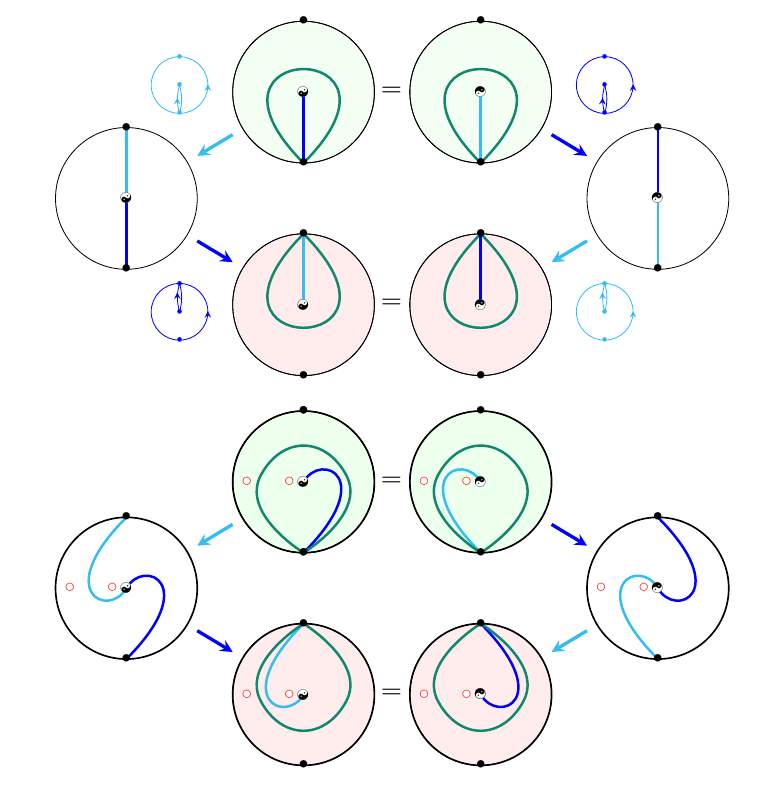}
\caption{Type $4.3^\circ$ square relations and its lift}
\label{fig:S2.lift2}
\end{figure}

\begin{definition}\label{def:gpd.b}
The \emph{exchange groupoid} $\eg(\sov)$
is the quotient of the path groupoid of $\EG(\sov)$ by the following
relations, starting at any triangulation $\T$ in $\EG(\sov)$:
\begin{itemize}
\item[Sq-1.] If two open arcs are not adjacent
  in any triangle of $\T$, then
  the forward flips with respect to them form a square in $\EG(\sov)$, as in \cite[Fig.~12]{KQ2},
  and we impose the commuting square relation.
\item[Sq-2.] If two open arcs are adjacent
  in two triangles of $\T$ that form a digon, then
  the forward flips with respect to them form a square in $\EG(\sov)$,
  as in the lower pictures of Figure~\ref{fig:S2.lift},
  and we impose the commuting square relation.
  Note that in the figure, the green/red disks are the source/sink of the square.
\item[Sq-3.] If two open arcs are adjacent
  in two triangles of $\T$ that form a self-folded triangle, then
  the forward flips with respect to them form a square in $\EG(\sov)$,
  as in the lower pictures of Figure~\ref{fig:S2.lift2},
  and we impose the commuting square relation.
  Note that in the figure, the green/red disks are the source/sink of the square.

\item[Pen.] If two open arcs are adjacent
  in only one triangle of $\T$, then they induce an oriented pentagon in $\EG(\sov)$, as in \cite[Fig.~13]{KQ2},
  and we impose the corresponding commuting pentagon relation;
\item[Hex.] If two open arcs are adjacent
  in two triangles of $\T$ that form an annulus,
  then they induce an oriented hexagon in $\EG(\sov)$, as in \cite[Fig.~14]{KQ2},
  and we impose the corresponding commuting hexagon relation.
\end{itemize}
\end{definition}

By checking that $F_*^\P$ in \eqref{eq:forgetv} preserves relations,
we obtain an upgraded groupoid version.

\begin{lemma}\label{lem:Kr}
The $\SBr(\sov)$-covering \eqref{eq:forgetv} upgrades to a $\SBr(\sov)$-covering functor $F_*^\P\colon\eg(\sov)\to\eg(\surfv)$.
\end{lemma}

Fix an initial triangulation $\Te$ in $\EG(\sov)$.
Let $\EGT(\sov)$ and $\egt(\sov)$ be the corresponding connected components in
$\EG(\sov)$ and $\eg(\sov)$, respectively.

\begin{theorem}\label{thm:BT}
The restricted covering $F_*^\P \colon \egt(\sov)\to \eg(\surfv)$
is a Galois $\BT(\sov)$-covering.
\end{theorem}

\begin{proof}
$\BT(\sov)$ preserves the fibers of the restricted covering by Lemma~\ref{lem:Kr}.
Thus, we only need to show that if $F_*^\P(\T_1,\sign_1)=F_*^\P(\T_2,\sign_2)$ for $(\T_i,\sign_i)$ in $\egt(\sov)$,
then they are related by the action of $\BT(\sov)$.

Since $\egt(\sov)$ is connected, there is a path $\widehat{p}:(\T_1,\sign_1)\to (\T_2,\sign_2)$. Let $p=F_*^\P(\widehat{p})\in\pi_1(\eg(\surfv),\RTe)$.
Thus, by Corollary~\ref{cor:FST},
$p$ is a product of local twists at $F_\ast(\T_1,\sign_1)=F_\ast(\T_2,\sign_2)$.
Lifting these local twists to $\egt(\sov)$ yields braid twists, which map $(\T_1, \sign_1)$ to $(\T_2, \sign_2)$. This completes the proof.
\end{proof}

\subsection{Universal covering}

\begin{theorem}\label{thm:BT=CT}
Let $\Tx$ be a tagged triangulation of $\smv$ and $\RTx=F_*(\Tx)$. Then $\egt(\sov)$ is simply connected and hence $F_*\colon \egt(\sov)\to \eg(\surfv)$ in Theorem~\ref{thm:BT} is the universal cover with Galois group
\begin{gather}\label{eq:BT=CT}
    \CBr(\RTx)=\pi_1(\eg(\smv),\RTx)\cong\BT(\sov).
\end{gather}
\end{theorem}

\begin{proof}
Let $\Tx=\{\gamma_i\}$ with dual $\Tx^*=\{\eta_i\}$.

By the covering map $F_*^\P$ in \Cref{thm:BT},
there is a short exact sequence of groups
$$
    1\to\pi_1(\egt(\sov),\Tx)\to\pi_1(\eg(\smv),\RTx)\xrightarrow{i_{\Tx}} \BT(\sov)\to 1.
$$
By \Cref{cor:FST}, \cite[Prop.~2.9]{KQ2} and isomorphism \eqref{eq:iso}, we have
$$
    \CBr(\smv)=\pi_1(\eg(\smv),\RTx).
$$
So $\pi_1(\eg(\smv),\RTx)$ is generated by the local twists.
So the map $i_{\Tx}$ is determined by sending the standard generators $t_{F(\gamma_i)}$ to the standard ones $B_{\eta_i}^{-1}$, cf. \Cref{lem:neg bt}.
By \Cref{cor:surj}, $i_{\T}$ is the inverse of $\mathrm{can}$ in \eqref{eq:can}.
Therefore,  $\pi_1(\eg(\smv),\RTx)\cong\BT(\sov)$ and $\pi_1(\egt(\sov),\Tx)=1$.
\end{proof}

\section{Application to moduli spaces of quadratic differentials}\label{part:D}

\subsection{Preliminaries}\
\paragraph{\textbf{Quadratic differentials}}\

We recall standard notions concerning quadratic differentials, see \cite{BS} and \cite{KQ2,Q24} for more details. Let $\rs$ be a compact Riemann surface and $\omega_\rs$ be its holomorphic cotangent bundle. A \emph{meromorphic quadratic differential} $\phi$ on $\rs$ is a meromorphic section of the line bundle $\omega_{\rs}^{2}$. In terms of a local coordinate $z$ on $\rs$, such a $\phi$ can be written as $\phi(z)=g(z)\, \dd z^2$, where $g(z)$ is a meromorphic function. The zeros and simple poles of $\phi$ are called \emph{finite singularities}, while poles of order at least two are called \emph{infinite singularities}.

\begin{definition}
A \emph{GMN differential} $\phi$ on $\rs$ is a meromorphic quadratic differential such that
\begin{itemize}
\item all zeros of $\phi$ are simple,
\item $\phi$ has at least one higher order (at least two) pole,
\item $\phi$ has at least one finite critical point.
\end{itemize}
Denote by $\Zer(\phi)$ the set of zeros of $\phi$,
$\Pol_j(\phi)$ the set of poles of $\phi$ with order $j$
and $\Pol(\phi)=\coprod_j\Pol_j(\phi)$.
\end{definition}

A GMN differential $\phi$ on $\rs$ determines the $\phi$-metric on $\surp$, which is defined locally
by pulling back the Euclidean metric on $\CC$ using a distinguished coordinate $\omega$,
cf.  \cite[Fig.~16]{Q24}.
Thus, there are geodesics on $\surp$, each of which has a constant phase with respect to $\omega$.
A \emph{trajectory} of a GMN differential $\phi$ on $\surp$
is a maximal horizontal geodesic $\gamma\colon(0,1)\to\surp$,
with respect to the $\phi$ metric.
We will usually consider horizontal trajectories unless otherwise specified.

\paragraph{\textbf{WKB triangulations}}\

We will mainly consider the following types of trajectories of a GMN differential $\phi$:
\begin{itemize}
\item \emph{saddle trajectories}, whose both ends are finite singularities;
\item \emph{separating trajectories} that connect a finite singularity and an infinite singularity;
\item \emph{generic trajectories} whose both ends are infinite singularities.
\end{itemize}
We say a GMN differential is \emph{saddle-free} if there are no saddle trajectories.
In such a case, by removing all separating trajectories (which are finitely many) from $\surp$, the remaining open surface splits as a disjoint union of
\begin{itemize}
\item a \emph{half-plane (infinite height horizontal strip)}, i.e., isomorphic to
$\{z\in \CC \mid \Imgy(z)>0\}$ equipped with the differential $\dd z^{2}$.
It is swept out by generic trajectories which connect a fixed pole to itself.
\item or a \emph{(finite height) horizontal strip}, i.e. isomorphic to
$\{z\in \CC\mid a<\Imgy(z)<b\}$ equipped with the differential $\dd z^{2}$ for some $a<b \in \RR$. It is swept out by generic trajectories connecting two (not necessarily distinct) poles.
\end{itemize}
We call this union the \emph{horizontal strip decomposition} of $\rs$ with respect to $\phi$.
In each horizontal strip, the trajectories are isotopic to each other
and there is a unique geodesic, the \emph{saddle connection}, connecting the two zeros on its boundary,
which is a trajectory with a certain fixed phase (other than zero).

A GMN-differential is called \emph{complete} if $\Pol_1(\phi)=\emptyset$.
\begin{definition}
When $\phi$ is saddle-free, there is a \emph{WKB triangulation} $\RT_\phi$ on $\surf$ induced from $\phi$,
where the arcs are (isotopy classes of inherited) generic trajectories.
Moreover, any triangle in $\RT_\phi$ either contains exactly one zero or is a self-folded triangulation with a simple pole as interior vertex.
If $\phi$ is in addition complete, then $\RT_\phi$ becomes a decorated triangulation $\T_\phi$ of $\surfo$,
with dual graph consisting of saddle trajectories.
\end{definition}

\paragraph{\textbf{Real blow-up}}\

\begin{construction}\label{cons:DMSV}
Let $\phi$ be a GMN differential on a Riemann surface $\rs$.
The \emph{real (oriented) blow-up} of $(\rs,\phi)$ is a DMSp $\rs^\phi_{\Zer(\phi)}$ constructed as follows:
\begin{itemize}
  \item $\rs^\phi$ is obtained from the underlying differentiable surface by replacing each pole $p\in\Pol(\phi)$
  with order $|\on{ord}_\phi(p)|\ge3$ by a boundary $\partial_p$,
  where the points on the boundary correspond to the real tangent directions at $p$.
  \item We mark the points on $\partial_p$ that correspond to the distinguished tangent directions. Thus, there are $|\on{ord}_\phi(p)|-2$ marked points on $\partial_p$.
  \item The simple and double poles become punctures.
  \item The (simple) zeros become decorations.
\end{itemize}
The generic trajectories become open arcs on $\rs^\phi_{\Zer(\phi)}$ and the saddle trajectories become closed arcs on $\rs^\phi_{\Zer(\phi)}$.
Note that there is extra information on the set of punctures $\Pol_{1,2}(\phi)$: it naturally divides into two subsets $\Pol_1(\phi)\cup\Pol_2(\phi)$.

If one forgets the zeros, one gets a MSp $\rs^\phi$.
\begin{figure}[htpb]\centering
\begin{tikzpicture}[yscale=.15,xscale=.08]
\foreach \j in {0,...,10}
    {\draw[Emerald!23, very thick](-20,\j)to(30,\j);}
\foreach \j in {0,10}
    {\draw[\sepa,ultra thick](-20,\j)to(30,\j);}
\draw[red,thick](0,0)\ww to(7,10)\ww;
\draw[blue!50,thick](3.5,0)to[bend left=-30]node[right]{$^{angle}$}(2,3);
\begin{scope}[shift={(0,13)}]
\foreach \j in {0,...,10}
    {\draw[Emerald!23, very thick](-20,\j)to(30,\j);}
\foreach \j in {0}
    {\draw[\sepa,ultra thick](-20,\j)to(30,\j);}
\draw[red, very thick](0,0)\ww;
\end{scope}
\end{tikzpicture}\qquad
\begin{tikzpicture}[scale=.5]
\foreach \k in {1,2,0}
{    \path (120*\k-30:4.5) coordinate (v2)
          (120*\k+120-30:4.5) coordinate (v1)
          (120*\k+60-30:2.25) coordinate (v3)
          (0,0) coordinate (v4);
  \foreach \j in {.25,.4,.55,.7,.85,1}
    {
      \path (v4)--(v3) coordinate[pos=\j] (m0);
      \draw[Emerald,thin] plot [smooth,tension=.5] coordinates {(v1)(m0)(v2)};
    }
  \draw[Emerald, ultra thick](v1)to(v2);
  \draw[thick,\sepa](v2)\nn to(v4)to(v1)\nn;
}
\draw(0,0)\ww;
\begin{scope}[rotate=30,shift={(0:4.5)}]
\foreach \k in {1,2,0}
{    \path (120*\k:4.5) coordinate (v2)
          (120*\k+120:4.5) coordinate (v1)
          (120*\k+60:2.25) coordinate (v3)
          (0,0) coordinate (v4);
  \foreach \j in {.25,.4,.55,.7,.85,1}
    {
      \path (v4)--(v3) coordinate[pos=\j] (m0);
      \draw[Emerald,thin] plot [smooth,tension=.5] coordinates {(v1)(m0)(v2)};
    }
  \draw[Emerald, ultra thick](v1)to(v2);
  \draw[thick,\sepa](v2)\nn to(v4)to(v1)\nn;
}
\end{scope}
\draw[red, very thick](0,0)\ww to(30:4.5)\ww;
\end{tikzpicture}

\begin{tikzpicture}[scale=2.0]
\begin{scope}[shift={(0,0)}]
\foreach \j in {.45,.35,.25,.15}
{\draw[Emerald,thick](0,0)circle(\j);}
\draw[red,very thick] plot [smooth,tension=1] coordinates
    {(0,-.7) (180:.55) (90:.55) (0:.55)  (0,-.7)};
\draw[ultra thick,\sepa](0,-.7)\ww to(0,-1);
\draw[ultra thick,Emerald] plot [smooth,tension=1] coordinates
    {(0,-1) (180:.8) (90:.8) (0:.8)  (0,-1)};
\draw[Emerald,thick] plot [smooth,tension=1] coordinates
    {(0,-1) (-.15,-.75) (175:.7) (90:.7) (5:.7) (.15,-.75) (0,-1)};
\draw(0,0)\dpole (0,-1)\nn;\draw(0,1)node{};
\end{scope}

\begin{scope}[shift={(2,0)}]
\draw[ultra thick, Emerald] plot [smooth,tension=1] coordinates
    {(0,-1) (180:.8) (90:.8) (0:.8)  (0,-1)};
\foreach \j in {.5,.6,.7,.38}
{\draw[Emerald,thick] plot [smooth,tension=.9] coordinates
    {(0,-1) (180:\j) (90:\j) (0:\j)  (0,-1)};}
\draw[Emerald,thick] (0,-1)
    .. controls +(63:1.9) and +(117:1.9) .. (0,-1)
    .. controls +(73:1.6) and +(107:1.6) .. (0,-1);
\draw[thick,Emerald](0,0)\dpole to (0,-1)\nn;
\draw[thick](0,0)\dpole (0,-1)\nn;
\draw(0,0)\spole;
\end{scope}

\begin{scope}[shift={(4,0)}]
\draw[ultra thick, Emerald] plot [smooth,tension=1] coordinates
    {(0,-1) (180:.8) (90:.8) (0:.8)  (0,-1)};
\foreach \j in {.6,.7}
{\draw[Emerald,thick] plot [smooth,tension=1] coordinates
    {(0,-1) (180:\j) (90:\j) (0:\j)  (0,-1)};}
\foreach \j in {.5}
{\draw[ultra thick,\sepa] plot [smooth,tension=1] coordinates
    {(0,-1) (180:\j) (90:\j) (0:\j)  (0,-1)};}
\draw[black,very thick](0,.5)\ww to(0,.01) coordinate (xx);
\draw[Emerald,thick] (0,-1)
    .. controls +(57:1.6) and +(75:.5) .. (xx);
\draw[Emerald,thick] (0,-1)
    .. controls +(63:1.4) and +(55:.4) .. (xx);
\draw[Emerald,thick] (0,-1)
    .. controls +(70:1.3) and +(30:.27) .. (xx);
\draw[Emerald,thick] (0,-1)
    .. controls +(75:1.2) and +(-30:.15) .. (xx);
\draw[Emerald,thick] (0,-1)
    .. controls +(80:1) and +(-60:.15) .. (xx);
\draw[Emerald,thick] (0,-1)
    .. controls +(85:.7) and +(-75:.2) .. (xx);

\draw[Emerald,thick] (0,-1)
    .. controls +(180-57:1.6) and +(180-75:.5) .. (xx);
\draw[Emerald,thick] (0,-1)
    .. controls +(180-63:1.4) and +(180-55:.4) .. (xx);
\draw[Emerald,thick] (0,-1)
    .. controls +(180-70:1.3) and +(180-30:.27) .. (xx);
\draw[Emerald,thick] (0,-1)
    .. controls +(180-75:1.2) and +(180--30:.15) .. (xx);
\draw[Emerald,thick] (0,-1)
    .. controls +(180-80:1) and +(180--60:.15) .. (xx);
\draw[Emerald,thick] (0,-1)
    .. controls +(180-85:.7) and +(180--75:.2) .. (xx);
\draw[ultra thick, Emerald](xx)\dpole to (0,-1)\nn;
\draw[ultra thick](xx)\dpole (0,-1)\nn;
\end{scope}
\end{tikzpicture}
\caption{Horizontal strips and WKB triangles}
\label{fig:self-folded}
\end{figure}
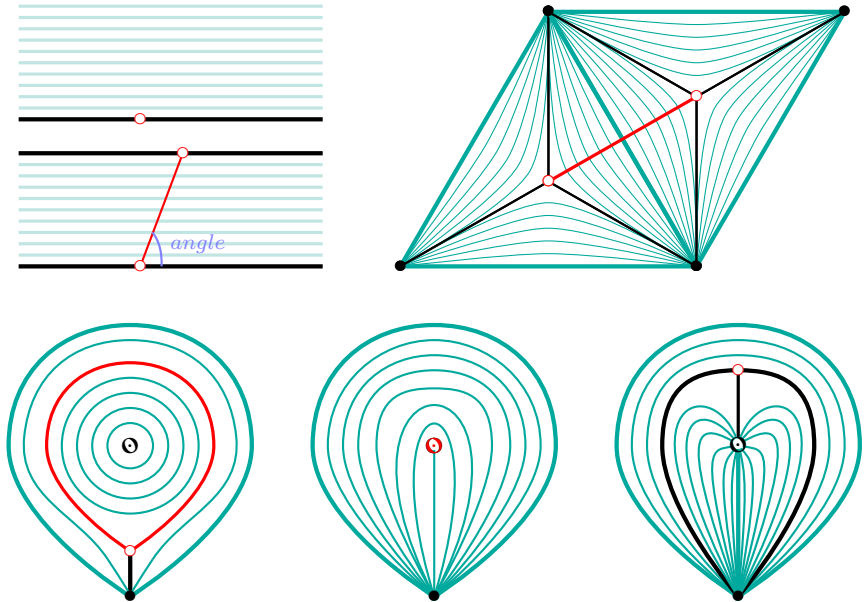
\end{construction}

\paragraph{\textbf{Moduli spaces with multi-strata}}\

A sign GMN differential $\phi^\sign$ is a GMN differential $\phi$ (on some Riemann surface) with a choice of sign at each double pole (of its residue), denoted by $\sign$.

\begin{definition}\label{def:Quad}
A \emph{$\surfv$-framed signed quadratic differential $\Psi^\sign=(\rs,\phi^\sign,\psi)$}
consists of a Riemann surface $\rs$ with a sign GMN differential $\phi^\sign$,
equipped with a diffeomorphism $\psi\colon\surfp \to\rs^\phi$,
preserving the marked points and punctures in the sense that
\begin{itemize}
  \item $\psi(\sun)=\Pol_{1,2}(\phi)$;
  \item $\psi$ sends marked points to marked points.
\end{itemize}

Two such data $(\rs_1,\phi_1^{\sign_1},\psi_1),(\rs_2,\phi_2^{\sign_2},\psi_2)$ are equivalent,
if there exists a biholomorphism $f\colon\rs_1\to\rs_2$ such that
\begin{itemize}
  \item $f^*(\phi_2^{\sign_2})=\phi_1^{\sign_1}$,
  \item $\psi_2^{-1}\circ f_*\circ\psi_1\in\Homeo_0(\surfp)$,
  preserving the marked points and punctures,
  where $f_*\colon\rs_1^{\phi_1}\to\rs_2^{\phi_2}$ is the induced diffeomorphism.
\end{itemize}
Denote by $\SQuad$ the moduli space of $\surfv$-framed signed quadratic differentials.
\end{definition}

Recall that the mapping class group $\MCG(\surfv)$ of $\surfv$ is defined to be
\[
    \MCG(\surfv)=\MCG(\surfp)\ltimes\ZZP,
\]
which naturally acts on signed triangulations and $\surfv$-framed signed quadratic differentials, cf. \cite{BQ,BS,Q24}.
The subgroup $\ZZP$ changes the signs at the corresponding puncture/double poles.
The moduli space appearing in \cite{BS} is the unframed version, i.e.
\[
    \Quad(\surfv)=\SQuad/\MCG(\surfv).
\]
The action is not free, but free on the complete part consisting of those differentials without simple poles.

\subsection{Teichm\"{u}ller framed version}\

Now we want to define the $\sov$-framed version of moduli space,
which is not trivial in this case since there are multi-strata that $\MCG(\sop)$ does not
act on $\Quad(\surfv)$ directly.

\begin{definition}
Let $\sov$ be a DMSv.
\begin{itemize}
\item
A \emph{collision pairing} $\us=(s_1,\ldots,s_d)$ on $\sop$ is a collection of $d(\le p)$ collision paths $s_i$ (connecting decorations $Z_i$ with punctures $p_i$ respectively),
such that
\begin{itemize}
  \item $s_i$ are simple and compatible (do not intersect).
  \item each puncture in $\uk_1$ appears as an endpoint of $s_i$  exactly once.
\end{itemize}
\item
The \emph{\emph{contracted surface}} $\sovs$ of $\sop$ w.r.t. $\us$ is a DMSp obtained from $\sop$ by contracting each path $s_i$ into an order $-1$ puncture,
where the other punctures have order $-2$.
So the number of decorations reduces from $\aleph$ to $\aleph-d$.
\item
A map $\cont\colon\sov\to \Sigma$
is called a \emph{collision contraction} with respect to $\us$,
if $\Sigma$ is homeomorphic to $\sovs$ as a weighted DMS.
such that
\begin{itemize}
  \item $\cont$ is a homeomorphism when restricted to $\sov\setminus\bigsqcup_{i=1}^{|\us|} s_i$.
  \item $\cont(s_i)$ equals an order $-1$ puncture.
\end{itemize}
\end{itemize}
\end{definition}

\begin{definition}\label{def:FQuad O}
An \emph{$\sov$-framed quadratic differential $(\rs,\phi^\sign,\psi)$}
consists of a sign GMN differential $\phi^\sign$ on a Riemann surface $\rs$,
equipped with a framing
\begin{gather}\label{eq:framing}
  \cont\colon\sov\to\rs^\phi_{\Zer(\phi)},
\end{gather}
which is a collision contraction with respect to some collision path $\us$, such that
\begin{itemize}
  \item $\cont(s_i)$ equals to a simple pole;
  \item $\cont(p)$ equals to a double pole for puncture $p$ that is not an endpoint of any $s_i$;
  \item $\psi$ sends decorations to decorations.
\end{itemize}

Two such differentials, denoted by $(\rs_j,\phi_j^{\sign_j},\widetilde{\psi}_{\us_j} )$ are equivalent, if there exists a biholomorphism $f\colon\rs_1\to\rs_2$ and a homeomorphism $g\in\LL$, such that
\begin{itemize}
  \item $f^*(\phi_2^{\sign_2})=\phi_1^{\sign_1}$,
\item there is a commutative diagram
\begin{gather}\label{eq:l}
\begin{tikzpicture}[xscale=2.5,yscale=1]
\draw (0,2) node (u0) {$\sov$}
      (1,2) node (u1) {$(\rs_1)^{\phi_1}_{\Zer(\phi_1)}$}
      (0,0) node (d0) {$\sov$}
      (1,0) node (d1) {$(\rs_2)^{\phi_2}_{\Zer(\phi_2)}$.};
\draw [-stealth]
 (u0) edge node [left] {$g$} (d0)
 (u1) edge node [right] {$f_*$} (d1)
 (u0) edge node [above] {$\widetilde{\psi}_{\us_1}$} (u1)
 (d0) edge node [below] {$\widetilde{\psi}_{\us_2}$} (d1)
;\end{tikzpicture}
\end{gather}
\end{itemize}
We denote by $\FQuad{}{\sov}$ the moduli space of $\sov$-framed quadratic differentials.
\end{definition}

So the point is that now $\MCG(\sov)$ naturally acts on $\FQuad{}{\sov}$ by left composing with the framing. In this way, we construct a $\MCG(\sov)$-covering of $\Quad(\surfv)$,
i.e. we obtain the following.
\begin{lemma}
There is a commutative diagram between moduli spaces:
\begin{equation}\label{eq:quad0}
\begin{tikzpicture}[xscale=1.6,yscale=0.8,baseline=(bb.base)]
\path (0,1) node (bb) {}; 
\draw (0,2) node (s0) {$\FQuad{}{\sov}$}
 (0,0) node (s1) {$\SQuad$}
 (2.5,1) node (s2) {$\Quad(\surfv)$.};
\draw [->, font=\scriptsize]
 (s0) edge node [left] {$\SBr(\sov)$} (s1)
 (s0) edge node [above] {$\MCG(\sov)$} (s2)
 (s1) edge node [below] {$\MCG(\surfv)$} (s2);
\end{tikzpicture}
\end{equation}
\end{lemma}

\subsection{Application to the topology of moduli spaces}\label{sec:skel}\

Consider the top two parts of the stratification of $\SQuad$,
analogous to the stratification of $\Quad(\surfv)$ from \cite[\S~5]{BS}
(cf. \cite[\S~4.2]{KQ2} and \cite[\S~4.2]{BMQS}):
\begin{align*}
 F_0(\surfv) &=\{ [\rs,\phi^\sign,\psi] \in \SQuad \mid \text{$\phi$ is complete and has no saddle trajectories} \},\\
 F_2(\surfv) &=\{ [\rs,\phi^\sign,\psi] \in \SQuad \mid \text{$\phi$ is complete and has exactly one saddle trajectory} \}.
\end{align*}
Then $B_0(\surfv) :=F_0(\surfv)$ is open and dense, and $F_2(\surfv)$ has codimension 1.
Furthermore,  $B_2(\surfv) := F_0(\surfv)\cup F_2(\surfv)$ is also open and dense with complement of codimension 2.

Let $\cub^\circ(\RT)$ be the subspace in $\SQuad$ consisting of
those saddle-free $[\psi]$ whose WKB triangulation is $\RT$.
Then
\begin{gather}\label{eq:Squad*}
    B_0(\surfv)=\bigcup_{\RT\in\EG(\surfv)} \cub^\circ(\RT).
\end{gather}
By the argument of \cite[Prop~4.9]{BS},
$\cub^\circ(\RT)\isom\UHP^{\RT}$ for
\[
  \UHP=\{z\in\CC\mid \Imgy(z)>0\}\subset\CC
\]
is the (strict) upper half plane,
and the coordinates are given by the periods of saddle connections in horizontal strips.
Thus the $\cub^\circ(\RT)$ are precisely the connected components of $B_0(\surfv)$.

The boundary of $\cub^\circ(\RT)$ meets $F_2(\surfv)$ in $2\numarc$ connected components,
which we denote $\partial^\sharp_\gamma\cub(\RT)$ and $\partial^\flat_\gamma\cub(\RT)$,
where the coordinate goes to the negative or positive real axis, respectively.
Then we have the following.

By \cite[Prop~5.8]{BS}, $\EG(\surfv)$ is a skeleton for $\SQuad$ in the following sense.

\begin{lemma}\label{lem:iota_surf}
There is a canonical embedding
$\skel_\surf\colon\EG(\surfv)\to\SQuad$ whose image is dual to $B_2(\surfv)$.
More precisely,
the embedding is unique up to homotopy, satisfying
\begin{itemize}
\item for each triangulation $\RT_\sign\in\EG(\surfv)$, the point $\skel_\surf(\RT_\sign)$ is in $\cub(\RT_\sign)$,
\item for each flip $x\colon\RT_\sign\to\tilt{(\RT_\sign)}{\sharp}{\gamma}$, the path $\skel_\surf(x)$ is in
$\cub(\RT_\sign) \;\cup\; \partial^\sharp_\gamma\cub(\RT_\sign) \;\cup\; \cub(\tilt{\RT_\sign}{\sharp}{\gamma})$,
connecting $\skel_\surf(\RT_\sign)$ to $\skel_\surf(\tilt{\RT_\sign}{\sharp}{\gamma})$ and intersecting
$\partial^\sharp_\gamma\cub(\RT_\sign)$ at exactly one point,
\end{itemize}
Moreover, $\skel_{\surf}$ induces a surjective map
\[\pi_1\EG(\surfv)\cong\pi_1B_2(\surfv)\to\pi_1\SQuad.\]
\end{lemma}

As we have (various) coverings of both $\EG(\sov)$ and $\SQuad$,
the lemma has a covering version.
Namely, fix an initial triangulation $\Te\in\EG(\sov)$ and a connected component $\FQuad{\Te}{\sov}$
that contains differentials $(\rs,\phi^\sign,\psi)$ such that $\psi(\Te)$ is the WKB triangulation of $\rs^\phi_{\Zer(\phi)}$.
Then there is a (unique up to homotopy) canonical embedding
\begin{equation}\label{eq:Sembed}
\skel_{\sov}\colon\EGT(\sov)\to\FQuad{\Te}{\sov}
\end{equation}
whose image is dual to $B_2(\sov)$
and which induces a surjective map
\[\skel_*\colon \pi_1\EGT(\sov)\to\pi_1\FQuad{\Te}{\sov}.\]

As \cite[Thm.~4.16 and Cor.~4.17]{KQ2} or \cite[Thm.~1.1]{Q24},
we have the following application (with the same proof).


\begin{theorem}\label{thm:pi1}
The moduli space $\FQuad{}{\sov}$ of $\sov$-framed quadratic differentials
consists of $\Ho{1}(\surfv)$ many connected components, each of which is isomorphic to $\FQuad{\emph{}\Te}{\sov}$.

Moreover, $\FQuad{\emph{}\Te}{\sov}$ is simply connected, or equivalently,
there is a natural isomorphism $\pi_1\SQuad\cong\CBr(\surfv)$.
\end{theorem}

\appendix
\section{Proof of \texorpdfstring{\Cref{lem:quo}}{Lemma}}\label{app:A}

By \Cref{lem:coh}, we will repeatedly use the fact that $\COR$ and their conjugations by $N$ are in $\CON$,
for $2g+b\leq r\leq 2g+b+p-1$, in the proofs.

\begin{lemma}
    For any $2g+b\leq r\leq 2g+b+p-1$, we have
    \begin{equation}\label{eq:llin2}
        \TAU:= \iv{\tau_{r-1}}^{\varepsilon_r}\kg{\tau_r}^{\varepsilon_{r-1}}\in\CON.
    \end{equation}
\end{lemma}
\begin{proof}
This follows as below:
\[\begin{array}{rl}
\TAU\xlongequal{\eqref{eq:ss06}}&\iv{\tau_{r-1}}^{\iv{\sigma_2}\kg\iv{\tau_r}\kg\iv{\sigma_1}\kg\iv{\sigma_2}}{\tau_r}^{\sigma_2\kg\tau_{r-1}\kg\iv{\sigma_1}\kg\iv{\sigma_2}}
    \xlongequal[\Br(\sigma_2,\tau_r)]{\Br(\sigma_2,\tau_{r-1})}\iv{\sigma_2}^{{\tau_{r-1}}\kg\iv{\tau_r}\kg\iv{\sigma_1}\kg\iv{\sigma_2}}{\sigma_2}^{\iv{\tau_r}\kg\tau_{r-1}\kg\iv{\sigma_1}\kg\iv{\sigma_2}}\\
=&\left(
\iv{\sigma_2}^ { {\tau_{r-1}} \iv{\tau_r} \iv{\tau_{r-1}}  \tau_{r} } \sigma_2
\right)^{\iv{\tau_r}\kg\tau_{r-1}\kg\iv{\sigma_1}\kg\iv{\sigma_2}}
=\left(
   \iv{\COR^{\tau_{r}}}  \COR^{\tau_{r} \sigma_2}
\right)^{\iv{\tau_r}\kg\tau_{r-1}\kg\iv{\sigma_1}\kg\iv{\sigma_2}}\;\in\CON.
\end{array} \qedhere\]
\end{proof}

    By the following simple relations of commutators
    \[
    [a,bc]=[a,b][a,c]^{\iv{b}},\quad [b,a]=\iv{[a,b]} \quad \text{and}\quad [\iv{a},b]=[b,a]^a,
    \]
    together with \Cref{lem:coh}, we only need to show \eqref{eq:co} for $\sbt$ in a generating set. since $r\geq 2g+b>2g$, by \eqref{eq:deltas2}, we have $\delta_r=\varepsilon_r\kg\iv{\varepsilon_{r-1}}$.

    For the case $\sbt\in N $, we have
    \[
        [\sbt,\delta_r^2]
    =\sbt\kg\delta_r^2\kg\iv{\sbt}\kg\iv{\delta_r}^2
    =\left({\sbt}^{\delta_r}\kg\iv{\sbt}^{\iv{\delta_r}}\right)^{\iv{\delta_r}}
    =\left({\sbt}^{\varepsilon_r\kg\iv{\varepsilon_{r-1}}}\kg\iv{\sbt}^{\varepsilon_{r-1}\kg\iv{\varepsilon_{r}}}\right)^{\iv{\delta_r}}.
    \]
    So by \Cref{lem:coh}, we only need to show ${\sbt}^{\varepsilon_r\kg\iv{\varepsilon_{r-1}}}\kg\iv{\sbt}^{\varepsilon_{r-1}\kg\iv{\varepsilon_{r}}}\in\CON$. This is equivalent to show the equivalence class of ${\sbt}^{\varepsilon_r\kg\iv{\varepsilon_{r-1}}}\kg\iv{\sbt}^{\varepsilon_{r-1}\kg\iv{\varepsilon_{r}}}$ in the quotient group $ N /\CON$ is the identity. We divide the proof into the following subcases.

    \paragraph{\textbf{In the case $\sbt=\sigma_i$, $2\leq i\leq \aleph$}}, by $\Co(\sigma_i,\delta_r)$, we have $[\sigma_i,\delta^2_r]=1$.

    \paragraph{\textbf{In the case $\sbt=\sigma_1$,}} we have
    \[
    \begin{array}{l}
    {\sigma_1}^{\varepsilon_r\kg\iv{\varepsilon_{r-1}}}\kg\iv{\sigma_1}^{\varepsilon_{r-1}\kg\iv{\varepsilon_{r}}}
    \xlongequal{\eqref{eq:ss05}}{\tau_r}^{\iv{\sigma_1}\kg\iv{\varepsilon_{r-1}}}\kg\iv{\tau_{r-1}}^{\iv{\sigma_1}\kg\iv{\varepsilon_{r}}}
    \xlongequal[\eqref{eq:ss03}]{\eqref{eq:ss02}}{\tau_r}^{\iv{x}\kg\iv{\tau_{r-1}}\kg\sigma_2}\kg\iv{\tau_{r-1}}^{{x}\kg{\tau_{r}}\kg\sigma_2}
    =\left({\tau_r}^{\iv{x}\kg\iv{\tau_{r-1}}}\kg\iv{\tau_{r-1}}^{{x}\kg{\tau_{r}}}\right)^{\sigma_2}\\
    \xlongequal[\Br(x,\tau_{r-1})]{\Br(x,\tau_r)}\left({x}^{{\tau_r}\kg\iv{\tau_{r-1}}}\kg\iv{x}^{\iv{\tau_{r-1}}\kg{\tau_{r}}}\right)^{\sigma_2}.
    \end{array}
    \]

    \paragraph{\textbf{In the case $\sbt=\tau_s$, for $s\geq r$,}} we have
    \[\begin{array}{l}
    {\tau_s}^{\varepsilon_r\kg\iv{\varepsilon_{r-1}}}\kg\iv{\tau_s}^{\varepsilon_{r-1}\kg\iv{\varepsilon_{r}}}
    \xlongequal{\eqref{eq:ss06}}{\tau_s}^{\sigma_2\kg\tau_r\kg\iv{\sigma_1}\kg\iv{\sigma_2}\kg\iv{\varepsilon_{r-1}}}\kg\iv{\tau_s}^{\sigma_2\kg\tau_{r-1}\kg\iv{\sigma_1}\kg\iv{\sigma_2}\kg\iv{\varepsilon_{r}}}\\
    \xlongequal[\eqref{eq:ss03}]{\eqref{eq:ss02}}{\tau_s}^{\iv{x}\kg\iv{\tau_{r-1}}\kg\sigma_2\kg\tau_{r-1}\kg\sigma_2\kg{\tau_r}^{\iv{x}\kg\iv{\tau_{r-1}}\kg\sigma_2\kg\tau_{r-1}}\kg\iv{\tau_{r-1}}\kg\iv{\sigma_2}}\kg\iv{\tau_s}^{\iv{x}\kg\iv{\tau_{r}}\kg\sigma_2\kg\tau_{r}\kg\sigma_2\kg{\tau_{r-1}}^{{x}\kg{\tau_{r}}\kg\sigma_2\kg\tau_{r}}\kg\iv{\tau_{r}}\kg\iv{\sigma_2}}\\
    \xlongequal[\Br(\sigma_2,\tau_r)]{\Br(\sigma_2,\tau_{r-1})}{\tau_s}^{\iv{x}\kg\tau_{r-1}\kg{x}\tau_r\kg\iv{x}\iv{\tau_{r-1}}}\kg\iv{\tau_s}^{\iv{x}\kg\iv{\tau_r}\kg\iv{x}\kg\tau_{r-1}\kg{x}\kg\tau_r}
    \xlongequal[\Br(x,\tau_{r-1})]{\Br(x,\tau_r)}{\tau_s}^{\iv{x}\kg\tau_{r-1}\kg\iv{\tau_r}\kg{x}\tau_r\iv{\tau_{r-1}}}\kg\iv{\tau_s}^{\iv{x}\kg\iv{\tau_r}\kg\tau_{r-1}\kg{x}\kg\iv{\tau_{r-1}}\kg\tau_r}.
    \end{array}\]

    \paragraph{\textbf{In the case $\sbt=\tau_s$, for $s\leq r-1$ and $s\notin\dgen$,}} we have
    \[\begin{array}{l}
    {\tau_s}^{\varepsilon_r\kg\iv{\varepsilon_{r-1}}}\kg\iv{\tau_s}^{\varepsilon_{r-1}\kg\iv{\varepsilon_{r}}}
    \xlongequal{\eqref{eq:ss06}}{\tau_s}^{\iv{\sigma_2}\kg\iv{\tau_r}\kg\iv{\sigma_1}\kg\iv{\sigma_2}\kg\iv{\varepsilon_{r-1}}}\kg\iv{\tau_s}^{\iv{\sigma_2}\kg\iv{\tau_{r-1}}\kg\iv{\sigma_1}\kg\iv{\sigma_2}\kg\iv{\varepsilon_{r}}}\\
    \xlongequal[\eqref{eq:ss03}]{\eqref{eq:ss02}}{\tau_s}^{{x}\kg{\tau_{r-1}}\kg\sigma_2\kg\tau_{r-1}\kg\iv{\sigma_2}\kg\iv{\tau_r}^{\iv{x}\kg\iv{\tau_{r-1}}\kg\sigma_2\kg\tau_{r-1}}\kg\iv{\tau_{r-1}}\kg\iv{\sigma_2}}\kg\iv{\tau_s}^{{x}\kg{\tau_{r}}\kg\sigma_2\kg\tau_{r}\kg\iv{\sigma_2}\kg\iv{\tau_{r-1}}^{{x}\kg{\tau_{r}}\kg\sigma_2\kg\tau_{r}}\kg\iv{\tau_{r}}\kg\iv{\sigma_2}}\\
    \xlongequal[\Br(\sigma_2,\tau_r)]{\Br(\sigma_2,\tau_{r-1})}{\tau_s}^{{x}\kg\tau_{r-1}\kg{x}\iv{\tau_r}\kg\iv{x}\iv{\tau_{r-1}}}\kg\iv{\tau_s}^{{x}\kg\iv{\tau_r}\kg\iv{x}\kg\iv{\tau_{r-1}}\kg{x}\kg\tau_r}
    \xlongequal[\Br(x,\tau_{r-1})]{\Br(x,\tau_r)}{\tau_s}^{{x}\kg\tau_{r-1}\kg\iv{\tau_r}\kg\iv{x}\tau_r\iv{\tau_{r-1}}}\kg\iv{\tau_s}^{{x}\kg\iv{\tau_r}\kg\tau_{r-1}\kg\iv{x}\kg\iv{\tau_{r-1}}\kg\tau_r}.
    \end{array}\]

    \paragraph{\textbf{In the case $\sbt=\tau_s$, for $s\leq r-1$ and $s\in\dgen$,}} we have
    \[\begin{array}{l}
    {\tau_s}^{\varepsilon_r\kg\iv{\varepsilon_{r-1}}}\kg\iv{\tau_s}^{\varepsilon_{r-1}\kg\iv{\varepsilon_{r}}}
    \xlongequal{\eqref{eq:ss06}}{\tau_s}^{{\sigma_2}\kg\iv{\tau_r}\kg\iv{\sigma_1}\kg\iv{\sigma_2}\kg\iv{\varepsilon_{r-1}}}\kg\iv{\tau_s}^{{\sigma_2}\kg\iv{\tau_{r-1}}\kg\iv{\sigma_1}\kg\iv{\sigma_2}\kg\iv{\varepsilon_{r}}}\\
    \xlongequal[\eqref{eq:ss03}]{\eqref{eq:ss02}}{\tau_s}^{{x}\kg\iv{\tau_{r-1}}\kg\sigma_2\kg\tau_{r-1}\kg{\sigma_2}\kg\iv{\tau_r}^{\iv{x}\kg\iv{\tau_{r-1}}\kg\sigma_2\kg\tau_{r-1}}\kg\iv{\tau_{r-1}}\kg\iv{\sigma_2}}\kg\iv{\tau_s}^{{x}\kg\iv{\tau_{r}}\kg\sigma_2\kg\tau_{r}\kg{\sigma_2}\kg\iv{\tau_{r-1}}^{{x}\kg{\tau_{r}}\kg\sigma_2\kg\tau_{r}}\kg\iv{\tau_{r}}\kg\iv{\sigma_2}}\\
    \xlongequal[\Br(\sigma_2,\tau_r)]{\Br(\sigma_2,\tau_{r-1})}{\tau_s}^{{x}\kg\tau_{r-1}\kg{x}\iv{\tau_r}\kg\iv{x}\iv{\tau_{r-1}}}\kg\iv{\tau_s}^{{x}\kg\iv{\tau_r}\kg\iv{x}\kg\iv{\tau_{r-1}}\kg{x}\kg\tau_r}
    \xlongequal[\Br(x,\tau_{r-1})]{\Br(x,\tau_r)}{\tau_s}^{{x}\kg\tau_{r-1}\kg\iv{\tau_r}\kg\iv{x}\tau_r\iv{\tau_{r-1}}}\kg\iv{\tau_s}^{{x}\kg\iv{\tau_r}\kg\tau_{r-1}\kg\iv{x}\kg\iv{\tau_{r-1}}\kg\tau_r}.
    \end{array}\]

    In each subcase, the equivalence class of ${\sbt}^{\varepsilon_r\kg\iv{\varepsilon_{r-1}}}\kg\iv{\sbt}^{\varepsilon_{r-1}\kg\iv{\varepsilon_{r}}}$ in the quotient group  $ N /\Cop$  is the identity. So we have shown \eqref{eq:co} for $\sbt\in N$.

    Now, we only need to continue to show \eqref{eq:co} for $\sbt=\varepsilon_s$, for $1\leq s\leq 2g+b+p-1$. By $\delta_r=\varepsilon_r\kg\iv{\varepsilon_{r-1}}$, one can check
    \[
        [\varepsilon_s,\delta_r^2]
        =[\varepsilon_s,\varepsilon_r][\varepsilon_{r-1},\varepsilon_s]^{\iv{\delta_r}}[\varepsilon_s,\varepsilon_r]^{\iv{\delta_r}}[\varepsilon_{r-1},\varepsilon_s]^{\iv{\delta_r}^2}.
    \]
    Note that the last factor $[\varepsilon_{r-1},\varepsilon_s]^{\iv{\delta_r}^2}=[\delta_r^2,[\varepsilon_{r-1},\varepsilon_s]][\varepsilon_{r-1},\varepsilon_s]$. Since $[\varepsilon_{r-1},\varepsilon_s]\in N $ by \Cref{lem:tech},
    we know that $[\delta_r^2,[\varepsilon_{r-1},\varepsilon_s]]\in\CON$. Hence by \Cref{lem:coh} we only need to show
    $[\varepsilon_s,\varepsilon_r][\varepsilon_{r-1},\varepsilon_s]^{\iv{\delta_r}}[\varepsilon_s,\varepsilon_r]^{\iv{\delta_r}}[\varepsilon_{r-1},\varepsilon_s]\in\CON$. Since
    \[\begin{array}{rl}
    &[\varepsilon_s,\varepsilon_r][\varepsilon_{r-1},\varepsilon_s]^{\iv{\delta_r}}[\varepsilon_s,\varepsilon_r]^{\iv{\delta_r}}[\varepsilon_{r-1},\varepsilon_s]\\=&\left([\varepsilon_s,\varepsilon_r]^{\varepsilon_r}[\varepsilon_{r-1},\varepsilon_s]^{\varepsilon_{r-1}}[\varepsilon_s,\varepsilon_r]^{\varepsilon_{r-1}}[\varepsilon_{r-1},\varepsilon_s]^{\varepsilon_r}\right)^{\iv{\varepsilon_r}}\\
    =&\left(\left([\varepsilon_{r-1},\varepsilon_s][\varepsilon_s,\varepsilon_r]\right)^{\varepsilon_r}\left([\varepsilon_{r-1},\varepsilon_s][\varepsilon_s,\varepsilon_r]\right)^{\varepsilon_{r-1}}\right)^{[\varepsilon_{r-1},\varepsilon_s]^{\varepsilon_r}\kg\iv{\varepsilon_r}},
    \end{array}\]
    by \Cref{lem:coh}, it suffices to show
    \[\left([\varepsilon_{r-1},\varepsilon_s][\varepsilon_s,\varepsilon_r]\right)^{\varepsilon_{r}}\left([\varepsilon_{r-1},\varepsilon_s][\varepsilon_s,\varepsilon_r]\right)^{\varepsilon_{r-1}}\in \CON.\]
    We divide the proof into the following subcases.

    \paragraph{\textbf{In the case $s\geq r$,}} we have
    \[
    \begin{array}{l}
    [\varepsilon_{r-1},\varepsilon_s][\varepsilon_s,\varepsilon_r]
    =[\varepsilon_{r-1},\varepsilon_s]\iv{[\varepsilon_r,\varepsilon_s]}
    \xlongequal{\eqref{eq:in-BT}}(\iv{\tau_{r-1}}\kg{\sigma_1}\kg\iv{\tau_{r-1}}^{x\kg\tau_s\sigma_2}\tau_{s})\kg
    \iv{(\iv{\tau_{r}}\kg{\sigma_1}\kg\iv{\tau_{r}}^{x\kg\tau_s\sigma_2}\tau_{s})}\\
    =\iv{\tau_{r-1}}\kg(\iv{\tau_{r-1}}\kg{\tau_r})^{x\kg\tau_s\kg\sigma_2\kg\iv{\sigma_1}}\kg{\tau_{r}}=\iv{\tau_{r-1}}\kg(\iv{\tau_{r-1}}\kg{\tau_r})^{{\tau_s}^{\iv{x}}\kg\sigma_2}\kg{\tau_{r}}.
    \end{array}
    \]
    So we have
    \[
    \begin{array}{l}
        \left([\varepsilon_{r-1},\varepsilon_s][\varepsilon_s,\varepsilon_r]\right)^{\varepsilon_{r}}
        =\left(\iv{\tau_{r-1}}\kg(\iv{\tau_{r-1}}\kg{\tau_r})^{{\tau_s}^{\iv{x}}\kg\sigma_2}\kg{\tau_{r}}\right)^{\varepsilon_{r}}
        \xlongequal[\Co({\tau_s}^{\iv{x}},\varepsilon_r)]{\eqref{eq:ss05}}
        \iv{\tau_{r-1}}^{\varepsilon_r}\left(\iv{\tau_{r-1}}^{\varepsilon_r}\kg\sigma_1\right)^{{\tau_s}^{\iv{x}}\kg\sigma_2}\kg{\sigma_1}.
    \end{array}
    \]
    Similarly, we have
    \[\left([\varepsilon_{r-1},\varepsilon_s][\varepsilon_s,\varepsilon_r]\right)^{\varepsilon_{r-1}}=\iv{\sigma_1}\left(\iv{\sigma_1}\kg{\tau_r}^{\varepsilon_{r-1}}\right)^{{\tau_s}^{\iv{x}}\kg\sigma_2}\kg{\tau_r}^{\varepsilon_{r-1}}.\]
    Hence, we have
    \[\begin{array}{l}
    \left([\varepsilon_{r-1},\varepsilon_s][\varepsilon_s,\varepsilon_r]\right)^{\varepsilon_{r}}\left([\varepsilon_{r-1},\varepsilon_s][\varepsilon_s,\varepsilon_r]\right)^{\varepsilon_{r-1}}
    =\iv{\tau_{r-1}}^{\varepsilon_r}\left(\TAU\right)^{{\tau_s}^{\iv{x}}\kg\sigma_2}\kg{\tau_r}^{\varepsilon_{r-1}},
    \end{array}\]
    which, by \eqref{eq:llin2}, is in $\CON$.

    \paragraph{\textbf{In the case $s\leq r-1$ and $s\notin\dgen$,}} we have
    \[
    \begin{array}{l}
        [\varepsilon_{r-1},\varepsilon_s][\varepsilon_s,\varepsilon_r]
        =\iv{[\varepsilon_{s},\varepsilon_{r-1}]}[\varepsilon_s,\varepsilon_r]
        \xlongequal{\eqref{eq:in-BT}}\iv{(\iv{\tau_{s}}\kg{\sigma_1}\kg\iv{\tau_{s}}^{x\kg\tau_{r-1}\sigma_2}\tau_{r-1})}\kg{(\iv{\tau_{s}}\kg{\sigma_1}\kg\iv{\tau_{s}}^{x\kg\tau_r\sigma_2}\tau_{r})}\\
        =\iv{\tau_{r-1}}\kg{\tau_{s}}^{x\kg\tau_{r-1}\kg\sigma_2}\kg\iv{\tau_s}^{x\kg\tau_r\kg\sigma_2}\kg{\tau_{r}}.
    \end{array}
    \]
    So we have
    \[
    \begin{array}{rl}
    &\left([\varepsilon_{r-1},\varepsilon_s][\varepsilon_s,\varepsilon_r]\right)^{\varepsilon_{r}}=\left(\iv{\tau_{r-1}}\kg{\tau_{s}}^{x\kg\tau_{r-1}\kg\sigma_2}\kg\iv{\tau_s}^{x\kg\tau_r\kg\sigma_2}\kg{\tau_{r}}\right)^{\varepsilon_{r}}\\
    \xlongequal[\Co({\tau_s}^{x},\varepsilon_r)]{\eqref{eq:ss05},\eqref{eq:ss06}}&{\iv{\tau_{r-1}}}^{\varepsilon_r}{\left({\tau_{s}}^{x}\right)}^{{\tau_{r-1}}^{\varepsilon_r}\kg\sigma_2}\kg\left(\iv{\tau_{s}}^{x}\right)^{\sigma_1\kg\sigma_2}\kg\sigma_1.
    \end{array}
    \]
    Here, $\Co({\tau_s}^{x},\varepsilon_r)$ can be easily check.
    Similarly, we have
    \[
    \left([\varepsilon_{r-1},\varepsilon_s][\varepsilon_s,\varepsilon_r]\right)^{\varepsilon_{r-1}}=\iv{\sigma_1}{\left({\tau_{s}}^{x}\right)}^{{\sigma_1}\kg\sigma_2}\kg\left(\iv{\tau_{s}}^{x}\right)^{{\tau_r}^{\varepsilon_{r-1}}\kg\sigma_2}\kg{\tau_r}^{\varepsilon_{r-1}}.
    \]
    Hence we have
    \[\begin{array}{l}
    \left([\varepsilon_{r-1},\varepsilon_s][\varepsilon_s,\varepsilon_r]\right)^{\varepsilon_{r}}\left([\varepsilon_{r-1},\varepsilon_s][\varepsilon_s,\varepsilon_r]\right)^{\varepsilon_{r-1}}
    ={\iv{\tau_{r-1}}}^{\varepsilon_r}{\left({\tau_{s}}^{x}\right)}^{{\tau_{r-1}}^{\varepsilon_r}\kg\sigma_2}\kg\left(\iv{\tau_{s}}^{x}\right)^{{\tau_r}^{\varepsilon_{r-1}}\kg\sigma_2}\kg{\tau_r}^{\varepsilon_{r-1}},
    \end{array}\]
    which, by \eqref{eq:llin2}, is in $\CON$.

    \paragraph{\textbf{In the case $s\leq r$ and $s\in\dgen$,}} we have
    \[
    \begin{array}{l}
    [\varepsilon_{r-1},\varepsilon_s][\varepsilon_s,\varepsilon_r]
    =\iv{[\varepsilon_{s},\varepsilon_{r-1}]}[\varepsilon_s,\varepsilon_r]\\
    \xlongequal{\eqref{eq:in-BT}}\iv{(\iv{\sigma_2}\kg\iv{\tau_s}\kg{\tau_{r-1}}^{\iv{x}\kg{\tau_s}\kg{\sigma_2}}\kg{\sigma_1}\kg{\tau_{r-1}}\kg{\sigma_2})}\kg
    (\iv{\sigma_2}\kg\iv{\tau_s}\kg{\tau_r}^{\iv{x}\kg{\tau_s}\kg{\sigma_2}}\kg{\sigma_1}\kg{\tau_r}\kg{\sigma_2})\\
    =\iv{\sigma_2}\kg\iv{\tau_{r-1}}\kg\iv{\tau_{r-1}}^{\iv{x}\kg{\tau_s}\kg{\sigma_2}\kg{\sigma_1}}\kg{\tau_r}^{\iv{x}\kg{\tau_s}\kg{\sigma_2}\kg{\sigma_1}}\kg{\tau_r}\kg{\sigma_2}
    =\iv{\sigma_2}\kg\iv{\tau_{r-1}}\kg\iv{\tau_{r-1}}^{{\tau_s}^{x}\kg{\sigma_2}}\kg{\tau_r}^{{\tau_s}^{x}\kg{\sigma_2}}\kg{\tau_r}\kg{\sigma_2}.
    \end{array}
    \]
    So we have
    \[
    \begin{array}{rl}
    &\left([\varepsilon_{r-1},\varepsilon_s][\varepsilon_s,\varepsilon_r]\right)^{\varepsilon_{r}}
    =\left(\iv{\sigma_2}\kg\iv{\tau_{r-1}}\kg\iv{\tau_{r-1}}^{{\tau_s}^{x}\kg{\sigma_2}}\kg{\tau_r}^{{\tau_s}^{x}\kg{\sigma_2}}\kg{\tau_r}\kg{\sigma_2}\right)^{\varepsilon_{r}}\\
    \xlongequal[\Co({\tau_s}^{x},\varepsilon_r)]{\eqref{eq:ss05},\eqref{eq:ss06}}&\iv{\sigma_2}\kg\iv{\tau_{r-1}}^{\varepsilon_r}\kg\left(\iv{\tau_{r-1}}^{\varepsilon_r}\right)^{{\tau_s}^{x}\kg{\sigma_2}}\kg{\sigma_1}^{{\tau_s}^{x}\kg{\sigma_2}}\kg{\sigma_1}\kg{\sigma_2}.
    \end{array}
    \]
    Similarly, we have
    \[
    \left([\varepsilon_{r-1},\varepsilon_s][\varepsilon_s,\varepsilon_r]\right)^{\varepsilon_{r-1}}=\iv{\sigma_2}\kg\iv{\sigma_1}\kg\iv{\sigma_1}^{{\tau_s}^{x}\kg{\sigma_2}}\kg\left({\tau_r}^{\varepsilon_{r-1}}\right)^{{\tau_s}^{x}\kg{\sigma_2}}\kg{\tau_r}^{\varepsilon_{r-1}}\kg{\sigma_2}.
    \]
    Hence we have
    \[\begin{array}{l}
    \left([\varepsilon_{r-1},\varepsilon_s][\varepsilon_s,\varepsilon_r]\right)^{\varepsilon_{r}}\left([\varepsilon_{r-1},\varepsilon_s][\varepsilon_s,\varepsilon_r]\right)^{\varepsilon_{r-1}}\\
    =\iv{\sigma_2}\kg\iv{\tau_{r-1}}^{\varepsilon_r}\kg\left(\iv{\tau_{r-1}}^{\varepsilon_r}\right)^{{\tau_s}^{x}\kg{\sigma_2}}\left({\tau_r}^{\varepsilon_{r-1}}\right)^{{\tau_s}^{x}\kg{\sigma_2}}\kg{\tau_r}^{\varepsilon_{r-1}}\kg{\sigma_2},
    \end{array}\]
    which, by \eqref{eq:llin2}, is in $\CON$.

%
%
%
%
%
%
%


\end{document}